\documentclass{amsart}    

\usepackage[
backend=biber,
style=numeric,
sorting=nyt
]{biblatex}

\usepackage{microtype}

\title{Spectral equivalences through nonstandard samplings}

\author{Fabrice Nonez}
\address{Department of Mathematics and Statistics, Concordia University, Montreal, Canada}
\email{fabrice.nonez@mail.concordia.ca}


\newtheorem{thm}{Theorem}[section]    
\newtheorem{theorem}{Theorem}[section] 
           
\newtheorem{lemma}[thm]{Lemma}        
\newtheorem{prop}[thm]{Proposition}

\theoremstyle{definition}
    
\newtheorem{definition}[thm]{Definition}

\theoremstyle{remark}
\newtheorem{remark}[thm]{Remark}

\newcounter{pointnumber}

\begin{document}
\begin{abstract}    
The goal of this paper is to introduce a process that generates, given Hilbert space $H$ and symmetric operator $A$, an embedding of $H$ into an $L_2$-space through which $A$ is extended by a multiplication operator. This process will depend on two parameters, the nonstandard sampling and the standard-biased scale. 

We will use that process to prove diverse versions of the general spectral theorem, showing its appeal. Furthermore, through landmark examples, we will observe that by carefully tweaking the two parameters, we can make the resulting spaces, embeddings and operators quite explicit and natural.
\end{abstract}
\maketitle
\section{Overview}
We start by stating this version of the spectral theorem.

\begin{theorem}[Spectral theorem]\label{theorem_spectral_symmetric}
	Let $H$ be a separable infinite-dimensional Hilbert space, and let $A$ be a densely-defined symmetric operator on $H$. 
	
	Then, there exists a compact metric space $(\Omega,d)$ endowed with a probability measure $\mu$ on its Borel $\sigma$-algebra, a borelian measurable function $m:\Omega\rightarrow\mathbb{R}$ generating self-adjoint multiplication operator $T$ on $L_2(\Omega,\mu)$, and an isometry $U:H\rightarrow L_2(\Omega,\mu)$ such that $T\circ U$ extends $U\circ A$. 
\end{theorem}

The formulation differs from the usual version in that $A$ is only assumed to be symmetric instead of self-adjoint, at the cost of $U$ not always being unitary. This also establishes the well-known fact that any symmetric operator has a self-adjoint extension in a larger Hilbert space. 

The main scope of this text is to consider a process that generates the objects described in Theorem~\ref{theorem_spectral_symmetric}, given a separable Hilbert space $H$ and a densely-defined symmetric operator $A$. This process will depend on two important parameters: a sampling for $A$ and a standard-biased scale compatible with the sampling. By proving that such parameters always exist, we will also have a direct proof of Theorem~\ref{theorem_spectral_symmetric}.

Contrasting with classical proofs, which often use the Cayley transform and spectrum analysis, this process relies more directly on the finite-dimensional version of the spectral theorem. It works both for bounded operators and unbounded ones, and for real Hilbert spaces as well as complex ones. Furthermore, it is surprisingly constructive, in that it may provide a path to construct the objects, as we will see in the examples.

The method is inspired by previous work in nonstandard analysis, concerning the resolution version of the spectral theorem. Most notably, the recent work done in \cite{GOLDBRING2021590} shows that nonstandard methods are suitable to prove the spectral theorem for unbounded self-adjoint operators. Even more recently, another proof has been considered in \cite{matsunaga2024shortnonstandardproofspectral}. Proofs for bounded operators can also be found in \cite{bernstein_spectral} and \cite{mooreboundedops}.  In all of these works, as well as in this one, the initial idea is to transfer the finite-dimensional version of the spectral theorem and apply it on a suitable choice of hyperfinite space and internal symmetric map. 

A version of this idea that directly uses ultraproducts can be found in~\cite{hirvonen2022ultraproductsspectraltheoremrigged}. We think that some of the ideas we expose could help with the questions found in that body of work, notably concerning generalizations to unbounded self-adjoint operators.

In Section \ref{section_sampling}, we define the parameters, which are the sampling for $A$ and the compatible scale. Then, in Sections \ref{section_loebspace} and \ref{section_hullspace}, we describe the process of this paper. Specifically, in Section \ref{section_loebspace}, we use the sampling and the scale to create a suitable Loeb space, and prove a version of Theorem~\ref{theorem_spectral_symmetric}, albeit with the $L_2$ space not being separable in many cases. In Section \ref{section_hullspace}, using hull techniques, we create an interesting compact metric space from that Loeb space, completing the process and the proof of Theorem~\ref{theorem_spectral_symmetric}.

Then, in Section \ref{section_sa}, we consider the case where $A$ is self-adjoint. Using a spectral resolution, whose existence is shown in Appendix \ref{appendix_pullback} using Theorem~\ref{theorem_spectral_symmetric}, we create a suitable sampling and scale for which the induced isometry is surjective, proving the usual version of the spectral theorem.  Finally, in Sections \ref{section_shift} and \ref{section_differential}, we look at some relevant classic examples for which tweaking the parameters of the process can provide an explicit, natural and intuitive realization of Theorem~\ref{theorem_spectral_symmetric}.
\subsection{Notation and conventions}
In the paper, some familiarity with nonstandard analysis will be assumed. For example, internal objects, hyperfinite sets, and Loeb measure theory will be used. A comprehensive theory about nonstandard analysis can be found in \cite{arkeryd2012nonstandard}. Of note, the notation will be consistent with that book, whereas for a standard object $a$, we will note its extension by ${}^*a$. 

Sometimes, ${}^*$ may be omitted when there is almost certainly no ambiguity . For example, if $r\in{}^*\mathbb{R}$ is nonstandard, $\cos(r)$ necessarily means $({}^*\cos)(r)$. About all other times, ${}^*$ will be use everywhere it is meant to be. If the expression "Let $\epsilon\in\mathbb{R}_{>0}$" is written, $\epsilon$ is meant to be standard.

In the construction, we will work with three probability spaces: an internal one, the Loeb-constructed one and the one from the nonstandard hull. We will try to denote objects related to the internal space with tilde symbol ( $\tilde{}$ ), ones related to Loeb space with subscript $L$ (${}_L$) and ones related to the hull space with the hat symbol ( $\hat{}$ ).

Throughout the paper, we assume $H$ is an infinite-dimensional separable Hilbert space, with underlying field $\mathbb{R}$ or $\mathbb{C}$, which we will denote by $\mathbb{K}$. Sometimes, we may use the conjugate $\bar{z}$ or the real part $\Re(z)$. Both are just the identity function if $\mathbb{K}=\mathbb{R}$. For closed subspace $V$, the operator $\operatorname{proj}_V$ on $H$ denotes the self-afjoint projection with image $V$.

Also, $A:\operatorname{dom}(A)\rightarrow H$ is a densely-defined symmetric linear operator on $H$. This operator is not assumed to be either bounded or self-adjoint, except when explicitly stated, notably in Section \ref{section_sa} and Appendix \ref{appendix_pullback}. Details about Hilbert spaces and unbounded operators can be found in \cite{conway2019course}.

We now fix ${}^*$, an (at least) $\aleph_1$-saturated enlargement\footnote{The existence of a suitable hyperfinite set is only used once to simplify the construction of a sampling. With more finesse, such a construction is possible given only countable saturation, so any ultrapower constructed using a countably incomplete ultrafilter would work well.} of objects in a superstructure containing at least $\mathbb{K}$ and $H$.

Open ends of intervals will use square brackets, for example $[1,2[$ instead of $[1,2)$. We will assume that $\mathbb{N}$ starts at $1$. For $k\in\mathbb{N}$ (or ${}^*\mathbb{N}$), we will use $[k]=\{1,2,3,\dots, k\}$. Also, for $k,m\in\mathbb{Z}$, we will denote $[k .. m[ = \{l\in\mathbb{Z}\; |\; k\leq l < m\}=\mathbb{Z}\cap [k,m[$. In such a set, we will often use $\oplus$ and $\ominus$ for the cyclic addition and subtraction (specifically, $n_1\oplus n_2\equiv n_1+ n_2$ mod $m-k$ and $n_1\oplus n_2 \in [k..m[$).

For any given function (standard, internal or external) $f$, we will denote its graph by $G(f)$. For given operators $T_1$ and $T_2$, we say that $T_1\leq T_2$ if $G(T_1)\subseteq G(T_2)$.
\section{Samplings and scales}\label{section_sampling}

Here, we define the main parameters we will use throughout this paper, the sampling for $A$ and the compatible scale. Those objects allow a representation of $A$ in a hyperfinite world, and will later allow the construction of a relevant standard measure.

\begin{definition}\label{definition_sampling}
	We call the internal triplet $(\tilde{H},\tilde{A},\tilde{\Omega})$ a sampling for $A$ if:
	
	\begin{enumerate}
		\item \label{sampling_def_H} $\tilde{H}<{}^*H$, and ${}^*\operatorname{dim}(\tilde{H})\in{}^*\mathbb{N}$. 
		\item \label{sampling_def_A} $\tilde{A}:\tilde{H}\rightarrow\tilde{H}$ is a ${}^*$linear symmetric operator on $\tilde{H}$.
		\item \label{sampling_def_Omega} $\tilde{\Omega}\subset\tilde{H}$ is an orthonormal ${}^*$basis of $\tilde{H}$ formed by eigenvectors of $\tilde{A}$.
		
		\item \label{sampling_def_approx} $G(A)\subset$ st$(G(\tilde{A}))$. In other words, for any $x\in \operatorname{dom}(A)$, there exists $\tilde{x}\in \tilde{H}$ such that $x=$st$(\tilde{x})$ and $Ax=$st$(\tilde{A}\tilde{x})$.
		\setcounter{pointnumber}{\value{enumi}}
	\end{enumerate}
\end{definition}

\begin{prop}
	There exists a sampling for $A$.
\end{prop}
\begin{proof}
	Let $\tilde{H}={}^*$span$(S)$, where $S$ is any ${}^*$finite subset such that dom$(A)\subset S\subset {}^*\operatorname{dom}(A)$. Then, condition (\ref{sampling_def_H}) holds, and dom$(A)\leq \tilde{H}\leq{}^*$dom$(A) \leq {}^*H$. 
	
	Now, let $\tilde{A}:= (\operatorname{proj}_{\tilde{H}}\circ {}^*A)|_{\tilde{H}}:\tilde{H}\rightarrow \tilde{H}$. Since $\tilde{H}$ is internal, $\operatorname{proj}_{\tilde{H}}$ is internal, and so $\tilde{A}$ is an internal linear operator. Furthermore, for all $x,y \in \tilde{H}$:
	\begin{align*}
		(\tilde{A}x,y)&=(\operatorname{proj}_{\tilde{H}}{}^*Ax,y) =({}^*Ax,y) =(x,{}^*Ay) =(x,\operatorname{proj}_{\tilde{H}}{}^*Ay) = (x,\tilde{A}y).
	\end{align*}
	Hence, $\tilde{A}$ is symmetric and condition (\ref{sampling_def_A}) holds. By transfer principle of finite-dimensional spectral theorem applied to $\tilde{A}$, there exists $\tilde{\Omega}$, an orthornormal ${}^*$basis of $\tilde{H}$ consisting of ${}^*$eigenvectors of $\tilde{A}$, and so we fix an arbitrary one. With that, condition (\ref{sampling_def_Omega}) holds.
	
	Furthermore, for any $x\in H$, and for all $\epsilon\in\mathbb{R}_{>0}$, there exists $y\in \tilde{H}$ such that $\|x-y\|<\epsilon$.  Indeed, $\tilde{H}\supset\operatorname{dom}(A)$, which is dense in $H$. Since projection minimizes distance, we have that $\|x-\operatorname{proj}_{\tilde{H}}x\|$ is infinitesimal and $x=\operatorname{st}(\operatorname{proj}_{\tilde{H}} x)$. Therefore, for any $x\in\operatorname{dom}(A)\subset \tilde{H}$, we take $\tilde{x}:=x$ so that $Ax\in H$ and $Ax=\operatorname{st}(\tilde{A} x)=\operatorname{st}(\tilde{A} \tilde{x})$.

	We get that condition (\ref{sampling_def_approx}) holds, and we conclude that $(\tilde{H},\tilde{A},\tilde{\Omega})$ is a sampling for $A$.
\end{proof}
\begin{remark}
	The implementations of $\tilde{H}$ and $\tilde{A}$ that were given here are the ones that were mainly considered in previous works \cite{bernstein_spectral}, \cite{mooreboundedops}, \cite{GOLDBRING2021590}, \cite{matsunaga2024shortnonstandardproofspectral}. Interestingly, Moore, in \cite{mooreboundedops}, defined the general concept of an hyperfinite extension of $A$ in the nonstandard hull, which has important similarities with samplings.  
\end{remark}
\begin{definition}\label{definition_scale}
	Let $N\in{}^*\mathbb{N}\setminus\mathbb{N}$, $(\tilde{e}_j)_{j\in[N]}$ be a hyperfinite sequence in ${}^*H$, and $(\tilde{c}_j)_{j\in[N]}$ be a hyperfinite sequence in ${}^*\mathbb{R}_{\geq0}$. We say that the pair formed by $((\tilde{e}_j)_{j\in[N]},(\tilde{c}_j)_{j\in[N]})$ is a standard-biased scale if:
	\begin{enumerate}
		\setcounter{enumi}{\value{pointnumber}}
		\item\label{scale_def_NS} For each $j\in\mathbb{N}$, both $\tilde{e}_j$ and $\tilde{c}_j$ are nearstandard, as respective elements of ${}^*H$ and ${}^*\mathbb{R}$. Furthermore, $\operatorname{st}(\tilde{c}_j \tilde{e}_j)\neq0$.
		\item\label{scale_def_dense} $(\operatorname{st}(\tilde{e}_j))_{j\in\mathbb{N}}$ spans a dense subset in $H$.
		\item \label{scale_def_proba} $\sum_{j\in[N]}\tilde{c}_j\|\tilde{e}_j\|^2=1$.
		\item \label{scale_def_bias} For any infinite $K$, $\operatorname{st}(\sum_{K<j\leq N} \tilde{c}_j\|\tilde{e}_j\|^2)=0$. 
		
		Equivalently, $\sum_{j\in\mathbb{N}}\operatorname{st}(\tilde{c}_j)\|\operatorname{st}(\tilde{e}_j)\|^2=1$.
		
		\setcounter{pointnumber}{\value{enumi}}
	\end{enumerate}

	Furthermore, we say that this scale is compatible with sampling $(\tilde{H},\tilde{A},\tilde{\Omega})$ if :
	\begin{enumerate}
		\setcounter{enumi}{\value{pointnumber}}
		\item \label{scale_def_compat_H} For each $j\in [N]$, $\tilde{e}_j\in\tilde{H}$.
		\item \label{scale_def_compat_A} For each $j\in \mathbb{N}$, $\tilde{A}\tilde{e}_j$ is nearstandard.
		\item \label{scale_def_compat_Omega} For each $f\in\tilde{\Omega}$, there exists $j\in[N]$ such that $\tilde{c}_j>0$ and $(\tilde{e}_j,f)\neq0$.
	\end{enumerate}

\end{definition}
\begin{remark}
	If $(\tilde{H},\tilde{A},\tilde{\Omega})$ is a sampling for $A$, it is straightforward to show that $\operatorname{st}(G(\tilde{A}))\subset G(A^*)$, where $A^*$ is the adjoint of $A$. Therefore, the condition that $\tilde{A}\tilde{e}_j$ is nearstandard is equivalent to $\operatorname{st}(\tilde{e}_j)\in\operatorname{dom}(A^*)$ and $\tilde{A}\tilde{e}_j\simeq A^* (\operatorname{st}(\tilde{e}_j))$.
\end{remark}
\begin{prop}
	For any given sampling for $A$, there exists a standard-biased scale compatible with it.
\end{prop}
\begin{proof}
	Let  $(\tilde{H},\tilde{A},\tilde{\Omega})$  be any sampling for $A$.
	
	First, let $(e_j)_{j\in\mathbb{N}}$ be a Hilbert basis of $H$ such that for all $j\in\mathbb{N}$, $e_j\in\operatorname{dom}(A)$. It exists, since $H$ is separable and $\operatorname{dom}(A)$ is dense. Using Property~(\ref{sampling_def_approx}) of Definition~\ref{definition_sampling}, for each standard $j$, choose $\tilde{e}_j$ such that $\operatorname{st}(\tilde{e}_j)=e_j$, $Ae_j=\operatorname{st}(\tilde{A}\tilde{e}_j)$.
	
	By countable saturation and overspill, this can be extended to a non-zero hyperfinite sequence $(\tilde{e}_j)_{j\in[K]}$ in $\tilde{H}$. As with the finite-dimensional case, we can further extend it to obtain  an internal ${}^*$finite sequence $(\tilde{e}_j)_{j\in[N]}$ that generates $\tilde{H}$.
	
	Let $\tilde{c}_j=\frac{1}{2^j(1-2^{-N})\|\tilde{e}_j\|^2}$ for $j\in[N]$.  We note $\tilde{c}_j\simeq\frac{1}{2^j\|e_j\|^2}=\frac{1}{2^j}$ for any standard $j$. Then, we know, by definitions, that conditions (\ref{scale_def_NS}), (\ref{scale_def_dense}), (\ref{scale_def_compat_H}) and (\ref{scale_def_compat_A}) hold, while conditions (\ref{scale_def_proba}), (\ref{scale_def_bias}) and (\ref{scale_def_compat_Omega}) are easily verified.
	
	We conclude that $((\tilde{e}_j)_{j\in[N]},(\tilde{c}_j)_{j\in[N]})$ forms a standard-biased scale compatible with $(\tilde{H},\tilde{A},\tilde{\Omega})$.
\end{proof}

\begin{remark}
	There will be multiple choices of samplings and suitable scales for a given $A$. The goal here is to define the next objects and properties in relation to any one of the choices. Ultimately, a good choice of sampling and scale will facilitate the description of the resulting space, as we will see in the examples of Sections \ref{section_shift} and \ref{section_differential}. The choice also affects the surjectivity of the constructed isometry of Theorem~\ref{theorem_spectral_symmetric}, as we will see in Section \ref{section_sa}.
	
	For Sections \ref{section_loebspace} and \ref{section_hullspace}, the sampling for $A$ given by $(\tilde{H},\tilde{A},\tilde{\Omega})$ and the compatible scale $((\tilde{e}_j)_{j=1}^N,(\tilde{c}_j)_{j=1}^N)$ will be fixed, but arbitrary. All objects defined in those sections will be considered in relation to this choice.
	
	Furthermore, for $j\in\mathbb{N}$ we will note $e_j=\operatorname{st}(\tilde{e}_j)\in\operatorname{dom}(A)$, and $c_j=\operatorname{st}(\tilde{c}_j)\in\mathbb{R}_{>0}$. We also note the internal eigenvalue function $\tilde{\lambda}:\tilde{\Omega}\rightarrow {}^*\mathbb{R}$, with the notation $\tilde{\lambda}_f$ for $\tilde{\lambda}(f)$.
\end{remark}
\section{The induced Loeb probability space}\label{section_loebspace}

\begin{definition}
	We define $\tilde{\mathcal{A}}:=\{V\subset \tilde{\Omega}\;|\; V $ is internal $\}=({}^*\mathcal{P})(\tilde{\Omega})$. As $\tilde{\mathcal{A}}$ is an internal $\sigma$-algebra on $\tilde{\Omega}$, we define $\tilde{\mu}:\tilde{\mathcal{A}}\rightarrow{}^*\mathbb{R}_{\geq0}$ with 
	\begin{align*}
		\tilde{\mu}(V):=\sum_{j=1}^{N}\tilde{c}_j\|\operatorname{proj}_{{}^*\operatorname{span}(V)}\tilde{e}_j\|^2=\sum_{j=1}^{N}\tilde{c}_j\sum_{f\in V}|(\tilde{e}_j,f)|^2.
	\end{align*}
\end{definition}
\begin{remark}
	For $f\in \tilde{\Omega}$, we will often write $\tilde{\mu}(f)$ instead of $\tilde{\mu}(\{f\})$.
\end{remark}
The following is a direct consequence of Definitions \ref{definition_sampling} and \ref{definition_scale}, most notably Properties (\ref{sampling_def_Omega}), (\ref{scale_def_proba}), (\ref{scale_def_compat_H}) and (\ref{scale_def_compat_Omega}):
\begin{prop}
	We have that $\tilde{\mu}$ is an internal probability measure on $\tilde{\Omega}$, and so for any $V \in \tilde{\mathcal{A}}$, $\tilde{\mu}(V)=\sum_{f\in V}\tilde{\mu}(f)$. Furthermore, for any $f\in\tilde{\Omega}$, $\tilde{\mu}(f)>0$.
\end{prop}
\begin{remark}
	Here, the purpose of $\tilde{\mu}$ is to weight each subset in terms of how much energy that subset has in the standard world. It has to be done internally though, so the scale is there to create such an internal bias towards the standard elements.
\end{remark}

\begin{definition}
	We define the triplet $(\Omega_L,\mathcal{A}_L,\mu_L)$ as the Loeb measure space induced from $(\tilde{\Omega},\tilde{\mathcal{A}},\tilde{\mu})$. Specifically:
	\begin{itemize}
		\item $\Omega_L=\tilde{\Omega}$.
		\item $\mathcal{A}_L$ is an external $\sigma$-algebra on $\Omega_L$, with $\tilde{\mathcal{A}}\subset \mathcal{A}_L$.
		\item $\mu_L:\mathcal{A}_L\rightarrow\mathbb{R}_{\geq 0}$ is a $\sigma$-additive complete probability measure.
		\item for any $V\in \tilde{\mathcal{A}}$, $\mu_L(V)=\operatorname{st}(\tilde{\mu}(V))$.
	\end{itemize}
	Let $H_L:=L_2(\Omega_L,\mathcal{A}_L,\mu_L)$ be the standard Hilbert space on $\mathbb{K}$.
\end{definition}
\begin{definition}
	Let $\tilde{U}:{}^*H\rightarrow ({}^*L_2)(\tilde{\Omega},\tilde{\mu})$ be defined by  $(\tilde{U}(x))(f)=\frac{(x,f)}{\sqrt{\tilde{\mu}(f)}}$.
\end{definition}
\begin{remark}
	$({}^*L_2)(\tilde{\Omega},\tilde{\mu})$ is meant here as the internal space of internal functions $u:\tilde{\Omega}\rightarrow{}^*\mathbb{K}$ with internal inner product 
	\begin{align*}
		(u,v)={}^*\int_{\tilde{\Omega}}u\bar{v} d\tilde{\mu}=\sum_{f\in \tilde{\Omega}}u(f)\bar{v}(f)\tilde{\mu}(f).
	\end{align*}
\end{remark}
\begin{prop}
	$\tilde{U}$ is an internal, linear operator. Furthermore, $\ker(\tilde{U})=(\tilde{H})^\perp$, while $\tilde{U}|_{\tilde{H}}$ is a ${}^*$isometry. Finally, if $x$ is nearstandard, then $|\tilde{U}(x)|^2$ is $\mu_L$-almost always limited, S-integrable and $\int_{\Omega_L} |\operatorname{st}\circ\tilde{U}(x)|^2 d\mu_L=\|\operatorname{st}(x)\|^2$.
\end{prop}
\begin{proof}
	$\tilde{U}$ is internal, as its domain and co-domain are internal, and as it is defined using only internal objects. It is clearly linear on ${}^*H$. Furthermore, for any $x$ in $\tilde{H}$, 
	\begin{align*}
		\|\tilde{U}(x)\|^2&=\sum_{f\in \tilde{\Omega}}|(\tilde{U}(x))(f)|^2\tilde{\mu}(f)\\
		&=\sum_{f\in \tilde{\Omega}}\left |\frac{(x,f)}{\sqrt{\tilde{\mu}(f)}}\right |^2\tilde{\mu}(f)\\
		&=\sum_{f\in \tilde{\Omega}}|(x,f)|^2=\|\operatorname{proj}_{\tilde{H}} x\|^2.
	\end{align*}
	This implies that $\ker(\tilde{U})=\ker(\operatorname{proj}_{\tilde{H}})=(\tilde{H})^\perp$. Also, $\tilde{U}|_{\tilde{H}}$ preserves norm, and so as a ${}^*$linear funcion it is a ${}^*$isometry.
	
	Let $x\in{}^*H$ be nearstandard. We first show that for almost every $f$ (w.r.t $\mu_L$), $(\tilde{U}(x))(f)$ is a limited number.
	
	For $r\in{}^*\mathbb{R}_{\geq0}$, we consider $B_r=\{f\in \tilde{\Omega} \;|\; r<|(\tilde{U}(x))(f)|^2\}$ and $B_\infty=\{f\in \tilde{\Omega} \;|\; (\tilde{U}(x))(f)$ is infinite$\}$. $B_r$ is internal, and we note that $B_\infty=\bigcap_{n\in\mathbb{N}}B_n$. So $B_\infty$ is Loeb-measurable and $\mu_L(B_\infty)=\lim_{n\rightarrow\infty}\operatorname{st}(\tilde{\mu}(B_n))$. Furthermore,
	\begin{align*}
		&\tilde{\mu}(B_n)=\sum_{f\in B_n}\tilde{\mu}(f)\leq\sum_{f\in B_n}\frac{|(x,f)|^2}{n}\leq\frac{\|x\|^2}{n}\\
		\implies& \operatorname{st}(\tilde{\mu}(B_n))\leq\frac{\|\operatorname{st}(x)\|^2}{n} \implies \mu_L(B_\infty)=0.
	\end{align*}
	Thus, $\operatorname{st}((\tilde{U}(x))(f))$  is $\mathbb{K}$-valued for almost every $f$. To conclude the proof, we will show that for any infinite $K\in {}^*\mathbb{R}_{>0}$, ${}^*\int_{B_K}|\tilde{U}(x)|^2d\tilde{\mu}$ is an infinitesimal. We have:
	\begin{align*}
		{}^*\int_{B_K}|\tilde{U}(x)|^2d\tilde{\mu}&=\sum_{f\in B_K}|(\tilde{U}(x))(f)|^2\tilde{\mu}(f)= \sum_{f\in B_K}|(x,f)|^2=\|\operatorname{proj}_{\operatorname{span}(B_K)}x\|^2.
	\end{align*}
	We want to show that $\|\operatorname{proj}_{\operatorname{span}(B_K)}x\|$ is infinitesimal. Let $\epsilon\in\mathbb{R}_{>0}$. Since $x$ is nearstandard and $(e_j)_{j\in\mathbb{N}}$ spans a dense subset of $H$, there exists $n\in\mathbb{N}$  and standard $(a_1,\dots,a_n)$ in $\mathbb{K}$ s.t. $\|x-\sum_{j=1}^n a_j\tilde{e}_j\|\leq\frac{\epsilon}{2}$, and so:
	\begin{align*}
		\|\operatorname{proj}_{\operatorname{span}(B_K)}x\|\leq\|\operatorname{proj}_{\operatorname{span}(B_K)}(x-\sum_{j=1}^n a_j\tilde{e}_j)\|+\|\operatorname{proj}_{\operatorname{span}(B_K)}\sum_{j=1}^n a_j\tilde{e}_j\|.
	\end{align*}
	
	We have, by Property~(\ref{scale_def_NS}) of Definition~\ref{definition_scale}, that $c_j>0$ for $j\in[n]$. Therefore:
	
	\begin{align*}
		\|\operatorname{proj}_{\operatorname{span}(B_K)}x\|&\leq\|x-\sum_{j=1}^n a_j\tilde{e}_j\|+\sum_{j=1}^n|a_j|\|\operatorname{proj}_{\operatorname{span}(B_K)}\tilde{e}_j\|\\
		&\leq \frac{\epsilon}{2}+\sum_{j=1}^n|a_j|\sqrt{\frac{1}{\tilde{c}_j} \tilde{c}_j\|\operatorname{proj}_{\operatorname{span}(B_K)}\tilde{e}_j\|^2}\\
		&\leq\frac{\epsilon}{2}+\sum_{j=1}^n\frac{|a_j|}{\sqrt{\tilde{c}_j}}\sqrt{\tilde{\mu}(B_K)}
	\end{align*}
	and thus
	\begin{align*}
		\operatorname{st}(\|\operatorname{proj}_{\operatorname{span}(B_K)}x\|)&\leq\frac{\epsilon}{2}+\sum_{j=1}^n \frac{|a_j|}{\sqrt{c_j}}\sqrt{\mu_L(B_K)}\\
		&\leq\frac{\epsilon}{2}+\sum_{j=1}^n \frac{|a_j|}{\sqrt{c_j}}\sqrt{\mu_L(B_{\infty})}=\frac{\epsilon}{2}<\epsilon.
	\end{align*}
	Since it holds for any standard $\epsilon>0$, we get $\|\operatorname{proj}_{\operatorname{span}(B_K)}x\|$ is infinitesimal, and so is ${}^*\int_{B_K}|\tilde{U}(x)|^2d\tilde{\mu}$. Since it holds for any $K$ infinite, we get $|\tilde{U}(x)|^2$ is S-integrable, and so $\int_{\Omega_L} |\operatorname{st}\circ\tilde{U}(x)|^2 d\mu_L=\operatorname{st}({}^*\int_{\tilde{\Omega}}|\tilde{U}(x)|^2d\tilde{\mu})=\operatorname{st}(\|\operatorname{proj}_{\tilde{H}}x\|^2)$. All we have left to show is $\operatorname{proj}_{\tilde{H}}x\simeq\operatorname{st}(x)$.
	
	By hypothesis, $x\simeq \operatorname{st}(x)$. Since $\operatorname{proj}_{\tilde{H}}$ is non-expanding, $\operatorname{proj}_{\tilde{H}}x\simeq \operatorname{proj}_{\tilde{H}}\operatorname{st}(x)$. By Property~(\ref{sampling_def_approx}) of Definition~\ref{definition_sampling}, we know there exists $\tilde{x}\in\tilde{H}$ such that $\tilde{x}\simeq \operatorname{st}(x)$. Since $\operatorname{proj}_{\tilde{H}}$ minimizes distance, we get $\operatorname{proj}_{\tilde{H}}\operatorname{st}(x)\simeq \operatorname{st}(x)$. Thus, $\operatorname{proj}_{\tilde{H}}x\simeq \operatorname{st}(x)$.
	
	Therefore, $\operatorname{st}(\operatorname{proj}_{\tilde{H}}x)=\operatorname{st}(x)$, from which we conclude $$\int_{\tilde{\Omega}} |\operatorname{st}\circ\tilde{U}(x)|^2 d\mu_L=\operatorname{st}(\|\operatorname{proj}_{\tilde{H}}x\|^2)=\|\operatorname{st}(x)\|^2.$$
\end{proof}
\begin{definition}
	We define $U_L:H\rightarrow H_L$ with $U_L(x)=\operatorname{st}\circ (\tilde{U}(x))$. Furthermore, we define $m_L:=\operatorname{st}\circ\tilde{\lambda}:\Omega_L\rightarrow\mathbb{R}\cup\{-\infty,\infty\}$ which defines a multiplication operator $T_L$ on $H_L$.
\end{definition}
\begin{theorem}[Partial spectral theorem]\label{theorem_spectral_loebspace}
	We have that $U_L$ is a well-defined linear isometry on $H$. The function $m_L$ is measurable and almost-everywhere real-valued, so that $T_L$ is a densely-defined self-adjoint operator on $H_L$. Furthermore, for any $(x,y)\in\operatorname{st}(G(\tilde{A}))$, $U_L(x)\in\operatorname{dom}(T_L)$ and $U_L(y)=T_L(U_L(x))$.  In particular, we have $U_L\circ A\leq T_L\circ U_L$  (in other words, $T_L\circ U_L$ extends $U_L\circ A$). 
\end{theorem}

\begin{proof}[Proof of the theorem]
	From the previous proposition, we know that for any $x\in H$, $U_L(x)$ is well-defined almost everywhere on $\Omega_L$ and that $U_L(x)\in H_L=L_2(\Omega_L,\mu_L)$ with $\|U_L(x)\|^2=\|x\|^2$. Since $U_L$ is linear, we get that $U_L$ is an isometry.
	
	From general Loeb theory, we know that $m_L$ is a measurable function. To show that it is almost everywhere real valued, for $n\in\mathbb{N}$, let $B_n=\{f\in \tilde{\Omega}\;|\; n<|\tilde{\lambda}_f|^2\}$ while $B_\infty=\{f\in \tilde{\Omega}\;|\;\tilde{\lambda}_f $ is infinite $\}=m_L^{-1}(\pm\infty)$. As with the previous proof, we have that $B_n$ is internal, while $B_\infty=\cap_{n\in\mathbb{N}}B_n$, so that $\mu_L(B_\infty)=\lim_{n\rightarrow\infty}\operatorname{st}(\tilde{\mu}(B_n))$. Then, let $\epsilon\in\mathbb{R}_{>0}$, and choose $n\in\mathbb{N}$ s.t. $\sum_{j=n+1}^N\tilde{c}_j\|\tilde{e}_j\|^2<\frac{\epsilon}{2}$, which exists by Property~(\ref{scale_def_bias}) of Definition~\ref{definition_scale} and underspill. For any $m\in\mathbb{N}$:
	\begin{align*}
		\tilde{\mu}(B_m)&=\sum_{j=1}^N\tilde{c}_j\sum_{f\in B_m}|(\tilde{e}_j,f)|^2\\
		&=\sum_{j=1}^n\tilde{c}_j\sum_{f\in B_m}|(\tilde{e}_j,f)|^2+\sum_{j=n+1}^N\tilde{c}_j\sum_{f\in B_m}|(\tilde{e}_j,f)|^2\\
		&\leq\sum_{j=1}^n\tilde{c}_j\sum_{f\in B_m}\frac{|\tilde{\lambda}_f|^2}{m}|(\tilde{e}_j,f)|^2+\sum_{j=n+1}^N\tilde{c}_j\|\tilde{e}_j\|^2\\
		&\leq\sum_{j=1}^n\tilde{c}_j\sum_{f\in B_m}\frac{1}{m}|(\tilde{e}_j,\tilde{\lambda}_ff)|^2+\frac{\epsilon}{2}\\
		&=\sum_{j=1}^n\tilde{c}_j\sum_{f\in B_m}\frac{1}{m}|(\tilde{e}_j,\tilde{A}f)|^2+\frac{\epsilon}{2}\\
		&=\sum_{j=1}^n\tilde{c}_j\sum_{f\in B_m}\frac{1}{m}|(\tilde{A}\tilde{e}_j,f)|^2+\frac{\epsilon}{2} \leq\frac{1}{m}\sum_{j=1}^n\tilde{c}_j\|\tilde{A}\tilde{e}_j\|^2+\frac{\epsilon}{2}.
	\end{align*}
	
	Hence, we have
	\begin{align*}
		 \mu_L(B_\infty)&\leq\limsup_{m\rightarrow\infty}\left(\operatorname{st}(\frac{1}{m}\sum_{j=1}^n\tilde{c}_j\|\tilde{A}\tilde{e}_j\|^2+\frac{\epsilon}{2})\right)\\
		&=\limsup_{m\rightarrow\infty}\left(\frac{1}{m}\sum_{j=1}^nc_j\|\operatorname{st}(\tilde{A}\tilde{e}_j)\|^2+\frac{\epsilon}{2} \right)=\frac{\epsilon}{2}<\epsilon.
	\end{align*}
	Since $\epsilon$ is arbitrary, we have $\mu_L(B_\infty)=0$. And so $m_L$ is measurable and $\mu_L$-almost everywhere real-valued. It is widely known that the induced multiplication operator $T_L$ is therefore a densely defined self-adjoint operator on $H_L$. 
	 
	For the next part of the theorem, let $(x_0,y_0)\in\operatorname{st}(G(\tilde{A}))\subset H\times H$, and let $\tilde{x}\in\tilde{H}$ such that $\tilde{x}\simeq x_0$ and $\tilde{A}\tilde{x}\simeq y_0$.
	
	By the previous proposition, $\|\operatorname{st}\circ\tilde{U}(\tilde{x}-x_0)\|_{H_L}=\|\operatorname{st}(\tilde{x}-x_0)\|_H=0$, so for almost every $f$ (w.r.t. $\mu_L$), $\operatorname{st}\left((\tilde{U}(\tilde{x}))(f)\right)=(U_L(x_0))(f)$. Similarly, we have that  $\operatorname{st}\left((\tilde{U}(\tilde{A}\tilde{x}))(f)\right)=(U_L(y_0))(f)$ holds almost everywhere. Let $B\subset \tilde{\Omega}$ be the set of $f$ for which:
	\begin{itemize}
		\item $\operatorname{st}(\tilde{\lambda}_f)=m_L(f)\in\mathbb{R}$.
		\item$\operatorname{st}\left((\tilde{U}(\tilde{x}))(f)\right)=(U_L(x_0))(f)\in\mathbb{K}$.
		\item$\operatorname{st}\left((\tilde{U}(\tilde{A}\tilde{x}))(f)\right)=(U_L(y_0))(f)\in\mathbb{K}$.
	\end{itemize}
	$B$ is then Loeb-measurable, and $\mu_L(\Omega_L\setminus B)=0$. For any $f\in B$:
	\begin{align*}
		m_L(f)(U_L(x_0))(f)&=\operatorname{st}\left(\tilde{\lambda}_f \cdot(\tilde{U}(\tilde{x}))(f)\right)\\
		&=\operatorname{st}\left(\tilde{\lambda}_f\frac{(\tilde{x},f)}{\sqrt{\tilde{\mu}(f)}}\right)=\operatorname{st}\left( \frac{(\tilde{x},\tilde{\lambda}_ff)}{\sqrt{\tilde{\mu}(f)}}\right)\\
		&=\operatorname{st}\left(\frac{(\tilde{x},\tilde{A}f)}{\sqrt{\tilde{\mu}(f)}}\right)=\operatorname{st}\left(\frac{(\tilde{A}\tilde{x},f)}{\sqrt{\tilde{\mu}(f)}}\right)\\
		&=\operatorname{st}\left((\tilde{U}(\tilde{A}\tilde{x}))(f)\right)=(U_L(y_0))(f).
	\end{align*}
	Therefore, $\int_{\Omega_L}|m_LU_L(x_0)|^2d\mu_L=\int_{\Omega_L}|U_L(y_0)|^2d\mu_L=\|y_0\|^2\in\mathbb{R}$ and so $U_L(x_0)\in\operatorname{dom}(T_L)$, with $T_L( U_L(x_0))=m_L\cdot U_L(x_0)=U_L(y_0)$. 
	
	In particular, if $x\in\operatorname{dom}(A)$ is arbitrary, then $(x,Ax)\in\operatorname{st}(G(\tilde{A}))$, implying $x\in\operatorname{dom}(T_L\circ U_L)$ and $(T_L\circ U_L)(x)=U_L(Ax)$. We conclude $U_L\circ A\leq T_L\circ U_L$.
\end{proof}

\begin{remark}
	For Theorem~\ref{theorem_spectral_loebspace}, $U_L$ is expected not to be unitary in most cases. In fact, $H_L$ may not even be separable. Also, a direct consequence of the theorem is that, for $j\in\mathbb{N}$, we have $U_L(\operatorname{st}(\tilde{A}\tilde{e}_j))=T_L(U_L(e_j))$. Since we can show $\operatorname{st}(\tilde{A}\tilde{e}_j)=A^*e_j$, where $A^*$ is the adjoint of $A$, we get $U_L(A^* e_j)=T_L(U_L(e_j))$ for any $j\in\mathbb{N}$. This is not expected to generalize on $\operatorname{dom}(A^*)$. In fact, one can show $U_L\circ A^*\leq T_L\circ U_L$ if and only if $A$ is essentially self-adjoint.
\end{remark}
	\section{The induced pseudometric and the hull space}\label{section_hullspace}
In this section, we will construct the objects of Theorem~\ref{theorem_spectral_symmetric},  completing its proof. In contrast with Theorem~\ref{theorem_spectral_loebspace}, the target space will be separable and the measure space will be a compact metric space.

First, we prove the following.
\begin{prop}
	For any $j\in [N]$, $\tilde{c}_j\max(|\tilde{U}(\tilde{e}_j)|^2)\leq 1$.
\end{prop}
\begin{proof}
	Let $f\in \tilde{\Omega}$. We have
	\begin{align*}
		\tilde{c}_j|(\tilde{U}(\tilde{e}_j))(f)|^2&=\frac{\tilde{c}_j|(\tilde{e}_j,f)|^2}{\tilde{\mu}(f)}\leq \frac{\tilde{\mu}(f)}{\tilde{
				\mu}(f)}=1.
	\end{align*}
\end{proof}
\begin{definition}
	We define $\tilde{d}:\tilde{\Omega}\times \tilde{\Omega}\rightarrow{}^*\mathbb{R}_{\geq 0}$ with
	\begin{align*}
		\tilde{d}(f_1,f_2)=\sum_{j=1}^N\tilde{c}_j^{\frac{3}{2}}\|\tilde{e}_j\|^2|(\tilde{U}(\tilde{e}_j))(f_1)-(\tilde{U}(\tilde{e}_j))(f_2)|.
	\end{align*}
\end{definition}
\begin{prop}
	$\tilde{d}$ is an internal pseudometric on $\tilde{\Omega}$. Furthermore $\tilde{d}\leq 2$.
\end{prop}
\begin{remark}
	As a clarification, a pseudometric is similar to a distance, with the exception that distance between points can be $0$. One can show that if the $(\tilde{e}_j)_{j\in[N]}$ forms a ${}^*$basis of $\tilde{H}$ and all $\tilde{c}_j$ are positive, then $\tilde{d}$ is actually a metric. However, such a condition is not necessary nor assumed for what comes next.
\end{remark}
\begin{proof}
	$\tilde{d}$ has an internal domain and co-domain. Moreover, its definition only uses internal objects. Thus, $\tilde{d}$ is internal. Furthermore, symmetry, non-negativity and triangle inequality directly hold.  We show the upper bound. For any $f_1$ and $f_2$ in $\tilde{\Omega}$,
	\begin{align*}
		\tilde{d}(f_1,f_2)\leq\sum_{j=1}^N \tilde{c}_j\|\tilde{e}_j\|^22\sqrt{\tilde{c}_j}\max(|\tilde{U}(\tilde{e}_j)|)\leq 2\sum_{j=1}^N \tilde{c}_j\|\tilde{e}_j\|^2=2.
	\end{align*} 
\end{proof}
Next, we construct a nonstandard hull based on $\tilde{d}$ to obtain a real metric space, which will be the main object we work on. The methods used are quite similar as what is found in the literature (see \cite{arkeryd2012nonstandard}). As the object here is an internal pseudometric instead of a standard metric, we prove that everything works out for the sake of completion.
\begin{definition}
	For $f_1$, $f_2$ in $\tilde{\Omega}$, $f_1\sim f_2$ if $\tilde{d}(f_1,f_2)$ is infinitesimal (so if $f_1$ and $f_2$ are either indistinguishable or infinitely close).
\end{definition}
We then have the following.
\begin{prop}\label{proposition_sim_iff_Uj}
	$\sim$ is an equivalence relation on $\tilde{\Omega}$. Furthermore, $f_1\sim f_2$ if and only if for all $j\in \mathbb{N}$, $\operatorname{st}((\tilde{U}(\tilde{e}_j))(f_1))=\operatorname{st}((\tilde{U}(\tilde{e}_j))(f_2))$.
\end{prop}
\begin{proof}
	It is direct that $\sim$ is an equivalence relation. Indeed, reflexivity follows from $\tilde{d}(f,f)=0$, symmetry follows from symmetry of $\tilde{d}$ and transitivity follows from triangle inequality. 
	
	Let $f_1,f_2\in\tilde{\Omega}$ such that  $f_1\sim f_2$, and let $j\in\mathbb{N}$. We know
	\begin{align*}
		\operatorname{st}(\tilde{c}_j^{\frac{3}{2}}\|\tilde{e}_j\|^2|(\tilde{U}(\tilde{e}_j))(f_1)-(\tilde{U}(\tilde{e}_j))(f_2)|)=0.
	\end{align*}
	
	 One of the three terms in the product must be infinitesimal. By Property (\ref{scale_def_NS}) of Definition~\ref{definition_scale}, both $\tilde{c}_j$ and $\|\tilde{e}_j\|$ are not. Therefore, $\operatorname{st}((\tilde{U}(\tilde{e}_j))(f_1))=\operatorname{st}((\tilde{U}(\tilde{e}_j))(f_2))$.
	
	On the other hand, suppose that  for all $j\in\mathbb{N}$,  we have that $\operatorname{st}((\tilde{U}(\tilde{e}_j))(f_1))=\operatorname{st}((\tilde{U}(\tilde{e}_j))(f_2))$. Since all involved terms are limited, we have
	\begin{align*}
		\operatorname{st}(\tilde{c}_j^{\frac{3}{2}}\|\tilde{e}_j\|^2|(\tilde{U}(\tilde{e}_j))(f_1)-(\tilde{U}(\tilde{e}_j))(f_2)|)=0
	\end{align*}
	 for any standard $j$. Furthermore, let $\epsilon\in\mathbb{R}_{>0}$. For any infinite $K\in[N]$, we have :
	\begin{align*}
		0\leq \sum_{j=K}^{N}\tilde{c}_j^{\frac{3}{2}}\|\tilde{e}_j\|^2|(\tilde{U}(\tilde{e}_j))(f_1)-(\tilde{U}(\tilde{e}_j))(f_2)|\leq 2\sum_{j=K}^{N}\tilde{c}_j\|\tilde{e}_j\|^2<\frac{\epsilon}{2}.
	\end{align*}
	
	The last part also holds by Property (\ref{scale_def_bias}) of Definition~\ref{definition_scale}. By underspill, there exists $k\in\mathbb{N}$ such that $\sum_{j=k}^{N}\tilde{c}_j^{\frac{3}{2}}\|\tilde{e}_j\|^2|(\tilde{U}(\tilde{e}_j))(f_1)-(\tilde{U}(\tilde{e}_j))(f_2)|<\frac{\epsilon}{2}$. Since the first $k$ terms are infinitesimal, we get $\tilde{d}(f_1,f_2)<\epsilon$. Since $\epsilon$ is arbitrary, we have $f_1\sim f_2$.
\end{proof}
\begin{definition}
	Let $\hat{\Omega}:= \Omega_L/\sim$, and $\hat{\nu}:\Omega_L\rightarrow\hat{\Omega}$ be the natural map. Furthermore, let $\hat{d}:\hat{\Omega}\times\hat{\Omega}\rightarrow\mathbb{R}_{\geq0}$ with $\hat{d}(\hat{\nu}(f_1),\hat{\nu}(f_2))=\operatorname{st}(\tilde{d}(f_1,f_2))$.
\end{definition}

\begin{prop}
	$\hat{d}$ is a well-defined classical distance on $\hat{\Omega}$. Furthermore, $(\hat{\Omega},\hat{d})$ is a compact metric space and therefore, it must be separable and second countable.
\end{prop}
\begin{proof}
	First, we must show $\hat{d}$ is well-defined. If $\hat{\nu}(f_1)=\hat{\nu}(f_1')$ and $\hat{\nu}(f_2)=\hat{\nu}(f_2')$, we have, since $\tilde{d}$ is an internal pseudometric:
	\begin{align*}
		|\tilde{d}(f_1,f_2)-\tilde{d}(f_1',f_2')|&\leq\tilde{d}(f_1,f_1')+\tilde{d}(f_2,f_2')\\
		\implies 0\leq|\operatorname{st}(\tilde{d}(f_1,f_2))-\operatorname{st}(\tilde{d}(f_1',f_2'))|&\leq \operatorname{st}(\tilde{d}(f_1,f_1')+\tilde{d}(f_2,f_2'))=0\\
		\implies \hat{d}(\hat{\nu}(f_1),\hat{\nu}(f_2))&=\hat{d}(\hat{\nu}(f_1'),\hat{\nu}(f_2')).
	\end{align*}
	
	Furthermore, the non-negativity, symmetry and triangle inequality of $d$ can be proven directly from the same properties of $\tilde{d}$. We also have $\hat{d}(\hat{\nu}(f_1),\hat{\nu}(f_2))=0$ if and only if $ \operatorname{st}(\tilde{d}(f_1,f_2))=0$, if and only if $f_1\sim f_2$, if and only if $\hat{\nu}(f_1)=\hat{\nu}(f_2)$, and so $\hat{d}$ is a metric on $\hat{\Omega}$. 
	
	Now we prove compactness by showing that $(\hat{\Omega},\hat{d})$ is totally bounded and complete.
	
	We first show the totally bounded property.
	
	Let $\epsilon\in\mathbb{R}_{>0}$. We find a finite set $\{f_1,f_2,\dots,f_k\}\subset \Omega_L$ such that for all $f\in \Omega_L$, there exists $l\in[k]$ with  $\hat{d}(\hat{\nu}(f),\hat{\nu}(f_l))<\epsilon$. First, let $n\in\mathbb{N}$ for which $\sum_{j=n+1}^N2\tilde{c}_j\|\tilde{e}_j\|^2\leq\frac{\epsilon}{2}$, so that, for any $f$, $f'$ in $\tilde{\Omega}$:
	\begin{align*}
		\hat{d}(\hat{\nu}(f),\hat{\nu}(f'))\leq\frac{\epsilon}{2}+\sum_{j=1}^nc_j^{\frac{3}{2}}\|e_j\|^2|\operatorname{st}((\tilde{U}(\tilde{e}_j))(f))-\operatorname{st}((\tilde{U}(\tilde{e}_j))(f'))|.
	\end{align*}
	This leads to the definition of the (clearly external) function $X_n:\Omega_L\rightarrow \mathbb{K}^n$ with $X_n(f)_j=c_j^{\frac{3}{2}}\|e_j\|^2\operatorname{st}((\tilde{U}(\tilde{e}_j))(f))$ so that the previous inequality becomes
	\begin{align*}
		\hat{d}(\hat{\nu}(f),\hat{\nu}(f'))\leq\frac{\epsilon}{2}+\|X_n(f)-X_n(f') \|_1.
	\end{align*}
	Since $|X_n(f)_j|\leq 1$ for any $f$ and any $j$, we have $X_n(\Omega_L)$ is a bounded set in $\mathbb{K}^n$, and so is totally bounded with respect to $\|\cdot\|_1$. Thus, there exists a finite subset $\{f_1,f_2,\dots,f_k\}$ of $\Omega_L$ such that for all $f\in \Omega_L$, there exists $l\in[k]$ for which  $\|X_n(f)-X_n(f_l)\|_1<\frac{\epsilon}{2}$, so $\hat{d}(\hat{\nu}(f),\hat{\nu}(f_l))<\epsilon$. Since $\epsilon$ is arbitrary, $\hat{\Omega}$ is totally bounded.
	
	Next, we prove the completeness.
	
	Let $(\hat{\nu}(f_k))_{k\in\mathbb{N}}$ be a Cauchy sequence in $\hat{\Omega}$. We use the classical technique of assuming $\hat{d}(\hat{\nu}(f_k),\hat{\nu}(f_{k+1}))<\frac{1}{2^k}$, without loss of generality. Indeed any Cauchy sequence has a subsequence with that property, and the convergence of a subsequence implies the convergence of the sequence.
	
	By $\aleph_1$-saturation, the sequence $(f_k)_{k\in\mathbb{N}}$ can be extended to an internal sequence $(\tilde{f}_k)_{k\in{}^*\mathbb{N}}$. Then, we use overspill on $\{k\in{}^*\mathbb{N}\;|\;\tilde{d}(\tilde{f}_k,\tilde{f}_{k+1})<\frac{1}{2^k}\}$. This set is internal and has $\mathbb{N}$ as a subset, therefore also has $[K]$ as a subset for some infinite natural $K$.
	
	We show $\lim_{k\rightarrow\infty}\hat{\nu}(f_k)=\hat{\nu}(\tilde{f}_K)$. Let $\epsilon\in\mathbb{R}_{>0}$. Let $n\in\mathbb{N}$ s.t. $\frac{1}{2^n}<\epsilon$. For any $k\in\mathbb{N}$ with $k>n$,
	\begin{align*}
		\hat{d}(\hat{\nu}(f_k),\hat{\nu}(\tilde{f}_K))&=\operatorname{st}(\tilde{d}(f_k,\tilde{f}_K))\leq\operatorname{st}\left(\sum_{l=k}^{K-1}\tilde{d}(\tilde{f}_l,\tilde{f}_{l+1})\right)\\
		&\leq\operatorname{st}\left( \sum_{l=k}^{K-1}\frac{1}{2^l} \right)=\frac{1}{2^{k-1}}\leq\frac{1}{2^n}<\epsilon
	\end{align*}
	as planned. Therefore, $\hat{\Omega}$ is complete. All of that shows $(\hat{\Omega},\hat{d})$ is a compact metric space, all of which are separable and second countable.
\end{proof}

\begin{remark}
	From now on, we use the notation $\hat{B}_r(x)=\{y\in\hat{\Omega}\;|\;\hat{d}(y,x)<r\}$, for balls in $\hat{\Omega}$ (here, $r\in\mathbb{R}_{>0}$, $x\in\hat{\Omega}$). Similarly, we use $\tilde{B}_r(f)=\{f'\in \tilde{\Omega}\;|\;\tilde{d}(f',f)<r\}$ for internal balls in $\tilde{\Omega}$ (here, $r\in{}^*\mathbb{R}_{>0}$, $x\in \tilde{\Omega}$).
\end{remark}
\begin{definition}
	We define $\hat{\mathcal{A}}=\operatorname{Borel}(\hat{\Omega},\hat{d})$, the $\sigma$-algebra of Borel sets on $\hat{\Omega}$.
\end{definition}
\begin{prop}
	$\hat{\nu}:(\Omega_L,\mathcal{A}_L)\rightarrow (\hat{\Omega}, \hat{\mathcal{A}})$ is a measurable function.
\end{prop}

\begin{proof}
	It is sufficient to show that for any standard $\epsilon>0$, and for any $f$ in $\tilde{\Omega}$, $\hat{\nu}^{-1}(\hat{B}_\epsilon(\hat{\nu}(f)))$ is Loeb-measurable. But then,
	\begin{align*}
		\hat{\nu}^{-1}(\hat{B}_\epsilon(\hat{\nu}(f)))=\{f'\in \tilde{\Omega}\;|\;\operatorname{st}(\tilde{d}(f',f))<\epsilon\}=\bigcup_{\substack{n\in\mathbb{N}\\n>\frac{1}{\epsilon}}}\tilde{B}_{\epsilon-\frac{1}{n}}(f).
	\end{align*}
	Therefore, $\hat{\nu}^{-1}(B_\epsilon(\hat{\nu}(f)))$ is a countable union of internal subsets of $\tilde{\Omega}$, and so it is Loeb-measurable.
\end{proof}	
\begin{definition}
	Let $\hat{\mu}:\hat{\mathcal{A}}\rightarrow\mathbb{R}_{\geq 0}$ be the pushforward probability measure of $\mu_L$ through $\hat{\nu}$, so that $\hat{\mu}(W)=\mu_L(\hat{\nu}^{-1}(W))$. Furthermore, let $\hat{H}$ be the $\mathbb{K}$-Hilbert space $\hat{H}:=L_2(\hat{\Omega},\hat{\mathcal{A}},\hat{\mu})$. 
\end{definition}
\begin{prop}
	$\hat{H}$ is separable.
\end{prop}
\begin{proof}
	$\hat{\Omega}$ is second countable, and so $\hat{\mathcal{A}}$ is countably generated. Therefore, $(\hat{\Omega},\hat{\mathcal{A}},\hat{\mu})$ is a separable probability space, which implies that $\hat{H}$ itself is separable.
\end{proof}
\begin{remark}
	Now, what's left is to prove that the relevant objects $U_L$ and $m_L$ can be push-forwarded through $\hat{\nu}$. In contrast with the measure and the metric, such functions don't push smoothly, but they do pull back well. First, we have to find objects that can easily be push-forwarded, then play some kind of tug-of-war between spaces to push what we need.
\end{remark}
\begin{definition}
	Let $\hat{\mathcal{I}}:\hat{H}\rightarrow H_L$ be the isometry defined by $\hat{\mathcal{I}}(g)=g\circ\hat{\nu}$.
\end{definition}
\begin{remark}
	The fact that $\hat{\mathcal{I}}$ is indeed a well-defined isometry is a well-known consequence of push-forward measures.
\end{remark}
\begin{prop}\label{proposition_existence_Uj}
	For any $j\in\mathbb{N}$, there exists a $\hat{d}$-continuous $\hat{U}_j:\hat{\Omega}\rightarrow\mathbb{K}$ such that $\hat{U}_j\circ\hat{\nu}=\operatorname{st}\circ(\tilde{U}(\tilde{e}_j))$ holds on $\Omega_L$. Furthermore, $U_L(H)\subset\hat{\mathcal{I}}(\hat{H})$.
\end{prop}
\begin{proof}
	Let $j\in\mathbb{N}$. 
	
	By Proposition~\ref{proposition_sim_iff_Uj}, we know that $\operatorname{st}\circ\left(\tilde{U}(\tilde{e}_j)\right)(f_1)=\operatorname{st}\circ\left(\tilde{U}(\tilde{e}_j)\right)(f_2)$ holds for any $f_1,f_2\in\tilde{\Omega}$ for which $\hat{\nu}(f_1)=\hat{\nu}(f_2)$. Therefore, there exists a unique function $\hat{U}_j:\hat{\Omega}\rightarrow\mathbb{K}$ such that $\hat{U}_j\circ\hat{\nu}=\operatorname{st}\circ\left(\tilde{U}(\tilde{e}_j)\right)$ holds on $\Omega_L$. Furthermore, for any such $f_1, f_2\in\Omega_L$:
	\begin{align*}
		\left|\left(\hat{U}_j(\hat{\nu}(f_1))-\hat{U}_j(\hat{\nu}(f_2))\right)\right|&=\operatorname{st}\left(\left|(\tilde{U}(\tilde{e}_j))(f_1)-(\tilde{U}(\tilde{e}_j))(f_2)\right|\right)\\
		&\leq\operatorname{st}\left(\frac{1}{\tilde{c}_j^{\frac{3}{2}}\|\tilde{e}_j\|^2}\tilde{d}(f_1,f_2)\right)=\frac{1}{\operatorname{st}(\tilde{c}_j^{\frac{3}{2}}\|\tilde{e}_j\|^2)}\hat{d}(\hat{\nu}(f_1),\hat{\nu}(f_2)).
	\end{align*}
	We remind that by Property (\ref{scale_def_NS}) of Definition~\ref{definition_scale}, the denominator of that last fraction is positive. Thus, $\hat{U}_j$ is $\hat{d}$-Lipschitz, therefore continuous. Since $\hat{\Omega}$ is a compact probability space, $\hat{U}_j$ is also $L_2$ on $\hat{\Omega}$. 
	
	As with the proof of Theorem~\ref{theorem_spectral_loebspace}, we know that $\tilde{e}_j\simeq e_j$, thus  $$\|\operatorname{st}\circ\left(\tilde{U}(\tilde{e}_j-e_j)\right)\|_{H_L}=0.$$ Therefore, as elements of $H_L$, we have $U_L(e_j)=\operatorname{st}\circ(\tilde{U}(\tilde{e}_j))$. We get that $\hat{\mathcal{I}}(\hat{U}_j)=\hat{U}_j\circ\hat{\nu}=\operatorname{st}\circ\left(\tilde{U}(\tilde{e}_j)\right)=U_L(e_j)$ as elements of $H_L$.
	
	From that, we deduce $U_L(\operatorname{span}(\{e_j\;|\;j\in\mathbb{N} \}))\subset\hat{\mathcal{I}}(\hat{H})$. Since both $U_L$ and $\hat{\mathcal{I}}$ are isometries (which map closed sets to closed sets), we can take the closure on both sides and obtain $U_L(H)\subset\hat{\mathcal{I}}(\hat{H})$, concluding the proof.
\end{proof}
\begin{remark}\label{remark_definition_Uj}
	From now on, the objects $(\hat{U}_j)_{j\in\mathbb{N}}$ are fixed, with the properties of Proposition~\ref{proposition_existence_Uj}.
\end{remark}

\begin{definition}
	We define the map $\hat{U}:=\hat{\mathcal{I}}^{-1}\circ U_L:H\rightarrow\hat{H}$.
\end{definition}
\begin{remark}
	Since $\hat{\mathcal{I}}$ and $U_L$ are isometries, $\hat{U}$ is also one. Furthermore, it is the unique function such that $\hat{\mathcal{I}}\circ\hat{U}=U_L$.
	
	Intuitively, what was straightforward to push through $\hat{\nu}$ was $\operatorname{st}\circ(\tilde{U}(\tilde{e}_j))$ to $\hat{U}_j$, which was used to define $\hat{U}(x)$ for any $x\in H$. The method resembles how the Fourier transform can be created on $L_2$ by first considering $L_2\cap L_1$. The connection with that example is in fact quite deeper, as we will see in Section \ref{section_differential}.
	
	Some minutia will be required to push $m_L$ through $\hat{\nu}$.
\end{remark}

\begin{prop}\label{proposition_pushforward_mult}
	There exists a measurable function $\hat{m}:\hat{\Omega}\rightarrow\mathbb{R}$ s.t. $m_L=\hat{m}\circ\hat{\nu}$, in the sense that they agree $\mu_L$-almost everywhere on $\Omega_L$.
\end{prop}
Before we prove this, we first need a definition and a lemma.
\begin{definition}\label{definition_support_total_Uj}
	 We define the sets
	\begin{align*}
		\hat{\Omega}'&=\bigcup_{j\in\mathbb{N}} \hat{U}_j^{-1}(\mathbb{K}\setminus\{0\})\\
		\Omega_L'&=\hat{\nu}^{-1}(\hat{\Omega'}).
	\end{align*}
\end{definition}
\begin{remark}\label{remark_OmegaL_prime_from_U_L}
	We directly have $\Omega_L'=\bigcup_{j\in\mathbb{N}}\left(\operatorname{st}\circ\left(\tilde{U}(\tilde{e}_j)\right)\right)^{-1}(\mathbb{K}\setminus\{0\})$. Therefore, for any $f\in\Omega_L$, $f\in \Omega_L\setminus\Omega_L'$ if and only if $\forall j\in\mathbb{N}$, $\operatorname{st}\left((\tilde{U}(\tilde{e}_j))(f)\right)=0$, so $\Omega_L\setminus\Omega_L'$ is either empty or one equivalence class. Thus, $\hat{\Omega}\setminus\hat{\Omega}'$ is either empty or a singleton.
\end{remark}
\begin{lemma}\label{lemma_full_measure_Uj}
	$\hat{\Omega}'\in\hat{\mathcal{A}}$, $\Omega_L'\in\mathcal{A}_L$, and $\hat{\mu}(\hat{\Omega}')=\mu_L(\Omega_L')=1$.
\end{lemma}
\begin{proof}
	Since each $\hat{U}_j$ is continuous, $\hat{\Omega}'$ is open. Therefore, $\hat{\Omega}'$ is in $\hat{\mathcal{A}}$. As $\hat{\nu}$ is a measurable function, $\Omega_L'\in\mathcal{A}_L$. By definition, we have $\hat{\mu}(\hat{\Omega}')=\mu_L(\Omega_L')$. To conclude the proof, it is sufficient to show $\mu_L(\Omega_L')=1$
	
	For $k,n$ in $\mathbb{N}$, let $X_{k,n}=\{f\in \tilde{\Omega}\;\mid\; \forall j\in[k],$ $\tilde{c}_j|(\tilde{U}(\tilde{e}_j))(f)|^2\leq\frac{1}{n}\}$. We have  that $X_{k,n}$ is internal, and
	\begin{align*}
		\Omega_L\setminus\Omega_L'=\{f\in\Omega_L\;\mid\;\forall j\in\mathbb{N},\;\operatorname{st}(\tilde{U}(\tilde{e}_j)(f))=0\} \subset\bigcap_{k,n\in\mathbb{N}}X_{k,n}.
	\end{align*}
	
	Furthermore, for $k$, $n$ in $\mathbb{N}$:
	\begin{align*}
		\tilde{\mu}(X_{k,n})&=\sum_{j=1}^N\tilde{c}_j\sum_{f\in X_{k,n}}|(\tilde{e}_j,f)|^2\\
		&=\sum_{j=1}^k\tilde{c}_j\sum_{f\in X_{k,n}}|(\tilde{e}_j,f)|^2+\sum_{j=k+1}^N\tilde{c}_j\sum_{f\in X_{k,n}}|(\tilde{e}_j,f)|^2 \\
		&\leq\sum_{j=1}^k\tilde{c}_j\sum_{f\in X_{k,n}}\tilde{\mu}(f)|(\tilde{U}(\tilde{e}_j))(f)|^2 +\sum_{j=k+1}^N \tilde{c}_j\|\tilde{e}_j\|^2\\
		&\leq \frac{1}{n}\sum_{j=1}^k\sum_{f\in X_{k,n}}\tilde{\mu}(f)+\sum_{j=k+1}^N \tilde{c}_j\|\tilde{e}_j\|^2\\
		&=\tilde{\mu}(X_{k,n})\frac{k}{n}+\sum_{j=k+1}^N \tilde{c}_j\|\tilde{e}_j\|^2\leq \frac{k}{n}+\sum_{j=k+1}^N \tilde{c}_j\|\tilde{e}_j\|^2. \\
	\end{align*}
	Hence,
	\begin{align*}	
		 \mu_L(\Omega_L\setminus\Omega_L')\leq\mu_L(X_{k,n})\leq\frac{k}{n}+\operatorname{st}\left(\sum_{j=k+1}^N \tilde{c}_j\|\tilde{e}_j\|^2\right).
	\end{align*}
	By Property~(\ref{scale_def_bias}) of Definition~\ref{definition_scale}, we know that $\lim_{k\rightarrow\infty}\operatorname{st}\left(\sum_{j=k+1}^N \tilde{c}_j\|\tilde{e}_j\|^2\right)=0$. By posing $n=k^2$ and taking limits, we find $\mu_L(\Omega_L\setminus\Omega_L')=0$, concluding the proof of the lemma.
\end{proof}

Now, we can proceed with the proof.
\begin{proof}[Proof of Proposition~\ref{proposition_pushforward_mult}]
	The idea is that $U_L(Ax)=m_L\cdot U_L(x)$, and so $m_L$ can be written in terms of objects that can be pushed through $\hat{\nu}$. One has to be careful when $U_L(x)=0$, which is why $\hat{\Omega}'$ is used. 
	
	First, we partition $\hat{\Omega}'$ with $\{C_j\}_{j\in\mathbb{N}}$, defined by: 
	
	\begin{align*}
		C_j:=\left(\hat{U}_j^{-1}(\mathbb{K}\setminus\{0\})\right)\setminus\left(\bigcup_{k=1}^{j-1}\left(\hat{U}_k^{-1}(\mathbb{K}\setminus\{0\})\right)\right).
	\end{align*} 
	
	All $C_j$ are measurable in $\hat{\Omega}$, and they are pairwise disjoint. Since $\bigcup_{j\in\mathbb{N}}C_j=\hat{\Omega}'$, $\{C_j\}_{j\in\mathbb{N}}$ forms a measurable countable partition.
		
	Furthermore, for each $j\in\mathbb{N}$, $\left(\frac{\hat{U}(\operatorname{st}(\tilde{A}\tilde{e}_j))}{\hat{U}_j}\right)|_{C_j}$ is almost-everywhere defined in $C_j$, with values in $\mathbb{K}$, and is measurable (more precisely, we assume an explicit representation of $\hat{U}(\operatorname{st}(\tilde{A}\tilde{e}_j))$ in its $L_2$ class). Thus, there exists a measurable function $\hat{m}:\hat{\Omega}\rightarrow\mathbb{R}$ such that for any $j\in\mathbb{N}$, $\hat{m}|_{C_j}=\Re\left(\frac{\hat{U}(\operatorname{st}(\tilde{A}\tilde{e}_j))}{\hat{U}_j}\right)\big|_{C_j}$. 
	
	All that is left is to show $m_L=\hat{m}\circ\hat{\nu}$. For any $j\in\mathbb{N}$, we have, $\mu_L$-almost everywhere on $\hat{\nu}^{-1}(C_j)$:
	\begin{align*}
		(\hat{m}\circ\hat{\nu})|_{\hat{\nu}^{-1}(C_j)}&=\hat{m}|_{C_j}\circ\hat{\nu}|_{\hat{\nu}^{-1}(C_j)}\\
		&=\Re\left(\frac{\hat{U}(\operatorname{st}(\tilde{A}\tilde{e}_j))}{\hat{U}_j}\right)\bigg|_{C_j}\circ\hat{\nu}|_{\hat{\nu}^{-1}(C_j)}\\
		&=\Re\left(\frac{\hat{U}(\operatorname{st}(\tilde{A}\tilde{e}_j))}{\hat{U}_j}\circ\hat{\nu}\right)\bigg|_{\hat{\nu}^{-1}(C_j)}\\
		&=\Re\left(\frac{U_L(\operatorname{st}(\tilde{A}\tilde{e}_j))}{U_L(e_j)}\right)\bigg|_{\hat{\nu}^{-1}(C_j)}\\
		&=\Re(m_L)|_{\hat{\nu}^{-1}(C_j)}=m_L|_{\hat{\nu}^{-1}(C_j)}.
	\end{align*}
	Furthermore, we have
	\begin{align*}
		\bigcup_{j\in\mathbb{N}}\hat{\nu}^{-1}(C_j)=\hat{\nu}^{-1}\left(\bigcup_{j\in\mathbb{N}} C_j\right)=\hat{\nu}^{-1}(\hat{\Omega}')=\Omega_L'.
	\end{align*}
	Therefore, $\hat{m}\circ\hat{\nu}$ and $m_L$ agree almost everywhere on $\Omega_L'$, and so agree almost-everywhere on $\Omega_L$.
\end{proof}
Now, along with $(\hat{\Omega},\hat{d})$, $\hat{\mu}$ and $\hat{U}$, we can define $\hat{m}$ and $\hat{T}$, completing the list of objects induced by our process.

\begin{definition}
	Fix $\hat{m}:\hat{\Omega}\rightarrow\mathbb{R}$ as any measurable function such that $\hat{m}\circ\hat{\nu}=m_L$ (in the almost-everywhere sense). Furthermore, let $\hat{T}$ be the induced multiplication operator on $\hat{H}$.
\end{definition}
\begin{remark}
	As a reminder, $\hat{T}$ is automatically densely defined and self-adjoint, since $\hat{m}$ is $\hat{\mu}$-almost everywhere real valued.
\end{remark}
We now show these objects realize the conditions of Theorem~\ref{theorem_spectral_symmetric}, completing its proof.
\begin{proof}[Proof of Theorem~\ref{theorem_spectral_symmetric}]
	For any $x\in\operatorname{dom}(A)$, we have, $\mu_L$-almost everywhere, as a direct consequence of Theorem~\ref{theorem_spectral_loebspace}:
	\begin{align*}
		(\hat{m}\cdot \hat{U}(x))\circ\hat{\nu}=m_L\cdot U_L(x)=T_L(U_L(x))=U_L(Ax)=(\hat{U}(Ax))\circ\hat{\nu}\\
		\implies \int_{\hat{\Omega}}|\hat{m}\cdot \hat{U}(x)|^2d\mu=\int_{\hat{\Omega}}|\hat{U}(Ax)|^2d\mu=\|Ax\|^2\in\mathbb{R}.
	\end{align*}
	Therefore, $\hat{U}(x)\in\operatorname{dom}(\hat{T})$. Furthermore, we have found $\hat{\mathcal{I}}(\hat{T}(\hat{U}x))=\hat{\mathcal{I}}(\hat{U}(Ax))$ from which we deduce $\hat{T}(\hat{U}x)=\hat{U}(Ax)$. Since $x$ is arbitrary, we have 
	\begin{align*}
		\hat{U}\circ A\leq \hat{T}\circ \hat{U}.
	\end{align*}
\end{proof}
	\section{Self-adjoint operators}\label{section_sa}

In this section specifically, we now assume $A$ is self-adjoint, and find out how Theorem~\ref{theorem_spectral_symmetric} relates to the other versions of the spectral theorem. More specifically, we will show the following theorem and discuss its relevancy:

\begin{theorem}\label{spectral_theorem}
	If $A$ is self-adjoint, there exists a sampling for $A$ paired with a compatible standard-biased scale such that the isometry $\hat{U}$ induced by the process is surjective.
\end{theorem}
\begin{remark}
	Here, the condition can directly be loosened to $A$ having some self-adjoint extension on $H$, since  a sampling for an extension of $A$ is also a sampling for $A$. It cannot be loosened further: if $\hat{U}$ is surjective, then $\hat{U}^{-1}\circ \hat{T}\circ\hat{U}$ is a self-adjoint extension of $A$.
\end{remark}

In proving this, the main problem is that not every sampling and scale will work just from $A$ being self-adjoint. In fact, if $B$ is a symmetric operator on $H$ with no self-adjoint extension, and $(\tilde{H},\tilde{B},\tilde{\Omega})$ is a sampling for $B$, then $(\tilde{H},\operatorname{id}_{\tilde{H}},\tilde{\Omega})$ is a sampling for $\operatorname{id_H}$. Furthermore, any standard-biased scale $((\tilde{e}_j)_{j\in[N]}(\tilde{c}_j)_{j\in[N]})$ compatible with $(\tilde{H},\tilde{B},\tilde{\Omega})$ will be compatible with $(\tilde{H},\operatorname{id}_{\tilde{H}},\tilde{\Omega})$. From inspection, neither $\tilde{B}$ nor $\operatorname{id}_{\tilde{H}}$ impact the resulting measure space or isometry, they only impact the multiplication operator. Thus, we get that even though $\operatorname{id}_H$ is self-adjoint (and bounded), the resulting isometry $\hat{U}$ from $(\tilde{H},\operatorname{id}_{\tilde{H}},\tilde{\Omega})$ and $((\tilde{e}_j)_{j\in[N]}(\tilde{c}_j)_{j\in[N]})$ will not be surjective.

What that seems to indicate is that the sampling or the scale has to be created using a method that cannot work unless the operator at least has a self-adjoint extension. For that, we turn to projection-valued measures:

\begin{definition}
	A function of the form $P:\operatorname{Borel}(\mathbb{R})\rightarrow\{\operatorname{proj}_E\;|\;\text{E $\leq$ H is closed} \}$ is called a projection-valued measure if 
	\begin{itemize}
		\item $P(\emptyset)=0$ and $P(\mathbb{R})=\operatorname{id}$.
		\item For all $V_1, V_2\in\operatorname{Borel}(\mathbb{R})$, $P(V_1\cap V_2)=P(V_1)P(V_2)$.
		\item If $(V_n)_{n\in\mathbb{N}}$ is a pairwise disjoint sequence in $\operatorname{Borel}(\mathbb{R})$, we have that the equality  $P(\bigsqcup_{n\in\mathbb{N}}V_n)=\sum_{n\in\mathbb{N}}P(V_n)$ holds.
	\end{itemize}
\end{definition}
First, we can extract some basic properties:
\begin{itemize}
	\item If $V_1$ and $V_2$ are disjoint, $P(V_1)(H)$ and $P(V_2)(H)$ are orthogonal subspaces.
	\item For any $x,y\in H$, $V\rightarrow (P(V)x,y)$ is a $\mathbb{K}$-valued finite signed measure, which we'll note $\mu_{x,y}$. Furthermore, if $x=y$, that measure is positive real.
\end{itemize}

Then, $P$ naturally induces a form of Borel functional calculus:

\begin{definition}
	For borelian measurable $f:\mathbb{R}\rightarrow\mathbb{R}$, we define the (possibly unbounded) operator $\int_{\mathbb{R}} fdP$ on $H$ such that:
	\begin{itemize}
		\item $\operatorname{dom}(\int_{\mathbb{R}} fdP)=\{x\in H \;|\;\int_{\mathbb{R}}|f|^2d\mu_{x,x}\in\mathbb{R} \}$.
		\item For any $x\in \operatorname{dom}(\int_{\mathbb{R}} fdP)$ and $y\in H$, we have $(y,(\int_{\mathbb{R}} fdP) x)=\int_{\mathbb{R}}f d\mu_{y,x}$.
	\end{itemize}
\end{definition}
\begin{remark}
	It is widely known (and provable with elementary operator theory and measure theory) that there is a unique operator satisfying these conditions and that it is densely-defined and self-adjoint. Furthermore, $\|(\int_{\mathbb{R}}fdP)x\|^2\leq \int_{\mathbb{R}}|f|^2d\mu_{x,x}$ for all applicable $x$. Finally, for any borelian mesurable $f$, $g$, and for any $a\in\mathbb{R}$, we have $\int_{\mathbb{R}}f_1dP+a\int_{\mathbb{R}}f_2dP\leq\int_{\mathbb{R}}(f_1+af_2)dP$.  Details can be found in \cite[Section X.4]{conway2019course}.
\end{remark}

Then, as a corollary of Theorem~\ref{theorem_spectral_symmetric} (or even \ref{theorem_spectral_loebspace}), we have the spectral resolution theorem.
\begin{theorem}[Spectral resolution theorem]\label{spec_res}
	For any self-adjoint operator $A$ on $H$, there exists a projection-valued measure $P$ on $\operatorname{Borel}(\mathbb{R})$ such that 
	\begin{align*}
		A=\int_{\mathbb{R}}\operatorname{id} dP.
	\end{align*}
	In other words, $A$ admits a spectral resolution.
\end{theorem}

The proof that it follows from Theorem~\ref{theorem_spectral_symmetric} is outlined in Appendix \ref{appendix_pullback}. We note that often, this theorem is simply called the spectral theorem, as it is known to be equivalent to the multiplication operator version. We use this to create the groundwork for the sampling and scale.

We now fix $P$, the projection-valued measure on $\operatorname{Borel}(\mathbb{R})$ resolving $A$. Let $\{g_k'\}_{k\in\mathbb{N}}$ be a dense spanning set in $H$. We define the sequence $(g_k)_{k\in\mathbb{N}}$ and $(H_k)_{k\in\mathbb{N}}$ recursively, so that:

\begin{itemize}
	\item $g_k=g_k'-\operatorname{proj}_{\sum_{l=1}^{k-1}H_l}g_k'$.
	\item $H_k=\overline{\operatorname{span}(\{P([a,b[)g_k)\;|\;a,b\in\mathbb{Q} \})}$.
\end{itemize}

We note that $g_1=g_1'$, which may or may not hold for other $k\in\mathbb{N}$. Since $g_k\in H_l^\perp$ for any $l<k$, we have that for $a_1,a_2,b_1,b_2\in\mathbb{Q}$, $k_1>k_2\in\mathbb{N}$:

\begin{align*}
	(P([a_1,b_1[)g_{k_1},P([a_2,b_2[)g_{k_2})&=(g_{k_1},P([a_1,b_1[\cap [a_2,b_2[)g_{k_2})\\
	&=(g_{k_1},P([\max(a_1,a_2),\min(b_1,b_2)[)g_{k_2})=0.
\end{align*}

Thus, whenever $k_1\neq k_2$, $H_{k_1}\perp H_{k_2}$. Furthermore, $g_k=\lim_{n\rightarrow\infty}P([-n,n[)g_k\in H_k$, and so $g_k'\in\bigoplus_{l=1}^k H_l$. From that, we conclude that $H=\overline{\bigoplus_{k\in\mathbb{N}}H_k}$.

Finally, let $(e_j)_{j\in\mathbb{N}}$ be a counting of the set $\{ P([a,b[)g_k\;|\; a,b\in\mathbb{Q}, k\in\mathbb{N} \}\setminus\{0\}$. For all $j$ we have $e_j\neq0$ by definition, and $\overline{\operatorname{span}(\{e_j\}_{j\in\mathbb{N}})}\supset \bigoplus_{k\in\mathbb{N}}H_k$, and so $(e_j)_{j\in\mathbb{N}}$ is dense-spanning. We also pose, for $j\in\mathbb{N}$, $a_j,b_j\in\mathbb{Q}$ and $k_j\in\mathbb{N}$ such that $e_j=P([a_j,b_j[)g_{k_j}$.

Now, we construct a suitable sampling for $A$ and compatible scale. Let $N\in{}^*\mathbb{N}\setminus\mathbb{N}$. Let $\tilde{S}=\{[\frac{l}{N!},\frac{l+1}{N!}[\}_{l=-N!^2}^{N!^2-1}$ so that it partitions $[-N!,N![$. 

We note that for any standard $a,b\in\mathbb{Q}$, we have  ${}^*[a,b[=\sqcup_{N!a\leq l <N!b}[\frac{l}{N!},\frac{l+1}{N!}[$, and so $\{s\in\tilde{S}\;|\;s\subset {}^*[a,b[ \}$ partitions ${}^*[a,b[$. By overspill, let $K\in{}^*\mathbb{N}\setminus\mathbb{N}$ such that $\forall j\in [K]$,  $\{s\in\tilde{S}\;|\;s\subset [{}^*a_j,{}^*b_j[ \}$ partitions $[{}^*a_j,{}^*b_j[$.

For $k\in{}^*\mathbb{N}$, $s\in\tilde{S}$, let $\tilde{g}_k^s:={}^*P(s){}^*g_k\in{}^*H_k$. We note that if $(k_1,s_1)\neq(k_2,s_2)$ in ${}^*\mathbb{N}\times\tilde{S}$, then $(\tilde{g}_{k_1}^{s_1},\tilde{g}_{k_2}^{s_2})=({}^*g_{k_1},{}^*P(s_1\cap s_2){}^*g_{k_2})=0$. Then, we define:
\begin{align*}
	\tilde{\Psi}:&=\{(k,s)\in{}^*\mathbb{N}\times\tilde{S}\;|\;\exists j\in[K]\text{ s.t. } ({}^*e_j,\tilde{g}_k^s)\neq 0 \}\\
	\tilde{H}&={}^*\operatorname{span}(\{\tilde{g}_k^s\;|\; (k,s)\in\tilde{\Psi} \})\\
	\tilde{\Omega}&=\{f_k^s:=\frac{\tilde{g}_k^s}{\|\tilde{g}_k^s\|}\;|\;(k,s)\in\tilde{\Psi} \}\\
	\tilde{A}&:\tilde{H}\rightarrow\tilde{H}\\ \tilde{A}&\left(\sum_{(k,s)\in\tilde{\Psi}}a_k\tilde{g}_k^s\right) =\sum_{(k,s)\in\tilde{\Psi}}\frac{l_s}{N!}a_k\tilde{g}_k^s,
\end{align*}
where $l_s$ is the unique hyperinteger such that $s=[\frac{l_s}{N!},\frac{l_s+1}{N!}[$. 

Finally, for $j_0\in[K]$, let $$\tilde{c}_{j_0}=\frac{\frac{1}{2^{j_0}}\min(1,\frac{1}{\|{}^*e_{j_0}\|^2})}{\sum_{j\in[K]}\frac{1}{2^{j}}\min(1,\frac{1}{\|{}^*e_{j}\|^2})\|{}^*e_j\|^2 }>0.$$ We note, through convergence of the corresponding sum, that the denominator is infinitely close to $\sum_{j\in\mathbb{N}}\frac{1}{2^{j}}\min(1,\frac{1}{\|e_{j}\|^2})\|e_j\|^2\in\mathbb{R}_{>0}$. Therefore, for any infinite $L$, we have both $\sum_{j=L}^K \tilde{c}_j\|{}^*e_j\|^2$ and $\sum_{j=L}^{K} \tilde{c}_j$ are infinitesimal.

We can now proceed to show the following.

\begin{prop}
	$(\tilde{H},\tilde{A},\tilde{\Omega})$ is a sampling for $A$, for which $(({}^*e_j)_{j\in[K]}, (\tilde{c}_j)_{j\in[K]})$ is a compatible standard-biased scale.
\end{prop}
\begin{proof}
	We first note that $\tilde{\Psi}$ is ${}^*$finite, since $\tilde{\Psi}\subset [\max_{j\in[K]}(k_j)]\times\tilde{S}$. From that and the fact that all $\tilde{g}_k^s$ are orthogonal to each other, we can establish Properties (\ref{sampling_def_H}), (\ref{sampling_def_A}) and~(\ref{sampling_def_Omega}) of Definition~\ref{definition_sampling}.
	
	Furthermore, for $j_0\in\mathbb{N}$, ${}^*e_{j_0}\neq 0$ is standard in $H$ (therefore nearstandard), and we know  $\operatorname{st}(\tilde{c}_{j_0})=\frac{\frac{1}{2^{j_0}}\min(1,\frac{1}{\|e_{j_0}\|^2})}{\sum_{j\in\mathbb{N}}\frac{1}{2^{j}}\min(1,\frac{1}{\|e_{j}\|^2})\|e_j\|^2}>0$. Thus, we know that Property (\ref{scale_def_NS}) of Definition~\ref{definition_scale} holds.
	
	Furthermore, from the fact that $\sum_{j=L}^K \tilde{c}_j\|{}^*e_j\|^2$ is infinitesimal whenever $L$ is infinite, we directly have that both $\sum_{j\in[K]}\tilde{c}_j\|{}^*e_j\|^2$ and $\sum_{j\in\mathbb{N}}\operatorname{st}(\tilde{c_j})\|e_j\|^2$ are equal to $1$, establishing Properties (\ref{scale_def_proba}) and (\ref{scale_def_bias}). Since we already established Property (\ref{scale_def_dense}), we have that $(({}^*e_j)_{j\in[K]},(\tilde{c}_j)_{j\in[K]})$  is a standard biased scale.
	
	We have Property (\ref{scale_def_compat_H}).  Indeed, since ${}^*e_j\in{}^*H_{{}^*k_j}$ and $\{s\in\tilde{S}\;|\;s\subset [{}^*a_j,{}^*b_j[ \}$ partitions $[{}^*a_j,{}^*b_j[$, we know
	\begin{align*}
		{}^*e_j={}^*P([{}^*a_j, {}^*b_j [){}^*g_{k_j}=\sum_{s\in \tilde{S}\ \text{and}\  s\subset [{}^*a_j,{}^*b_j[}\tilde{g}_{k_j}^s.
	\end{align*}
	
	For any of the terms in the summation, we have $({}^*e_j,\tilde{g}_{k_j}^s)=\|\tilde{g}_{k_j}^s\|^2$. Therefore, for any non-zero $\tilde{g}_{k_j}^s$ in the sum, $(k_j,s)\in\tilde{\Psi}$, and so ${}^*e_j\in\tilde{H}$.
	
	Furthermore, Property (\ref{scale_def_compat_Omega}) holds by definition of $\tilde{\Psi}$.

	And so, what we have left to show to conclude the proof are Properties (\ref{sampling_def_approx}) and (\ref{scale_def_compat_A}), specifically that $G(A)\subset\operatorname{st}(G(\tilde{A}))$ and that for all $j\in \mathbb{N}$, $\tilde{A}{}^*e_j$ is nearstandard. We start with the latter first. 
	
	Let $x=e_j=P([a_j,b_j[)g_{k_j}$ for arbitrary $j\in\mathbb{N}$. We note that $\int_{\mathbb{R}}|\operatorname{id}|^2d\mu_{x,x}=\int_{[a_j,b_j[}|\operatorname{id}|^2d\mu_{x,x}\leq \max(|a_j|,|b_j|)^2\|x\|^2$. And so $x\in\operatorname{dom}(A)$. We show  $Ax=\operatorname{st}(\tilde{A}({}^*x))$. We have:
	\begin{align*}
		{}^*x&=\sum_{(k,s)\in\tilde{\Psi}}({}^*x,f_k^s)f_k^s=\sum_{(k,s)\in\tilde{\Psi}}\frac{1}{\|\tilde{g}_k^s \|^2}({}^*\left(P([a_j,b_j[)g_{k_j}\right),\tilde{g}_k^s)\tilde{g}_k^s=\sum_{\substack{(k,s)\in\tilde{\Psi}\\k=k_j\\s\subset{}^*[a_j,b_j[}}\tilde{g}_{k}^s.
	\end{align*}
	We remind that $(k_j,s)\in\tilde{\Psi}$ is equivalent to $\tilde{g}_{k_j}^s\neq 0$ whenever $s\in\tilde{S}$ and $s\subset{}^*[a_j,b_j[$. Therefore:
	\begin{align*}
		\tilde{A}({}^*x)&=\sum_{\substack{(k,s)\in\tilde{\Psi}\\k=k_j\\s\subset{}^*[a_j,b_j[}}\frac{l_s}{N!}\tilde{g}_{k}^s=\sum_{\substack{s\in\tilde{S}\\ s\subset{}^*[a_j,b_j[ }}\frac{l_s}{N!}\tilde{g}_{k_j}^s\\
		&=\sum_{N!a_j\leq l <N!b_j}\frac{l}{N!}{}^*P([\frac{l}{N!},\frac{l+1}{N!}[){}^*g_{k_j}\\
		&=\sum_{N!a_j\leq l <N!b_j}\frac{l}{N!}{}^*P([\frac{l}{N!},\frac{l+1}{N!}[){}^*x.
	\end{align*}
	
	For $m\in\mathbb{N}$, let $y_m=\sum_{m! a_j\leq l<m! b_j}\frac{l}{m!}P([\frac{l}{m!},\frac{l+1}{m!}[)x$. Then:
	
	\begin{align*}
		y_m&=\left(\int_{\mathbb{R}}\sum_{m!a_j\leq l<m!b_j}\frac{l}{m!}\mathbf{1}_{[\frac{l}{m!},\frac{l+1}{m!}[} dP\right)x
	\end{align*}
and so
	\begin{align*}
	 \|Ax-y_m\|^2&=\left\|\int_{\mathbb{R}}\left(\operatorname{id}-\sum_{m!a_j\leq l<m!b_j}\frac{l}{m!}\mathbf{1}_{[\frac{l}{m!},\frac{l+1}{m!}[}  \right)dP x\right\|^2\\
		&\leq \int_{[a_j,b_j[}\left|\operatorname{id}-\sum_{m!a_j\leq l<m!b_j}\frac{l}{m!}\mathbf{1}_{[\frac{l}{m!},\frac{l+1}{m!}[} \right|^2d\mu_{x,x}.
	\end{align*}
	
	We note that for $t\in[a_j,b_j[$,
	\begin{align*}
		\left|t-\sum_{m!a_j\leq l<m!b_j}\frac{l}{m!}\mathbf{1}_{[\frac{l}{m!},\frac{l+1}{m!}[}(t)\right|&=\left|t-\frac{\lfloor t m!\rfloor}{m!}\right|\leq\frac{1}{m!}\\
		\implies \|Ax-y_m\|^2&\leq\int_{[a_j,b_j[}\frac{1}{m!^2}d\mu_{x,x}=\frac{\|x\|^2}{m!^2}.
	\end{align*}  
	
	Therefore, $Ax=\lim_{m\rightarrow\infty} y_m$. And so, $\tilde{A}{}^*x={}^*y_N\simeq Ax$. Since ${}^*x={}^*e_j$ for an arbitrary $j\in\mathbb{N}$, Property (\ref{scale_def_compat_A}) holds. Also, $\tilde{A}(0)\simeq A(0)$ holds trivially, implying that for any $a,b\in\mathbb{Q}$ and $k\in\mathbb{N}$, $(P([a,b[)g_k, A\left(P([a,b[)g_k\right))\in\operatorname{st}(G(\tilde{A}))$.
	
	To conclude the proof, we use that $\operatorname{st}(G(\tilde{A}))$ is closed in $H\times H$ to show that it contains at least $G(A)$.
	
	Let $n\in\mathbb{N}$. For any $x\in P([-n,n[)(H)$, we have $\int_\mathbb{R}|\operatorname{id}|^2d\mu_{x,x}\leq n^2\|x\|^2$. And so, $A$ is bounded on $P([-n,n[)(H)$. Furthermore, since $(e_j)_{j\in\mathbb{N}}$ is dense-spanning in $H$, $(P([-n,n[) e_j)_{j\in\mathbb{N}}$ is dense-spanning in $P([-n,n[)(H)$. Thus, for any $x\in P([-n,n[)(H)$, $x\in\operatorname{dom}(A)$ and: 
	\begin{align*}
		(x,Ax)&\in\overline{\operatorname{span}(\{ (P([-n,n[) e_j, AP([-n,n[) e_j)\;|\;j\in\mathbb{N} \})}.\\
	\end{align*}
	Furthermore, noting $P([-n,n[) e_j=P([a,b [)g_{k_j}$, where $a=\max(-n,a_j)$ and $b=\min(n,b_j)$, we have
	\begin{align*}
		(x,Ax)&\in \operatorname{st}(G(\tilde{A})). 
	\end{align*} 
	
	Now, if $x\in\operatorname{dom}(A)$ is arbitrary, let $x_n=P([-n,n[)x$ for $n\in\mathbb{N}$. Then, we know $x=\lim_{n\rightarrow\infty}x_n$. Furthermore, if $z_n:=x-x_n$, $\mu_{z_n,z_n}$ is supported on $\mathbb{R}\setminus [-n,n[$, and $\mu_{z_n,z_n}(V)=\mu_{x,x}(V)$ for any $V\subset \mathbb{R}\setminus [-n,n[$. Therefore:
	
	\begin{align*}
		\|Ax-Ax_n\|^2&\leq \int_{\mathbb{R}}|\operatorname{id}|^2d\mu_{z_n,z_n}\\
		&=\int_{\mathbb{R}\setminus[-n,n[}|\operatorname{id}|^2d\mu_{x,x}.
	\end{align*}
	
	That last integral converges to $0$ as $n\rightarrow\infty$ by the monotonic convergence theorem, and so $(x,Ax)=\lim_{n\rightarrow\infty} (x_n,Ax_n)$ in $G(A)$. Therefore, $(x,Ax)\in\operatorname{st}(G(\tilde{A}))$. Since $x\in\operatorname{dom}(A)$ is arbitrary, we have Property (\ref{sampling_def_approx}) and this concludes the proof.
\end{proof}

We now want to show that the resulting $\hat{U}$ is surjective. In most cases, to do that, we need to explicit the description of the objects involved. Specifically, the strategy here is to find an explicit description of $\operatorname{st}\circ(\tilde{U}(\tilde{e}_j))$ on a set of (Loeb) probability $1$. To that end, we reuse the definitions for $\hat{\Omega}'$ and $\Omega_L'$ of \ref{definition_support_total_Uj}.

Let, for $a,b\in{}^*\mathbb{Q}$ and $k\in{}^*\mathbb{N}$:
\begin{align*}
	\tilde{\Omega}_{a,b,k}:&=\{f_k^s\in\tilde{\Omega}\;|\;s\subset [a,b[ \},\\
	\tilde{\chi}_{_{a,b,k}}:&=\mathbf{1}_{\tilde{\Omega}_{a,b,k}}\;\text{ on } \tilde{\Omega}.
\end{align*} 
Finally, for $j\in[K]$, let $\tilde{\chi}^{(j)}=\tilde{\chi}_{_{a_j,b_j,k_j}}$. Then, for any $f_k^s\in\tilde{\Omega}$ and $j\in[K]$:

\begin{align*}
	(f_k^s,{}^*e_j)&=\frac{1}{\|\tilde{g}_k^s\|}({}^*P(s){}^*g_k,{}^*P([{}^*a_j,{}^*b_j[){}^*g_{k_j})\\
	&=\begin{cases}
		\frac{1}{\|\tilde{g}_k^s\|}({}^*P(s){}^*g_k,{}^*P(s){}^*g_k) \qquad \text{if $k={}^*k_j$ and $s\subset [{}^*a_j,{}^*b_j[$}\\
		0\qquad  \text{otherwise}
	\end{cases}\\
	&=\|\tilde{g}_k^s\|\ \tilde{\chi}^{(j)}(f_k^s).
\end{align*}

We then remark that by definition of $\tilde{\Psi}$, for any $f_k^s\in\tilde{\Omega}$, there exists $j\in[K]$ for which $\tilde{\chi}^{(j)}(f_k^s)\neq 0$. Furthermore, we find, for any $f_k^s\in\tilde{\Omega}$:

\begin{align*}
	\tilde{\mu}(f_k^s)&=\sum_{j\in[K]}\tilde{c}_j|(f_k^s,{}^*e_j)|^2\\
	&=\|\tilde{g}_k^s\|^2\sum_{j\in[K]}\tilde{c}_j\tilde{\chi}^{(j)}(f_k^s).
\end{align*}
And so, for any $j_0\in[K]$:
\begin{align*}
	(\tilde{U}({}^*e_{j_0}))(f_k^s)=\frac{({}^*e_{j_0},f_k^s)}{\sqrt{\tilde{\mu}(f_k^s)}}=\frac{\tilde{\chi}^{(j_0)}(f_k^s)}{\sqrt{\sum_{j\in[K]}\tilde{c}_j\tilde{\chi}^{(j)}(f_k^s)}}.
\end{align*}

From Remark~\ref{remark_OmegaL_prime_from_U_L}, we have $$f\in\Omega_L'\implies \exists j\in\mathbb{N} \text{ s.t. } (\tilde{U}(\tilde{e}_j))(f)\neq 0\implies\exists j\in\mathbb{N} \text{ s.t. } f\in\tilde{\Omega}_{a_j, b_j, k_j}.$$ Conversely, since $0< \sum_{j\in[K]} \tilde{c}_j\tilde{\chi}^{(j)}(f)\leq\sum_{j\in[K]} \tilde{c}_j\simeq \sum_{j\in\mathbb{N}}c_j\in\mathbb{R}_{>0}$, we have that if $f\in\tilde{\Omega}_{a_j, b_j, k_j}$ for $j\in\mathbb{N}$, then $\operatorname{st}((\tilde{U}(\tilde{e}_j))(f))>0$, thus $f\in\Omega_L'$. We deduce $\Omega_L'= \bigcup_{j\in\mathbb{N}} \tilde{\Omega}_{a_j, b_j, k_j}$.

By Lemma~\ref{lemma_full_measure_Uj}, we know that $\mu_L(\Omega_L')=1$. Furthermore, if $f_k^s\in\Omega_L'$, we know $\sum_{j=L}^{K}\tilde{c}_j\tilde{\chi}^{(j)}(f_k^s)\leq\sum_{j=L}^{K}\tilde{c}_j$, which is infinitesimal for any infinite $L$. Thus, we have that  $\sum_{j\in[K]}\tilde{c}_j\tilde{\chi}^{(j)}(f_k^s)\simeq \sum_{j\in\mathbb{N}}c_j\tilde{\chi}^{(j)}(f_k^s)$.  Since there exists  $j\in\mathbb{N}$ for which $\tilde{\chi}^{(j)}(f_k^s)\neq 0$, we have, for any $j_0\in\mathbb{N}$ and $f_k^s\in\Omega_L'$:

\begin{align*}
	\operatorname{st}\left((\tilde{U}(e_{j_0}))(f_k^s)\right)=\frac{\tilde{\chi}^{(j_0)}(f_k^s)}{\sqrt{\sum_{j\in\mathbb{N}}c_j\tilde{\chi}^{(j)}(f_k^s)}}.
\end{align*}

From that we have $\tilde{\Omega}_{a_j, b_j, k_j}=\left(\operatorname{st}\circ (\tilde{U}(\tilde{e}_j))\right)^{-1}(\mathbb{K}\setminus\{0\})$. We now define $\hat{\Omega}_{a,b,k}=\hat{\nu}(\tilde{\Omega}_{a,b,k})$ for $a,b\in\mathbb{Q}$ and $k\in\mathbb{N}$. We can now show the following.

\begin{prop}
	$\{\hat{\Omega}_{a,b,k}\;|\;a,b\in\mathbb{Q}, k\in\mathbb{N} \}$ forms an open basis of the topology of $\hat{\Omega}'$ induced by $\hat{d}$.
\end{prop}
\begin{proof}
	First we show that $\hat{\Omega}_{a,b,k}$ is open for any $a,b\in\mathbb{Q}$ and $k\in\mathbb{N}$. 
	
	We note that for such $a,b$ and $k$,  $\tilde{\Omega}_{a,b,k}=\{f\in\tilde{\Omega}\;|\;({}^*(P([a,b[)g_k),f)>0 \}$. Thus, whenever $\hat{\Omega}_{a,b,k}$ is non-empty, we know $P([a,b[)g_k\neq 0$, and so there exists $j\in\mathbb{N}$ such that $P([a,b[)g_k=e_j=P([a_j,b_j[)g_{k_j}\implies \tilde{\Omega}_{a,b,k}=\tilde{\Omega}_{a_j,b_j,k_j}$. Then, since $\hat{\nu}$ is surjective:
	\begin{align*}
		\hat{\Omega}_{a,b,k}&=\hat{\nu}(\tilde{\Omega}_{a_j,b_j,k_j})=\hat{\nu}\left(\left(\operatorname{st}\circ(\tilde{U}(\tilde{e}_{j}))\right)^{-1}(\mathbb{K}\setminus\{0\})\right)\\
		&=\hat{\nu}\left(\hat{\nu}^{-1}(\hat{U}_j^{-1}(\mathbb{K}\setminus\{0\}))\right)=\hat{U}_j^{-1}(\mathbb{K}\setminus\{0\}).
	\end{align*}
	Here, $\hat{U}_j$ is defined per Remark~\ref{remark_definition_Uj}. Since $\hat{U}_j$ is continuous for any standard $j$, we know $\hat{\Omega}_{a,b,k}$ is open.
	
	Now, we show that for any $f_k^s\in\Omega_L'$ and any $\epsilon\in\mathbb{R}_{>0}$, there exists $a,b\in\mathbb{Q}$ such that $\hat{\nu}(f_k^s)\in\hat{\Omega}_{a,b,k}\subset \hat{B}_{\epsilon}(\hat{\nu}(f_k^s))$. First, let standard $\delta\in\mathbb{R}_{>0}$ such that for any $t\in\mathbb{R}$ for which $|t-\sum_{j\in\mathbb{N}}c_j\tilde{\chi}_j(f_k^s)|<\delta$, we get $t>0$ and $|\frac{1}{t}-\frac{1}{\sum_{j\in\mathbb{N}}c_j\tilde{\chi}_j(f_k^s)}|<\epsilon\frac{\min(1,\|e_1\|)}{2}$. Then, let $m\in\mathbb{N}$ such that:
	\begin{itemize}
		\item $\sum_{j\in\mathbb{N}\setminus[m]}c_j\|e_j\|^2<\frac{\epsilon}{4}$.
		\item $\sum_{j\in\mathbb{N}\setminus[m]}c_j<\delta$.
	\end{itemize}
	We note, that there must be $j\in[m]$ for which $\tilde{\chi}^{(j)}(f_k^s)=1$, as $\sum_{j\in\mathbb{N}}c_j\tilde{\chi}^{(j)}(f_k^s)\geq\delta$. Then, for any $f\in\Omega_L$ such that $\tilde{\chi}^{(j)}(f)=\tilde{\chi}^{(j)}(f_k^s)$ whenever $j\in[m]$, we have $f\in\Omega_L'$, and:
	
	\begin{align*}
		\hat{d}(\hat{\nu}(f),\hat{\nu}(f_k^s))&=\sum_{j\in\mathbb{N}}c_j^\frac{3}{2}\|e_j\|^2\left|\operatorname{st}\left((\tilde{U}(\tilde{e}_j))(f)\right)-\operatorname{st}\left((\tilde{U}(\tilde{e}_j))(f_k^s)\right) \right|\\
		&\leq\sum_{j\in[m]}c_j^\frac{3}{2}\|e_j\|^2\left|\frac{\tilde{\chi}^{(j)}(f)}{\sum_{l\in\mathbb{N}}c_l\tilde{\chi}^{(l)}(f)}-\frac{\tilde{\chi}^{(j)}(f_k^s)}{\sum_{l\in\mathbb{N}}c_l\tilde{\chi}^{(l)}(f_k^s)}\right|\\ 
		&\qquad +\sum_{j\in\mathbb{N}\setminus[m]}2c_j^\frac{3}{2}\|e_j\|^2\frac{1}{\sqrt{c_j}}\\
		&<\frac{\epsilon}{2}+ \left|\frac{1}{\sum_{j\in\mathbb{N}}c_j\tilde{\chi}^{(j)}(f)}-\frac{1}{\sum_{j\in\mathbb{N}}c_j\tilde{\chi}^{(j)}(f_k^s)}\right|\sum_{j\in[m]}c_j^{\frac{3}{2}}\|e_j\|^2\tilde{\chi}^{(j)}(f_k^s)\\
		&\leq\frac{\epsilon}{2}+\frac{\epsilon}{2}\cdot\min(1,\|e_1\|)\sum_{j\in[m]}c_j\|e_j\|^2\sqrt{c_j}<\epsilon.
	\end{align*}
	
	Of note, to show why the second-to-last inequality holds, we remark the inequalities $|\sum_{j\in\mathbb{N}}c_j\tilde{\chi}^{(j)}(f)-\sum_{j\in\mathbb{N}}c_j\tilde{\chi}^{(j)}(f_k^s)|\leq\sum_{j\in\mathbb{N}\setminus[m]}c_j<\delta$. For the last one, we remark $c_j\leq\frac{1}{\min(1,\|e_1\|^2)}$. Now, all we need to do is find $a,b\in\mathbb{Q}$ such that $f_k^s\in \tilde{\Omega}_{a,b,k}$ and $\forall f\in \tilde{\Omega}_{a,b,k}$, $\forall j\in[m]$, $\tilde{\chi}^{(j)}(f)=\tilde{\chi}^{(j)}(f_k^s)$.
	
	Let $I=\bigcap_{\substack{j\in[m]\\ k_j=k}}\left(\mathbf{1}_{\mathbb{R}\setminus[a_j,b_j[}+\tilde{\chi}^{(j)}(f_k^s)(\mathbf{1}_{[a_j,b_j[}-\mathbf{1}_{\mathbb{R}\setminus[a_j,b_j[})\right)^{-1}(\{1\})$, noting each term is $[a_j,b_j[$ or $\mathbb{R}\setminus[a_j,b_j[$ depending if $s\subset{}^*[a_j,b_j[$ or not. One can show inductively that $I$ is a finite union of intervals of the form $[a,b[$ where each endpoint is rational. Therefore, since $s\subset{}^*I$, there is $a,b\in\mathbb{Q}$ such that $[a,b[\subset I$ and $s\subset{}^*[a,b[$.
	
	Thus, $f_k^s\in\tilde{\Omega}_{a,b,k}$, and if $f_{k_2}^{s_2}\in\tilde{\Omega}_{a,b,k}$, then $k_2=k$ and $s_2\subset{}^*[a,b[$. Therefore, $s_2\subset{}^*I$, so for any $j\in [m]$ such that $k_j=k$, $s_2\subset{}^*[a_j,b_j[\iff s\subset{}^*[a_j,b_j[$. In other words,  $\forall j\in[m]$, $\tilde{\chi}^{(j)}(f_k^s)=\tilde{\chi}^{(j)}(f_{k_2}^{s_2})$, and so $\hat{d}(\hat{\nu}(f_{k_2}^{s_2}),\hat{\nu}(f_k^s))<\epsilon$.
	
	We conclude $\hat{\nu}(f_k^s)\in\hat{\Omega}_{a,b,k}=\hat{\nu}(\tilde{\Omega}_{a,b,k})\subset \hat{B}_{\epsilon}(\hat{\nu}(f_k^s))$. Since $f_k^s\in\Omega_L'$ and $\epsilon\in\mathbb{R}_{>0}$ are arbitrary, we have that $\{\hat{\Omega}_{a,b,k}\;|\;a,b\in\mathbb{Q} \text{ and } k\in\mathbb{N} \}$ generate the topology of  $(\hat{\Omega}',\hat{d})$.
\end{proof}

This also means that $\{\hat{\Omega}_{a,b,k}\;|\;a,b\in\mathbb{Q} \text{ and } k\in\mathbb{N} \}$ generate the Borel algebra $\hat{\mathcal{A}}$. And so, if $\hat{\chi}_{_{a,b,k}}:= \mathbf{1}_{\hat{\Omega}_{a,b,k}}$, we have $\{\hat{\chi}_{_{a,b,k}}\;|\;a,b\in\mathbb{Q} \text{ and } k\in\mathbb{N} \}$ is dense-spanning in $\hat{H}$. We note that $\hat{\chi}_{_{a,b,k}}\circ\hat{\nu}=\tilde{\chi}_{_{a,b,k}}$.

We can now proceed with the proof of Theorem~\ref{spectral_theorem}.

\begin{proof}[Proof of Theorem~\ref{spectral_theorem}]
	We show that $\hat{U}$ is surjective by building sufficiently many elements in $\hat{U}(H)$.
	
	Let $x\in H$, and let $a,b\in\mathbb{Q}$, $k\in\mathbb{N}$. Then, for any  $f_l^s\in\tilde{\Omega}$:
	\begin{align*}
		\left(\tilde{U}(P([a,b[)\operatorname{proj}_{H_k}x)\right)(f_l^s)&=\frac{(P([a,b[)\operatorname{proj}_{H_k}x,f_l^s)}{\sqrt{\tilde{\mu}(f_l^s)}}\\
		&=\mathbf{1}_{s\subset{}^*[a,b[}\mathbf{1}_{k=l}\frac{(x,f_l^s)}{\sqrt{\tilde{\mu}(f_l^s)}}\\
		&=\tilde{\chi}_{_{a,b,k}}(f_l^s)(\tilde{U}(x))(f_l^s).
	\end{align*}
	
	Thus, $\tilde{U}\left(P([a,b[)\operatorname{proj}_{H_k}x \right)=\tilde{\chi}_{_{a,b,k}}\cdot\tilde{U}(x)$. As all involved functions are $\mu_L$ almost-everywhere nearstandard, we have $U_L\left(P([a,b[)\operatorname{proj}_{H_k}x\right)=\tilde{\chi}_{_{a,b,k}}\cdot U_L(x)$. Therefore, since all involved elements have a pushforward through $\hat{\nu}$, we get
	
	\begin{align*}
		\hat{U}\left(P([a,b[)\operatorname{proj}_{H_k}x\right)&=\hat{\chi}_{_{a,b,k}}\cdot\hat{U}(x).
	\end{align*}

	Consequently, for any $x\in H$, $\hat{\chi}_{_{a,b,k}}\cdot\hat{U}(x)\in\hat{U}(H)$. To conclude the proof, we show that $1\in \hat{U}(H)$. It is indeed sufficient, as that implies $\hat{\chi}_{_{a,b,k}}\in \hat{U}(H)$ for any $a,b\in\mathbb{Q}$, $k\in\mathbb{N}$, and those span a dense subset of $\hat{H}$. 
	
	We have that for any $j\in\mathbb{N}$ and any Borel set $B$ of $\hat{\Omega}$, $\mathbf{1}_B\cdot\hat{U}_j\in\hat{U}(H)$. Indeed, $\hat{U}_j=\hat{U}(e_j)\in \hat{U}(H)$ is bounded, and $\mathbf{1}_B$ is an $L_2$-limit point of the set  $\operatorname{span}(\{\hat{\chi}_{_{a,b,k}}\;|\;a,b\in\mathbb{Q}\text{ and } k\in\mathbb{N}\})$. 
	
	By inspection, we know that for each $j\in\mathbb{N}$, $\hat{U}_j(\hat{\Omega})\subset [0,\frac{1}{c_j}]$. With this, we can consider $X=\sum_{j\in\mathbb{N}}\frac{c_j}{2^j}\hat{U}_j:\hat{\Omega}\rightarrow\mathbb{K}$. Then $X(\hat{\Omega})\subset [0,1]$, and $X(f)=0$ only holds for  $f\notin\hat{\Omega}'$, and so $X$ is almost-everywhere non-zero. We also have that for every Borel set $B\in\hat{\mathcal{A}}$, $\mathbf{1}_B\cdot X=\sum_{j\in\mathbb{N}}\frac{c_j}{2^j} \mathbf{1}_B\cdot \hat{U}_j\in \hat{U}(H)$. Then, for $n\in\mathbb{N}$, we can consider:
	\begin{align*}
		X_n=\sum_{k=1}^n\frac{n}{k}\mathbf{1}_{X^{-1}(]\frac{k-1}{n},\frac{k}{n}])}X. 
	\end{align*}
	For each $n\in\mathbb{N}$, we have $X_n\in \hat{U}(H)$. Furthermore, for every $t\in X^{-1}(]0,1])$, we have $X_n(t)=\frac{nX(t)}{\lceil nX(t)\rceil}$. Therefore, for such $t$:
	\begin{align*}
		|1-X_n(t)|=\frac{\lceil nX(t)\rceil-nX(t)}{\lceil nX(t)\rceil}\leq\frac{1}{\lceil nX(t)\rceil}\leq \frac{1}{nX(t)}.
	\end{align*}
	That implies that for every $t\in X^{-1}(]0,1])$, $1=\lim_{n\rightarrow\infty}X_n(t)$. We also have $|X_n|\leq 1$ everywhere. By the dominated convergence theorem, we get $1=\lim_{n\rightarrow\infty} X_n$ in $\hat{H}$. Thus, $1\in \hat{U}(H)$, implying $\hat{U}$ is surjective, as discussed previously.
\end{proof}

\begin{remark}
	Upon inspection, the induced space $(\hat{\Omega}',\hat{d})$ is homeomorphic to $(\mathbb{N}\times\mathbb{R},d)$, where $d(x,y)=\sum_{j\in\mathbb{N}}\frac{1}{2^j}\left|\mathbf{1}_{\{k_j\}\times[a_j,b_j[}(x)-\mathbf{1}_{\{k_j\}\times[a_j,b_j[}(y)\right|$. 
\end{remark}

We have proven the multiplicative version of the spectral theorem as well.

\begin{theorem}[Spectral theorem for self-adjoint operators]
	If $H$ is a separable Hilbert space on $\mathbb{K}$, and $A:\operatorname{dom}(A)\rightarrow H$ is a densely-defined self-adjoint operator on $H$, then there exists a probability space $(\Omega,\mathcal{A},\mu)$, a unitary operator $U:H\rightarrow L_2(\Omega,\mathcal{A},\mu)$ and a real multiplication operator $T$ on $L_2(\Omega,\mathcal{A},\mu)$ such that $U\circ A=T\circ U$.
\end{theorem}
\begin{remark}
	We note that $U\circ A=T\circ U$ holds (and not just $U\circ A\leq T\circ U$) as long as $A$ is self-adjoint, as shown in Proposition~\ref{sa_reduction} of the appendix.
\end{remark}

Also, we remark that our proof is not the quickest way of proving this theorem, going from the spectral resolution theorem. If that is the goal, it is quicker to use that $\bigoplus_{k\in\mathbb{N}}L_2(\mathbb{R},\mu_{g_k,g_k})\sim\bigoplus_{k\in\mathbb{N}}H_k$, whose completion is $H$. Theorem~\ref{spectral_theorem} shows not only that in Theorem~\ref{theorem_spectral_symmetric}, $U$ can be surjective  whenever $A$ is self-adjoint, it also shows that surjectivity of $U$ can be achieved by a sufficiently adequate choice of sampling and scale. 

As stated earlier, to ensure the right properties in the resulting space, we had to construct a sampling and a compatible scale that behaved well with the known properties of the operator, in this case one that made use of the spectral resolution. This ties to the next sections. For the rest of the paper, we work with specific examples of symmetric operators. With each example, we carefully craft a sampling and a scale that are well suited to the properties of the operator. In doing so, we find quite rich resulting hull spaces.
\section{Example: Shift operator}\label{section_shift}
The first such example of interest is with $\mathbb{K}=\mathbb{C}$, $H=l_2(\mathbb{Z})$, and $A=\frac{L+R}{2}$ where $L$ and $R$ are the left and right shift operators. $A$ is bounded and $\|A\|\leq 1$ since $L$ and $R$ are isometries. Furthermore, $L$ and $R$ are the adjoint of each other, so $A$ is self-adjoint. The goal is to construct a natural sampling and scale that allows for an explicit description of $\hat{H}$, $\hat{U}$ and $\hat{T}$.

\subsection{Constructing a sampling and a scale} As the title implies, here we first construct an explicit and suitable sampling for $A$ with a specific compatible scale.

First, we consider $\{g_l\}_{l\in\mathbb{Z}}$, the canonical Hilbert basis of $H$ . We will also enumerate it with $(e_j)_{j\in\mathbb{N}}$, where $(e_1,e_2,e_3,e_4,e_5,\dots)=(g_0,g_1,g_{-1},g_2,g_{-2},\dots)$.

Let $N\in{}^*\mathbb{N}\setminus\mathbb{N}$ be an infinite odd number, and let $M=\frac{N-1}{2}$. We consider $\tilde{H}={}^*\operatorname{span}(\{{}^*e_j\}_{j=1}^N)$. We note that  
\begin{align*}
	\{{}^*e_j\}_{j=1}^N=\{{}^*g_l\;|\;l\in[-M..M]\}.
\end{align*}

Here, $\tilde{A}$ will be defined to take full advantage of the flexibility offered by Definition~\ref{definition_sampling}. Let $\tilde{R}$ and $\tilde{L}:\tilde{H}\rightarrow \tilde{H}$ be the right-shift and left-shift operators modulo $N$, so that $\tilde{R}({}^*g_l)={}^*g_{l\oplus 1}$, $\tilde{L}({}^*g_l)={}^*g_{l\ominus1}$, where $\oplus$ and $\ominus$ are the addition and subtraction modulo $N$ (i.e. $l\oplus1\equiv l+1$ mod $N$ and $|l\oplus1|\leq M$, similarly for $l\ominus 1$). Then, let $\tilde{A}=\frac{\tilde{L}+\tilde{R}}{2}$.

Furthermore, let $\tilde{\Omega}=\{f_k\}_{k=1}^N$, with
\begin{align*}
	f_k=\frac{1}{\sqrt{N}}\sum_{l=-M}^{M} e^{-2\pi i\frac{kl}{N}}\ {}^*g_l.
\end{align*}
Finally, for $j\in[N]$, let $\tilde{c}_j=\frac{1}{2^j(1-2^{-N})}$. We now prove:

\begin{prop}
	$(\tilde{H}, \tilde{A}, \tilde{\Omega})$ constitutes a sampling for $A$, with which the pair formed by $(({}^*e_j)_{j\in N},(\tilde{c}_j)_{j\in N})$ is a compatible scale.
\end{prop}

\begin{proof}
	By definition, $\tilde{H}$ is an internal subspace of ${}^*H$ generated by a ${}^*$finite set, so Property (\ref{sampling_def_H}) of Definition~\ref{definition_sampling} holds. We have that $\tilde{R}$ and $\tilde{L}$ are internal linear isometries on $\tilde{H}$ and inverses of each other, and so they are the adjoint of each other. Thus, Property (\ref{sampling_def_A}) holds. 
	
	By definition of $\tilde{H}$, $f_k\in\tilde{H}$ for each $k\in[N]$. Using the internal definition principle, we get that $\tilde{\Omega}$ is an internal subset of $\tilde{H}$. Furthermore, since the set $\{{}^*g_l\;|\;-M\leq l\leq M \}$ is an orthonormal ${}^*$basis of $\tilde{H}$, we have, for any $k_1, k_2\in[N]$:
	
	\begin{align*}
		(f_{k_1},f_{k_2})&=\sum_{l=-M}^{M}\left(\frac{1}{\sqrt{N}}e^{-2\pi i\frac{k_1l}{N}}\right)\left(\frac{1}{\sqrt{N}}e^{2\pi i\frac{k_2l}{N}}\right)=\frac{1}{N}\sum_{l=-M}^{M}(e^{2\pi i\frac{k_2-k_1}{N}})^l\\
		&=\begin{cases}
			\frac{1}{N}\sum_{l=-M}^{M} 1 & \text{if } k_1=k_2\\
			\frac{1}{N}\frac{(e^{2\pi i\frac{k_2-k_1}{N}})^{M+1}-(e^{2\pi i\frac{k_2-k_1}{N}})^{-M}}{e^{2\pi i\frac{k_2-k_1}{N}}-1} & \text{otherwise }  
		\end{cases}\\
		&=\begin{cases}
			1 & \text{if } k_1=k_2\\
			0 & \text{otherwise}.
		\end{cases}
	\end{align*}
	
	Therefore, $\tilde{\Omega}$ is an internal orthonormal set of $\tilde{H}$. Since ${}^*|\tilde{\Omega}|=N={}^*\operatorname{dim}(\tilde{H})$, $\tilde{\Omega}$ is an internal orthonormal basis of $\tilde{H}$. Furthermore, we have
	
	\begin{align*}
		\tilde{R} f_k&=\sum_{l=-M}^M\frac{1}{\sqrt{N}}e^{-2\pi ik\frac{l}{N} }({}^*g_{l\oplus1})=e^{2\pi i\frac{k}{N}}\sum_{l=-M}^M\frac{1}{\sqrt{N}}e^{-2\pi ik\frac{l+1}{N} }({}^*g_{l\oplus1})\\
		&=e^{2\pi i\frac{k}{N}}\sum_{l=-M}^M\frac{1}{\sqrt{N}}e^{-2\pi ik\frac{l\oplus1}{N} }({}^*g_{l\oplus1})=e^{2\pi i\frac{k}{N}}\frac{1}{\sqrt{N}}\sum_{l=-M}^{M} e^{-2\pi i\frac{kl}{N}}({}^*g_l)\\
		&=e^{2\pi i\frac{k}{N}} f_k.
	\end{align*}
	The same way, we find $\tilde{L}f_k=e^{-2\pi i\frac{k}{N}}f_k$. Therefore we have
	\begin{align*}
		\tilde{A}f_k=\frac{e^{2\pi i\frac{k}{N}}+e^{-2\pi i\frac{k}{N}}}{2}f_k=\cos(2\pi k/N)f_k.
	\end{align*} 
	Thus, Property~(\ref{sampling_def_Omega}) holds. We now show Property~(\ref{sampling_def_approx}). 
	
	Let $x\in H$, for which we define $\tilde{x}:=\operatorname{proj}_{\tilde{H}} {}^*x=\sum_{l=-M}^{M}({}^*x,{}^*g_l){}^*g_l$. Since we know $x=\lim_{m\rightarrow\infty}\sum_{l=-m}^m (x,g_l) g_l$, we have $x=\operatorname{st}(\tilde{x})$. 
	
	Just like with $A$, we note $\|\tilde{A}\|\leq 1$.  Furthermore, we note that for any $l\in]-M..M[$, $\tilde{A}{}^*g_l={}^*A{}^*g_l=\frac{g_{l-1}+g_{l+1}}{2}$. Therefore:
	\begin{align*}
		\|\tilde{A} \tilde{x} -{}^*(Ax)\|&\leq \|(\tilde{A}-{}^*A)\tilde{x} \|+\|{}^*A(\tilde{x} -{}^*x)\|\\
		&\leq \|{}^*x-\tilde{x}\|+\|(\tilde{A}-{}^*A)\sum_{l=-M}^{M}({}^*x,{}^*g_l){}^*g_l\|\\ &=\|{}^*x-\tilde{x}\|+\|(\tilde{A}-{}^*A)(({}^*x,{}^*g_{-M}){}^*g_{-M}+({}^*x,{}^*g_{M}){}^*g_{M})\|\\
		&\leq\|{}^*x-\tilde{x}\|+2\|({}^*x,{}^*g_{-M}){}^*g_{-M}+({}^*x,{}^*g_{M}){}^*g_{M}\|\\
		&\leq \|{}^*x-\tilde{x}\|+2(|({}^*x,{}^*g_{-M})|+|({}^*x,{}^*g_{M})|).
	\end{align*}	
	Therefore, $\operatorname{st}(\|\tilde{A} \tilde{x} -{}^*(Ax)\|)=0$, and $(x,Ax)=\operatorname{st}(\tilde{x},\tilde{A}\tilde{x})$ in the graph norm. Since $x\in H$ is arbitrary, $G(A)\subset\operatorname{st}(G(\tilde{A}))$, thus Property~(\ref{sampling_def_approx}) holds and $(\tilde{H}, \tilde{A}, \tilde{\Omega})$ is a sampling for $A$.

	Furthermore, since $(e_j)_{j\in\mathbb{N}}$ is a Hilbert basis of $H$, and $\operatorname{st}(\tilde{c}_j)=\frac{1}{2^j}$ for each standard $j$, Properties (\ref{scale_def_NS}), (\ref{scale_def_dense}) directly hold, while Properties (\ref{scale_def_proba}) and (\ref{scale_def_bias}) are easily verified. By definition, $({}^*e_j)_{j\in[N]}$ is a generating set of $\tilde{H}$, so Properties (\ref{scale_def_compat_H}) and (\ref{scale_def_compat_Omega}) hold. Finally, as established previously, if $j\in\mathbb{N}$, $\tilde{A}\ {}^*e_j={}^*A\ {}^*e_j={}^*(Ae_j)$. Thus, $\tilde{A}\ {}^*e_j$ is standard, and Property~(\ref{scale_def_compat_A}) holds.
	
	Therefore, $(({}^*e_j)_{j\in[N]}, (\tilde{c}_j)_{j\in[N]})$ is a standard-biased scale compatible with the sampling $(\tilde{H},\tilde{A},\tilde{\Omega})$, concluding the proof.
\end{proof}

\begin{remark}
	We have also shown that for $f_k\in\tilde{\Omega}$, $\tilde{\lambda}_{f_k}=\cos(2\pi\frac{k}{N})$.
\end{remark}
\subsection{Nature of the hull space} 
Here, we show that $(\hat{\Omega},\hat{\mathcal{A}},\hat{\mu})$ is in fact the same measure space as $\mathbb{R}/\mathbb{Z}$ up to measure-preserving homeomorphism. 

Since $|(f_k,\tilde{e}_j)|^2=\frac{1}{N}$ is always true, we find that for internal $V\subset \tilde{\Omega}$:
\begin{align*}
	\tilde{\mu}(V)=\frac{1}{1-2^{-N}}\sum_{j=1}^N\frac{1}{2^j}\sum_{f\in V}\frac{1}{N}=\frac{{}^*|V|}{N},
\end{align*} 
where ${}^*|V|\in\mathbb{N}$ is the internal number of elements of $V$. And so $\tilde{\mu}$ is the uniform internal probability measure on $\tilde{\Omega}$. Furthermore, we find that for any standard $l\in\mathbb{Z}$,
\begin{align*}
	(\tilde{U}({}^*g_l))(f_k)=\frac{({}^*g_l,f_k)}{\sqrt{\tilde{\mu}(f_k)}}=\frac{e^{2\pi il\frac{k}{N}}/\sqrt{N}}{1/\sqrt{N}}=e^{2\pi il\frac{k}{N}}.
\end{align*}
Since the exponential is continuous, we find that for any $l\in\mathbb{Z}$ and $f_k\in\tilde{\Omega}$, $\operatorname{st}\left((\tilde{U}({}^*g_l))(f_k)\right)=e^{2\pi i l\operatorname{st}(\frac{k}{N})}$. We can use this to find that $f_{k_1}\sim f_{k_2}\iff e^{2\pi i\operatorname{st}(\frac{k_1}{N})}=e^{2\pi i\operatorname{st}(\frac{k_2}{N})}\iff\operatorname{st}\left(\frac{k_2-k_1}{N}\right)\in\mathbb{Z}\iff \operatorname{st}(\frac{k_1}{N})\equiv\operatorname{st}(\frac{k_2}{N})$ mod $1$.

Quite naturally, this induces $\rho:\hat{\Omega}\rightarrow \mathbb{R}/\mathbb{Z}$ with $\rho(\hat{\nu}(f_k)):=\operatorname{st}(\frac{k}{N})$ mod $1$.

\begin{prop}
	If we consider the probability space $(\mathbb{R}/\mathbb{Z},\operatorname{Borel}(\mathbb{R}/\mathbb{Z}),\mathcal{L})$, so that $\mathbb{R}/\mathbb{Z}$ is equipped with its usual quotient space topology and $\mathcal{L}$ is the Lebesgue measure, then $\rho$ is a measure preserving homeomorphism. 
\end{prop}

\begin{proof}

The previous equivalence chain shows $\rho$ is well-defined and injective. It is also surjective, since for any $t\in [0,1)$, there is $k\in[N]$ such that $\operatorname{st}(\frac{k}{N})=t$, implying $t$ mod $1=\rho(\hat{\nu}(f_k))$. Thus $\rho$ is bijective. We now show that $\rho$ is an homeomorphism.

In fact, since $\hat{\Omega}$ is compact and $\mathbb{R}/\mathbb{Z}$ is Hausdorff, it is sufficient to show that $\rho$ is continuous. First, let $\phi:\mathbb{R}/\mathbb{Z}\rightarrow\mathbb{S}^1$ be the classical homeomorphism with 
\begin{align*}
	\phi(t\text{ mod }1)=e^{2\pi i t}.
\end{align*}
If $\hat{U}_j$ is defined as per Remark~\ref{remark_definition_Uj} for $j\in\mathbb{N}$,  we have $\hat{d}$-continuous $\hat{U}_2$ with $\hat{U}_2\circ\hat{\nu}=\operatorname{st}\circ(\tilde{U}({}^*e_2))$ on $\Omega_L$. Therefore, for any $f_k\in\Omega_L$,
\begin{align*}
	\hat{U}_2(\hat{\nu}(f_k))&=\operatorname{st}((\tilde{U}({}^*e_2))(f_k))=\operatorname{st}((\tilde{U}({}^*g_1))(f_k))=e^{2\pi i\operatorname{st}(\frac{k}{N})}.
\end{align*}
Thus, $\hat{U}_2(\hat{\Omega})\subset\mathbb{S}^1$. Furthermore, we have
\begin{align*}
	\hat{U}_2(\hat{\nu}(f_k))&=\phi(\operatorname{st}(\frac{k}{N})\text{ mod }1)=\phi(\rho(\hat{\nu}(f_k))).
\end{align*}
Therefore,  $\hat{U}_2=\phi\circ\rho$, which implies $\rho=\phi^{-1}\circ \hat{U}_2$. Consequently, $\rho$ is continuous, and a homeomorphism. We get that $\rho$ is also an isomorphism between the $\sigma$-algebras $(\hat{\Omega},\hat{\mathcal{A}})$ and $(\mathbb{R}/\mathbb{Z},\operatorname{Borel}(\mathbb{R}/\mathbb{Z}))$. 

We now prove that $\rho$ is measure preserving, given Lebesgue measure $\mathcal{L}$. Let $a, b$ be two standard real numbers such that $0\leq a\leq b <1$. We consider the set $S=[a,b]$ mod $1$. We have that
\begin{align*}
	(\rho\circ\hat{\nu})^{-1}(S)&=\{f_k\in \tilde{\Omega}\;|\;a\leq\operatorname{st}(\frac{k}{N})\leq b \}\\
	&=\{f_k\in \tilde{\Omega}\;|\;\forall n\in\mathbb{N}\ a-\frac{1}{n}\leq\frac{k}{N}\leq b+\frac{1}{n} \}\\
	&=\bigcap_{n\in\mathbb{N}}\{f_k\in \tilde{\Omega}\;|\; a-\frac{1}{n}\leq\frac{k}{N}\leq b+\frac{1}{n} \}.
\end{align*}
And so,
\begin{align*}
	\hat{\mu}(\rho^{-1}\left(S\right))&=\mu_L((\rho\circ\hat{\nu})^{-1}(S))\\
	&=\lim_{n\rightarrow\infty}\operatorname{st}(\tilde{\mu}(\{f_k\in \tilde{\Omega}\;|\;\ a-\frac{1}{n}\leq\frac{k}{N}\leq b+\frac{1}{n} \}))\\
	&=\lim_{n\rightarrow\infty}\operatorname{st}(\frac{{}^*|\{f_k\in \tilde{\Omega}\;|\; a-\frac{1}{n}\leq\frac{k}{N}\leq b+\frac{1}{n} \}|}{N})\\
	&=\lim_{n\rightarrow\infty}\operatorname{st}(\frac{{}^*|\{k\in[N]\;|\; Na-\frac{N}{n}\leq k\leq Nb+\frac{N}{n} \}|}{N})\\
	&=\lim_{n\rightarrow\infty}\operatorname{st}(\frac{Nb+2\frac{N}{n}-Na+s_n}{N})=\lim_{n\rightarrow\infty}b-a+\frac{2}{n}=b-a\\
	&=\mathcal{L}(S).
\end{align*}
We note $s_n\in\{-1,0,1\}$ so the second-to-last equality holds. Since such $S$ generate the Borel algebra of $\mathbb{R}/\mathbb{Z}$, we find that $\rho$ is measure preserving, concluding the proof.
\end{proof}
\subsection{Relation to the Fourier series} All that is left is to calculate the form $\hat{U}$ and $\hat{T}$ takes in this space.

Thanks to the previous proposition, we have that $\rho$ is a measure space isomorphism between spaces  $(\hat{\Omega},\hat{\mathcal{A}},\hat{\mu})$ and $(\mathbb{R}/\mathbb{Z},\operatorname{Borel}(\mathbb{R}/\mathbb{Z}),\mathcal{L})$. This induces a unitary operator $\Theta:L_2(\mathbb{R}/\mathbb{Z}) \rightarrow L_2(\hat{\Omega})$.

$U=\Theta^{-1}\circ \hat{U}:H\rightarrow L_2(\mathbb{R}/\mathbb{Z})$ is also an isometry. For $l\in\mathbb{Z}$, and $t\in \mathbb{R}$, we have, almost everywhere:
\begin{align*}
	(U(g_l))(t\text{ mod } 1)&=(U_L(g_l))(f_k)=e^{2\pi i l\operatorname{st}(\frac{k}{N})}=e^{2\pi i l t}.
\end{align*}
Here, $f_k$ is such that $\rho(\hat{\nu}(f_k))=t\text{ mod }1$. Thus, $U(g_l)=\exp(2\pi i l\cdot)$. And so, quite interestingly, $U$ associates any sequence of $l_2(\mathbb{Z})$ to its corresponding Fourier series in $L_2(\mathbb{R}/\mathbb{Z})$. The following is then widely known (a proof can be found in \cite[Chapter~4]{rudin1987real}).

\begin{prop}
	Both $U$ and $\hat{U}$ are surjective.
\end{prop}

We find that $T=\Theta^{-1}\circ \hat{T}\circ \Theta$ is also a multiplication operator on $L_2(\mathbb{R}/\mathbb{Z})$ induced by $m=\hat{m}\circ\rho^{-1}$, and that $U\circ A=T\circ U$. We can then calculate that 
\begin{align*}
	m_L(f_k)=\cos(2\pi\operatorname{st}(\frac{k}{N}))=\cos(2\pi\rho(\hat{\nu}(f_k))),
\end{align*}
from which we conclude $m(t)=\cos(2\pi t)$. The fact that $T$ is unitarily equivalent to $A$ is far from surprising; it is widely known. What seems quite intriguing is that the presented process, given that sampling and scale, constructs so naturally and intrinsically this very equivalence.
\section{Example: Differential operator}\label{section_differential}

We consider one of the most important example operators studied in this theory, the differential operator. If $\mu$ is the Lebesgue measure on $\mathbb{R}$, let $H=L_2(\mathbb{R}, \mu)$ (with $\mathbb{K}=\mathbb{C}$), and $A=-i\frac{d}{dx}$ defined on $\operatorname{dom}(A)=C_c^{\infty}(\mathbb{R})$.  

First, we will prove a lemma that will be used multiple times throughout this section.

\begin{lemma}\label{lp-approx}
	Let $p\in\mathbb{R}_{\geq1}$, and let $f:\mathbb{R}\rightarrow\mathbb{C}$ such that:
	\begin{itemize}
		\item $f$ is continuous on $\mathbb{R}$.
		\item $f\in L_p(\mathbb{R})$. 
		\item $|f|$ is decreasing at infinity, in the sense that there exists $m\in\mathbb{N}$ such that for any $x,y\in\mathbb{R}$, $x>y\geq m$ or $x<y\leq -m$ implies $|f(x)|\leq |f(y)|$.
	\end{itemize}
	Furthermore, suppose that $\tilde{f}\in{}^*(L_p(\mathbb{R}))$ can be formulated by
	\begin{align*}
		\tilde{f}=\sum_{-L\leq k\leq M}{}^*f(c_k^{(r)})\mathbf{1}_{[\frac{k}{N},\frac{k+1}{N}[},
	\end{align*} 
	where $L,M,N\in{}^*\mathbb{N}\setminus\mathbb{N}$, $\frac{L}{N}$ and $\frac{M}{N}$ are both infinite, $r\in\mathbb{Z}$ and $c_k^{(r)}\in[\frac{k+r}{N},\frac{k+r+1}{N}]$ for each $k$. 
	
	Then, $\operatorname{st}(\tilde{f})=f$ in $L_p(\mathbb{R})$, and for any $K\in{}^*\mathbb{N}\setminus\mathbb{N}$, $\operatorname{st}(\|\mathbf{1}_{{}^*\mathbb{R}\setminus [-K,K[}\tilde{f}\|_p)=0$.
\end{lemma}

\begin{proof}
	We start with the last part. If $K\in{}^*\mathbb{N}\setminus\mathbb{N}$ and $X={}^*\mathbb{R}\setminus [-K,K[$, we have:
	\begin{align*}
		\|\mathbf{1}_X\tilde{f}\|_p^p&=\sum_{k=-L}^{-NK-1}\left\||{}^*f(c_k^{(r)})|\mathbf{1}_{[\frac{k}{N},\frac{k+1}{N}[}\right\|_p^p+\sum_{k=NK}^{M}\left\||{}^*f(c_k^{(r)})|\mathbf{1}_{[\frac{k}{N},\frac{k+1}{N}[}\right\|_p^p\\
		&=\sum_{k=-L}^{-NK-1}|{}^*f(c_k^{(r)})|^p\frac{1}{N}+\sum_{k=NK}^{M}|{}^*f(c_k^{(r)})|^p\frac{1}{N}\\
		&=\sum_{k=-L}^{-NK-1}\left\||{}^*f(c_k^{(r)})|\mathbf{1}_{_{[\frac{k+r+1}{N},\frac{k+r+2}{N}[}}\right\|_p^p+\sum_{k=NK}^{M}\left\||{}^*f(c_k^{(r)})|\mathbf{1}_{_{[\frac{k+r-1}{N},\frac{k+r}{N}[}}\right\|_p^p.
	\end{align*}
	For any $k$ in the left sum, any $x$ in $[\frac{k+r+1}{N},\frac{k+r+2}{N}[$ is infinite, negative, and greater than $c_k^{(r)}$. Therefore, $|{}^*f(c_k^{(r)})|\leq |{}^*f(x)|$ for any $x\in[\frac{k+r+1}{N},\frac{k+r+2}{N}[$. The same way, for any $k$ in the right sum, $|{}^*f(c_k^{(r)})|\leq |{}^*f(x)|$ for any $x\in[\frac{k+r-1}{N},\frac{k+r}{N}[$. We get
	\begin{align*}
		\|\mathbf{1}_{_{{}^*\mathbb{R}\setminus [-K,K[}}\tilde{f}\|_p^p&\leq \sum_{k=-L}^{-NK-1}\left\||{}^*f|\mathbf{1}_{[\frac{k+r+1}{N},\frac{k+r+2}{N}[}\right\|_p^p+\sum_{k=NK}^{M}\left\||{}^*f|\mathbf{1}_{[\frac{k+r-1}{N},\frac{k+r}{N}[}\right\|_p^p\\
		&\leq\|\mathbf{1}_{_{{}^*\mathbb{R}\setminus[-K+\frac{r}{N},K+\frac{r}{N}[}}{}^*f\|_p^p\simeq 0.
	\end{align*}
	
	We conclude $\operatorname{st}(\|\mathbf{1}_{{}^*\mathbb{R}\setminus [-K,K[}\tilde{f}\|_p)=0$. We now prove $\tilde{f}\simeq {}^*f$ in ${}^*(L_p(\mathbb{R}))$.
	
	Let $\epsilon\in\mathbb{R}_{>0}$. Using underspill, let $n\in\mathbb{N}$ such that $\|\mathbf{1}_{{}^*\mathbb{R}\setminus {}^*[-n,n[}\tilde{f}\|_p<\frac{\epsilon}{3}$ and $\|\mathbf{1}_{\mathbb{R}\setminus [-n,n[}f\|_p<\frac{\epsilon}{3}$. 
	
	For any $x\in{}^*[-n,n[$, we have $\tilde{f}(x)={}^*f(c_k^{(r)})$, where $k$ is such that $x\in[\frac{k}{N},\frac{k+1}{N}[$. Since $c_k^{(r)}$ is in $[\frac{k+r}{N},\frac{k+r+1}{N}[$, we have $|x-c_k^{(r)}|\leq\frac{|r|+1}{N}$, and so $x\simeq c_k^{(r)}$. Since $f$ is continuous, it is uniformly continuous on $[-n,n]$. Therefore, ${}^*f(x)\simeq {}^*f(c_k^{(r)})=\tilde{f}(x)$. Therefore, $\mathbf{1}_{{}^*[-n,n[}|{}^*f-\tilde{f}|\leq\frac{\epsilon}{3(2n)^{\frac{1}{p}}+1}\mathbf{1}_{{}^*[-n,n[}$. We then have:
	
	\begin{align*}
		\|{}^*f-\tilde{f}\|_p&\leq \|\mathbf{1}_{{}^*\mathbb{R}\setminus {}^*[-n,n[}(\tilde{f}-{}^*f)\|_p+\|\mathbf{1}_{{}^* [-n,n[}(\tilde{f}-{}^*f)\|_p\\
		&\leq\|\mathbf{1}_{{}^*\mathbb{R}\setminus {}^*[-n,n[}\tilde{f}\|_p+\|\mathbf{1}_{\mathbb{R}\setminus [-n,n[}f\|_p+\|\mathbf{1}_{{}^*[-n,n[}|\tilde{f}-{}^*f|\|_p\\
		&<\frac{2\epsilon}{3}+\frac{\epsilon}{3(2n)^{\frac{1}{p}}+1}(2n)^{\frac{1}{p}}<\epsilon.
	\end{align*}
	Since $\epsilon$ is arbitrary, we conclude that $\operatorname{st}(\tilde{f})=f$.
\end{proof}

\subsection{Constructing a sampling} Here, the idea behind $\tilde{A}$ will come from numerical approximations of the derivative. As with Section \ref{section_shift}, we will also try to endow $\tilde{H}$ and $\tilde{A}$ with sufficient symmetries.

Let $N\in{}^*\mathbb{N}\setminus\mathbb{N}$ such that $N=N_0!$ for some $N_0\in{}^*\mathbb{N}\setminus\mathbb{N}$. That way, for all $n\in\mathbb{N}, n$ divides $N$. We partition $[-N,N[$ with $\tilde{S}:=\{s_k:=[\frac{k}{N},\frac{k+1}{N}[ \;|\; k\in [-N^2.. N^2[ \}$. We define $\tilde{H}:=\operatorname{span}(\{\mathbf{1}_s \;|\; s\in\tilde{S} \})$.

Then, we use a technique that is similar to what was done in Section \ref{section_shift}. Let $\tilde{R}:\tilde{H}\rightarrow\tilde{H}$ be the right-shift operator modulo $2N^2$, so that it is linear and $\tilde{R}(\mathbf{1}_{s_k})=\mathbf{1}_{s_{k\oplus1}}$. The same way, we let $\tilde{L}:\tilde{H}\rightarrow\tilde{H}$ be the left-shift operator with $\tilde{L}(\mathbf{1}_{s_k})=\mathbf{1}_{s_{k\ominus1}}$. Then, let $\tilde{A}:=-i\frac{\tilde{L}-\tilde{R}}{2/N}$.

Furthermore let, for $k\in [-N^2..N^2[$, $f_k=\frac{1}{\sqrt{2N}}\sum_{l=-N^2}^{N^2-1}\exp(2\pi i\frac{kl}{2N^2})\mathbf{1}_{s_l}$, and $\tilde{\Omega}:=\{f_k\}_{k=-N^2}^{N^2-1}$. We now show:

\begin{prop}
	$(\tilde{H},\tilde{A},\tilde{\Omega})$ forms a sampling for $A$.
\end{prop}
\begin{proof}
	
By definition, $\tilde{H}$ is an internal subspace of ${}^*H$. Noting that $\{\mathbf{1}_s\;|\;s\in\tilde{S}\}$ is an orthogonal set, we have that ${}^*\operatorname{dim}(\tilde{H})=|\tilde{S}|=2N^2\in{}^*\mathbb{N}$, and Property~(\ref{sampling_def_H}) holds. We know both $\tilde{R}$ and $\tilde{L}$ are internal linear operators on $\tilde{H}$. Furthermore, we have, on $k,l\in[-N^2.. N^2[$, $(\mathbf{1}_{s_{k\oplus1}},\mathbf{1}_{s_l})=(\mathbf{1}_{s_k},\mathbf{1}_{s_{l\ominus1}})$. Therefore, the adjoint of $\tilde{R}$ is $\tilde{L}$, and vice versa. We conclude that $\tilde{A}$ is an internal symmetric operator on $\tilde{H}$, proving Property~(\ref{sampling_def_A}). 

Then, for $k_1,k_2\in[-N^2..N^2[$, we can evaluate:
\begin{align*}
	(f_{k_1},f_{k_2})&=\frac{1}{2N}\sum_{l=-N^2}^{N^2-1}\exp(2\pi i\frac{k_1l}{2N^2})\overline{\exp(2\pi i\frac{k_2l}{2N^2})}\frac{1}{N}\\
	&=\frac{1}{2N^2}\sum_{l=-N^2}^{N^2-1}\exp\left(2\pi i \frac{k_1-k_2}{2N^2}l\right).\\
\end{align*}
If $k_1=k_2$, we get 
\begin{align*}
	(f_{k_1},f_{k_2})&=\frac{1}{2N^2}\sum_{l=-N^2}^{N^2-1}1=1.
\end{align*}
If $k_1\neq k_2$, then $\frac{k_1-k_2}{2N^2}\notin{}^*\mathbb{Z}$, and we get
\begin{align*}
	(f_{k_1},f_{k_2})&=\frac{1}{2N^2}\frac{\exp\left(2\pi i \frac{k_1-k_2}{2N^2}N^2\right)-\exp\left(2\pi i \frac{k_1-k_2}{2N^2}(-N^2)\right)}{\exp\left(2\pi i \frac{k_1-k_2}{2N^2}\right)-1}\\
	&=\frac{1}{2N^2}\frac{\exp\left(\pi i(k_1-k_2)\right)-\exp\left(-\pi i (k_1-k_2)\right)}{\exp\left(2\pi i \frac{k_1-k_2}{2N^2}\right)-1}=0.
\end{align*}
Therefore, $\tilde{\Omega}$ is an internal orthonormal system. Furthermore, ${}^*|\tilde{\Omega}|=2N^2={}^*\operatorname{dim}(\tilde{H})$. Therefore, $\tilde{\Omega}$ is an orthonormal ${}^*$basis of $\tilde{H}$. We then evaluate:
\begin{align*}
	\tilde{R}(f_k)&=\frac{1}{\sqrt{2N}}\sum_{l=-N^2}^{N^2-1} \exp(2\pi i \frac{kl}{2N^2})\mathbf{1}_{s_{l\oplus1}}\\
	&=\frac{1}{\sqrt{2N}}\sum_{l=-N^2}^{N^2-1} \exp(2\pi i \frac{k(l\ominus1)}{2N^2})\mathbf{1}_{s_l}\\
	&=\frac{1}{\sqrt{2N}}\sum_{l=-N^2}^{N^2-1} \exp(2\pi i \frac{k(l-1)}{2N^2})\mathbf{1}_{s_l}\\
	&=\exp(2\pi i \frac{-k}{2N^2})f_k.
\end{align*}
The same way, $\tilde{L}(f_k)=\exp(2\pi i\frac{k}{2N^2})f_k$. We then get:
\begin{align*}
	\tilde{A}(f_k)=-i\frac{\exp(2\pi i \frac{k}{2N^2})-\exp(2\pi i\frac{-k}{2N^2})}{2/N}f_k=N\sin(\pi\frac{k}{N^2})f_k.
\end{align*} 
Therefore, $\tilde{\Omega}$ consists of eigenvectors of $\tilde{A}$, establishing Property~(\ref{sampling_def_Omega}). All that is left to show is Property~(\ref{sampling_def_approx}), stating $G(A)\subset \operatorname{st}(G(\tilde{A}))$.

Let $g\in\operatorname{dom}(A)=C_c^\infty (\mathbb{R})$, and let $\tilde{g}=\sum_{k=-N^2}^{N^2-1}{}^*g(\frac{k}{N})\mathbf{1}_{[\frac{k}{N},\frac{k+1}{N}[}\in\tilde{H}$. Since $g$ is compactly supported, $|g|$ is decreasing at infinity. Therefore, using $p=2$, $g$ respects all conditions of Lemma~\ref{lp-approx}, and so $\operatorname{st}(\tilde{g})=g$. Furthermore, since ${}^*g(x)=0$ for any infinite $x$, and since $k\oplus1=k+1$ and $k\ominus1=k-1$ for any limited $\frac{k}{N}$:
\begin{align*}
	\tilde{A}\tilde{g}&=\frac{-i}{2/N}\sum_{k=-N^2}^{N^2-1}{}^*g(\frac{k}{N})(\mathbf{1}_{s_{k\ominus1}}-\mathbf{1}_{s_{k\oplus1}})\\
	&=\frac{-i}{2/N}\sum_{k=-N^2}^{N^2-1}({}^*g(\frac{k\oplus1}{N})-{}^*g(\frac{k\ominus1}{N}))\mathbf{1}_{s_{k}}\\
	&=\frac{-i}{2/N}\sum_{k=-N^2}^{N^2-1}({}^*g(\frac{k+1}{N})-{}^*g(\frac{k-1}{N}))\mathbf{1}_{s_{k}}\\
	&=\frac{-i}{2}\sum_{k=-N^2}^{N^2-1}\frac{({}^*g(\frac{k+1}{N})-{}^*g(\frac{k}{N}))+({}^*g(\frac{k}{N})-{}^*g(\frac{k-1}{N}))}{1/N}\mathbf{1}_{[\frac{k}{N},\frac{k+1}{N}[}\\
	&=\frac{-i}{2}\sum_{k=-N^2}^{N^2-1}({}^*(g')(c_k)+{}^*(g')(c_{k-1}))\mathbf{1}_{[\frac{k}{N},\frac{k+1}{N}[}\\
	&=\frac{-i}{2}\left(\sum_{k=-N^2}^{N^2-1}{}^*(g')(c_k)\mathbf{1}_{[\frac{k}{N},\frac{k+1}{N}[}+\sum_{k=-N^2}^{N^2-1}{}^*(g')(c_{k-1})\mathbf{1}_{[\frac{k}{N},\frac{k+1}{N}[} \right).
\end{align*}
Where $c_k\in[\frac{k}{N},\frac{k+1}{N}[$ is given by transfer of the mean value property that $g$ has. Since $c_{k-1}\in[\frac{k-1}{N},\frac{k-1+1}{N}[$, and since $g'$ is also continuous and compactly supported, we can apply Lemma~\ref{lp-approx} to both sums and we get $\tilde{A}\tilde{g}\simeq\frac{-i}{2}(g'+g')=Ag$. Thus, $(g,Ag)\in\operatorname{st}(G(\tilde{A}))$. We conclude $G(A)\subset\operatorname{st}(G(\tilde{A}))$, ending the proof.
\end{proof}
\subsection{Constructing a scale} Here, we want to construct not only a scale that respects the criteria, but also have sufficiently well-behaved interactions with $\tilde{\Omega}$. To this end, we will have $\tilde{e}_j$ to be different shifts of the same $\tilde{e}$, so that we can have $(f_k,\tilde{e_j})$ be simply a unit multiplication of the same $(f_k,\tilde{e})$. For the choice of $\tilde{e}$, it will greatly simplify later calculations if:
\begin{itemize}
	\item $\operatorname{st}(\tilde{e})$ is an even function, has continuous square-integrable derivative and whose shifts span a dense subset.
	\item $\operatorname{st}\left(\frac{(f_k,\tilde{e})}{|(f_k,\tilde{e})|}\right)=1$ almost everywhere with respect to $\mu_L$. Intuitively, the expected criterion would be that $\operatorname{st}(\tilde{e})$ is itself in $L_1(\mathbb{R})$, with Fourier transform always positive.
\end{itemize}

With this in mind, we define $\tilde{e}\in\tilde{H}$ by:
\begin{align*}
	\tilde{e}=\left(\frac{2}{\pi}\right)^{\frac{1}{4}}\left(\sum_{k=-N_1}^{N_1} \exp(-(\frac{k}{N})^2)\mathbf{1}_{s_k}\right)+\frac{i}{N}\mathbf{1}_{s_0},
\end{align*}

where $N_1\in{}^*\mathbb{N}\setminus\mathbb{N}$, and for all $q\in\mathbb{Q}$, we have $\sqrt{N}<\frac{N_1}{N}+q<N$. For example, $\lfloor N^{\frac{5}{3}}\rfloor$ works.

Furthermore, let $e:=\left(\frac{2}{\pi}\right)^{\frac{1}{4}}\exp(-(\boldsymbol{\cdot})^2)$. We note that $e\in L_2(\mathbb{R})$ with $\|e\|=1$.

Let $(q_j)_{j\in\mathbb{N}}$ be a counting of $\mathbb{Q}$ starting with $q_1=0$. Since $N$ is divisible by any standard natural, we know for any $j\in\mathbb{N}$, $k_j:=Nq_j\in{}^*\mathbb{Z}$. Let $(k_j)_{j\in[N]}$ be an internal extension of this sequence in ${}^*\mathbb{Z}$. 

We then define, for $j\in[N]$, $\tilde{e}_j=\tilde{R}^{k_j}\tilde{e}$, noting that $\tilde{R}$ is an internal isometry on $\tilde{H}$ with inverse $\tilde{L}$. Furthermore, let $\tilde{c}_j:=\frac{1}{2^j\|\tilde{e}\|^2(1-2^{-N})}$, for $j\in[N]$. We now show:

\begin{prop}
	$((\tilde{e}_j)_{j\in[N]},(\tilde{c}_j)_{j\in[N]})$ forms a standard biased scale compatible with $(\tilde{H},\tilde{A},\tilde{\Omega})$.
\end{prop}
\begin{proof}
For this proof, let $\mathcal{T}_t:L_2(\mathbb{R})\rightarrow L_2(\mathbb{R})$ with $(\mathcal{T}_t(f))(x)=f(x-t)$ for any $t\in\mathbb{R}$.
	 	
First, since $\frac{i}{N}\mathbf{1}_{s_0}$ is infinitesimal, we note $\tilde{e}\simeq\sum_{k=-N_1}^{N_1}{}^*e(\frac{k}{N})\mathbf{1}_{[\frac{k}{N},\frac{k+1}{N}[}$. Furthermore $e$ is continuous, and $\forall x,y\in\mathbb{R}$, $|x|\leq|y|\implies |f(x)|\leq |f(y)|$. Therefore, by Lemma~\ref{lp-approx}, $e=\operatorname{st}(\tilde{e})$.

We note that for each $k\in[-N_1-1..N_1+1]$ and $j\in\mathbb{N}$, $-N^2\leq k+k_j\leq N^2-1$, since $-N<-(\frac{N_1}{N}-q_j)<\frac{N_1}{N}+q_j<N$. Therefore, $k\oplus k_j=k+k_j$, so $\tilde{e}_j={}^*\mathcal{T}_{q_j}(\tilde{e})$. Since $\mathcal{T}_{q_j}$ is an isometry of $L_2(\mathbb{R})$, we have $\operatorname{st}(\tilde{e}_j)=\mathcal{T}_{q_j}(e)\neq 0$. We also have $\operatorname{st}(\tilde{c}_j)=\frac{1}{2^j}>0$.

Through elementary means, we can show that $(\operatorname{st}(\tilde{e}_j))_{j\in\mathbb{N}}$ spans a dense subset. Since that fact is already known (or can be shown using Fourier transforms and convolutions), we will only give a broad idea as to how. To keep everything visual, in the following list $f(x)$ is written as the function $f:\mathbb{R}\rightarrow\mathbb{C}$. Let $V=\overline{\operatorname{span}(\{\operatorname{st}(\tilde{e}_j)\}_{j\in\mathbb{N}})}$. Then, with dominated convergence, we can show:

\begin{itemize}
	\item $\forall q\in\mathbb{Q}$, $e^{qx}e^{-x^2}\in V$ ($(\mathcal{T}_{q_j}(e))(x)=e^{-q_j^2}e^{2q_jx}e(x)$).
	\item $\forall t\in\mathbb{R}$, $e^{tx}e^{-x^2}\in V$, using $e^{r_n x}e^{-x^2}$, where $r_n\rightarrow t$.
	\item $\forall t\in\mathbb{R}$, $xe^{tx}e^{-x^2}\in V $, using $\frac{e^{(t+1/n)x}-e^{tx}}{1/n}e^{-x^2}$, and using the mean value theorem to dominate the sequence.
	\item $\forall n\in\mathbb{N}, t\in\mathbb{R}, x^ne^{tx}e^{-x^2}\in V$, using the same method applied by recurrence.
	\item $\forall t\in\mathbb{R}, e^{itx}e^{-x^2}\in V$, using $\sum_{k=0}^n\frac{(it)^k}{k!}x^ke^{-x^2}$.
\end{itemize}
For $k\in\mathbb{N}$, if $P_kf$ is the periodic extension of $f|_{[-k,k[}$, then we continue:
\begin{itemize}
	\item $\forall f\in C_c^{\infty}(\mathbb{R}),k\in\mathbb{N}$, $ \overline{\operatorname{supp}(f)}\subset]-k,k[\implies (P_kf)(x)e^{-x^2}\in V$, using $(\sum_{l=-n}^n \widehat{(P_kf)}_le^{\frac{2\pi i l}{2k}x})) e^{-x^2}$, noting the Fourier series converge uniformly ($P_kf$ is smooth).
	\item  $\forall f\in C_c^{\infty}(\mathbb{R}),f\in V$, using, for $]-n,n[\supset\overline{\operatorname{supp}(f)}$,   $(P_n(fe^{(\boldsymbol{\cdot}^2)}))(x)e^{-x^2}$
	\item $\forall f\in H, f\in V$, since $C_c^{\infty}(\mathbb{R})$ is dense.
\end{itemize}

Thus,  $(\operatorname{st}(\tilde{e}_j))_{j\in\mathbb{N}}$ spans a dense subset. Furthermore, we know $\|\tilde{e}_j\|=\|\tilde{e}\|$ since $\tilde{R}$ is an isometry. Therefore, $\sum_{j\in[N]}\tilde{c}_j\|\tilde{e}_j\|^2=\sum_{j\in[N]}\frac{1}{2^j(1-2^{-N})}=1$. We also have $\|\operatorname{st}(\tilde{e}_j)\|=\|e\|=1$ for any $j\in\mathbb{N}$, and so $\sum_{j\in\mathbb{N}}\operatorname{st}(\tilde{c}_j)\|\operatorname{st}(\tilde{e}_j)\|^2=\sum_{h\in\mathbb{N}}\frac{1}{2^j}=1$. 

We have shown Properties (\ref{scale_def_NS}) to (\ref{scale_def_bias}) of Definition~\ref{definition_scale}, and we conclude that the pair $((\tilde{e}_j)_{j\in[N]},(\tilde{c}_j)_{j\in[N]})$ is a standard-biased scale.

We now show it is compatible with $(\tilde{H},\tilde{A},\tilde{\Omega})$. By definition, $\tilde{e}_j=\tilde{R}^{k_j}\tilde{e}\in\tilde{H}$, so Property~(\ref{scale_def_compat_H}) holds. Furthermore, we have:
\begin{align*}
	\tilde{A}\tilde{e}&=\tilde{A}\left( \left(\sum_{k=-N_1}^{N_1} {}^*e(\frac{k}{N})\mathbf{1}_{s_k}\right)+\frac{i}{N}\mathbf{1}_{s_0}  \right)\\
	&=\frac{1}{2}(\mathbf{1}_{s_{-1}}-\mathbf{1}_{s_1})+\sum_{k=-N_1}^{N_1}{}^*e(\frac{k}{N})\frac{-i}{2/N}(\mathbf{1}_{s_{k\ominus1}}-\mathbf{1}_{s_{k\oplus1}})\\
	&\simeq \sum_{k=-N_1}^{N_1}{}^*e(\frac{k}{N})\frac{-i}{2/N}(\mathbf{1}_{s_{k-1}}-\mathbf{1}_{s_{k+1}})\\
	&=\frac{-i}{2/N}\left(\sum_{k=-N_1-1}^{N_1-1}{}^*e(\frac{k+1}{N})\mathbf{1}_{s_{k}}-\sum_{k=-N_1+1}^{N_1+1}{}^*e(\frac{k-1}{N})\mathbf{1}_{s_{k}} \right)\\
	&=\left(\frac{-i}{2}\sum_{k=-N_1+1}^{N_1-1}\frac{{}^*e(\frac{k+1}{N})-{}^*e(\frac{k-1}{N})}{1/N}\mathbf{1}_{s_{k}}\right)+w.
\end{align*}
Here, $w$ is a sum of the four terms left out. We note that for any standard integer $r$, $\frac{|\pm N_1+r|}{N}\geq \frac{N_1-|r|}{N}>\frac{N_1}{N}-1>\sqrt{N}$. And so $0\leq\frac{{}^*e(\frac{\pm N_1+r}{N})}{1/N}\leq N{}^*e(\sqrt{N})=\left( \frac{2}{\pi} \right)^{\frac{1}{4}} N\exp(-N)\simeq 0$. Therefore, $w$ is infinitesimal, and so we proceed with the same strategy as before:
\begin{align*}
	\tilde{A}\tilde{e}&\simeq \frac{-i}{2}\sum_{k=-N_1+1}^{N_1-1}\frac{{}^*e(\frac{k+1}{N})-{}^*e(\frac{k-1}{N})}{1/N}\mathbf{1}_{s_{k}}\\
	&=\frac{-i}{2}\sum_{k=-N_1+1}^{N_1-1}\left(\frac{{}^*e(\frac{k+1}{N})-{}^*e(\frac{k}{N})}{1/N}+\frac{{}^*e(\frac{k}{N})-{}^*e(\frac{k-1}{N})}{1/N} \right)\mathbf{1}_{s_{k}}\\
	&=\frac{-i}{2}\left(\sum_{k=-N_1+1}^{N_1-1}{}^*(e')(c_k)\mathbf{1}_{[\frac{k}{N},\frac{k+1}{N}[} + \sum_{k=-N_1+1}^{N_1-1}{}^*(e')(c_{k-1}) \mathbf{1}_{[\frac{k}{N},\frac{k+1}{N}[}\right).
\end{align*}

Again $c_k\in]\frac{k}{N},\frac{k+1}{N}[$ is given by the mean value theorem applied to $e$ transferred to ${}^*e$. We have $e'(t)=-2t\left( \frac{2}{\pi} \right)^{\frac{1}{4}}\exp(-t^2)$, and so $e'$ is continuous and in $L_2(\mathbb{R})$. From studying $e''$, we also find that $|e'|$ is decreasing at infinity. And so, we can apply Lemma~\ref{lp-approx} and conclude that $\operatorname{st}(\tilde{A}\tilde{e})=\frac{-i}{2}(e'+e')=-ie'$.

We note that $\tilde{R}$ and $\tilde{A}=-2\frac{\tilde{R}^{-1}-\tilde{R}}{2/N}$ commute. As before, for any $k$ in $[-N_1-1 .. N_1+1]$ and standard natural $j$, we have $k\oplus k_j=k+k_j$. Therefore, we get $\tilde{A}\tilde{e}_j=\tilde{R}^{k_j}\tilde{A}\tilde{e}={}^*\mathcal{T}_{q_j}(\tilde{A}\tilde{e})$ for any standard natural $j$. Thus $\tilde{A}\tilde{e}_j$ is nearstandard in $H$ with $\operatorname{st}(\tilde{A}\tilde{e}_j)=-i\mathcal{T}_{q_j}e'$, proving Property~(\ref{scale_def_compat_A}).

All that is left to verify is Property~(\ref{scale_def_compat_Omega}), requiring that for all $f_k\in\tilde{\Omega}$, there exists $j\in[N]$ such that $(f_k,\tilde{e}_j)\neq 0$. We note that since $q_1=0$, we have $k_1=0$ and $\tilde{e}_1=\tilde{e}$. And so, using that $e$ is an even real-valued function:

\begin{align*}
	(f_k,\tilde{e}_1)&=(f_k,\frac{i}{N}\mathbf{1}_{s_0})+\sum_{l=-N_1}^{N_1}{}^*e(\frac{l}{N})(f_k,\mathbf{1}_{s_l})\\
	&=\frac{1}{\sqrt{2N}}\left(\frac{-i}{N}\|\mathbf{1}_{s_0}\|^2+\sum_{l=-N_1}^{N_1}{}^*e(\frac{l}{N})\exp(2\pi i\frac{kl}{2N^2})\|\mathbf{1}_{s_l}\|^2 \right)\\
	&=\frac{1}{N\sqrt{2N}}\left(\frac{-i}{N}+{}^*e(0)+\sum_{l=1}^{N_1}{}^*e(\frac{l}{N})(\exp(2\pi i\frac{kl}{2N^2})+\exp(2\pi i\frac{-kl}{2N^2}))\right)\\
	&=\frac{1}{N\sqrt{2N}}\left(\frac{-i}{N}+e(0)+2\sum_{l=1}^{N_1}{}^*e(\frac{l}{N})\cos(2\pi\frac{kl}{2N^2})\right).
\end{align*}

Consequently, $\operatorname{Im}((f_k,\tilde{e}_1))=\frac{-1}{N^2\sqrt{2N}}\neq 0$. Therefore, for any $f_k\in\tilde{\Omega}$, there exists $j\in[N]$ such that $(f_k,\tilde{e}_j)\neq 0$. Finally, we conclude that the standard-biased scale $((\tilde{e}_j)_{j\in[N]},(\tilde{c}_j)_{j\in[N]} )$ is compatible with sampling $(\tilde{H},\tilde{A},\tilde{\Omega})$.
\end{proof}
\subsection{Characteristics of the Loeb space}
We start by noting for any $f_k\in\tilde{\Omega}$ and $j\in[N]$, $(f_k,\tilde{e}_j)=(\tilde{L}^{k_j}f_k,\tilde{e})=\exp(2\pi i \frac{k k_j}{2N^2})(f_k,\tilde{e})$.
Let $V\in\tilde{\mathcal{A}}$, i.e. an internal subset of $\tilde{\Omega}$. We have
\begin{align*}
	\tilde{\mu}(V)&=\sum_{j\in[N]}\tilde{c}_j\sum_{f_k\in V}|(\tilde{e}_j,f_k)|^2\\
	&=\sum_{j\in[N]}\tilde{c}_j\sum_{f_k\in V}|(\tilde{e},f_k)|^2\\
	&=\sum_{j\in[N]}\tilde{c}_j\|\operatorname{proj}_V\tilde{e}\|^2\\
	&=\frac{\|\operatorname{proj}_V\tilde{e}\|^2}{\|\tilde{e}\|^2}\sum_{j\in[N]}\tilde{c}_j\|\tilde{e}_j\|^2=\left(\frac{\|\operatorname{proj}_V\tilde{e}\|}{\|\tilde{e}\|}\right)^2.
\end{align*} 

Furthermore, for any $j\in[N]$ and $f_k\in\tilde{\Omega}$, we have
\begin{align*}
	(\tilde{U}(\tilde{e}_j))(f_k)=\frac{(\tilde{e}_j,f_k)}{\sqrt{\tilde{\mu}(f_k)}}=\exp(-2\pi i\frac{k_jk}{2N^2})\frac{(\tilde{e},f_k)}{|(\tilde{e},f_k)|/\|\tilde{e}\|}.
\end{align*}

Let $\Omega_{\mathbb{R}}=\{f_k\in\Omega_L\;|\; \frac{k}{N} \text{ is limited} \}$. Furthermore, for standard $a,b\in\mathbb{R}$, let $\tilde{\Omega}_a^b=\{f_k\in\tilde{\Omega} \;|\; a\leq \frac{k}{N}< b \}\in \tilde{\mathcal{A}}$. Also, let $g_0:\mathbb{R}\rightarrow\mathbb{R}_{>0}$ with 
\begin{align*}
	g_0(\omega)=\left(\frac{\pi}{2}\right)^{\frac{1}{2}}\exp\left(-\frac{\pi^2\omega^2}{2}\right).
\end{align*}

We show the following.

\begin{prop}
	 We have, for any $a, b\in\mathbb{R}$:
	\begin{align*}
		\mu_L(\tilde{\Omega}_a^b)=\int_{[a,b[}g_0d\mu.
	\end{align*}
	Furthermore, $\Omega_{\mathbb{R}}\in\mathcal{A}_L$ and $\mu_L(\Omega_{\mathbb{R}})=1$.
\end{prop} 
\begin{proof}
We have $\Omega_{\mathbb{R}}=\bigcup_{n\in\mathbb{N}}\tilde{\Omega}_{-n}^n$, thus $\Omega_{\mathbb{R}}\in\mathcal{A}_L$. Furthermore, for any $a<b\in\mathbb{R}$,

\begin{align*}
	\operatorname{st}(\tilde{\mu}(\tilde{\Omega}_a^b))&=\operatorname{st}(\|\operatorname{proj}_V\tilde{e}\|^2)\\
	&=\operatorname{st}\left(\sum_{f_k\in\tilde{\Omega}_a^b}|(\tilde{e},f_k)|^2\right)\\
	&=\operatorname{st}\left(\sum_{Na\leq k < Nb}\left|\frac{1}{N\sqrt{2N}}\left(\frac{i}{N}+\sum_{l=-N_1}^{N_1}{}^*e(\frac{l}{N})\exp(-2\pi i\frac{kl}{2N^2}) \right)\right|^2 \right)\\
	&=\operatorname{st}\left( \frac{1}{2N^3}\sum_{Na\leq k < Nb}\left(\frac{1}{N^2}+\left|\sum_{l=-N_1}^{N_1}{}^*e(\frac{l}{N})\exp(-2\pi i\frac{kl}{2N^2}) \right|^2\right) \right)\\
	&=\operatorname{st}\left(\frac{\lceil Nb\rceil-\lceil Na\rceil}{2N^5}+\frac{1}{2N}\sum_{Na\leq k < Nb}\left|\frac{1}{N}\sum_{l=-N_1}^{N_1}{}^*e(\frac{l}{N})\exp(-2\pi i\frac{kl}{2N^2}) \right|^2 \right)\\
	&=\frac{1}{2}\operatorname{st}\left(\frac{1}{N}\sum_{Na\leq k < Nb}\left|\frac{1}{N}\sum_{l=-N_1}^{N_1}{}^*e(\frac{l}{N})\exp(-2\pi i\frac{kl}{2N^2}) \right|^2 \right).
\end{align*}

We first look at the interior sum. For $\omega\in{}^*\mathbb{R}$, we define:

\begin{align*}
	\tilde{h}_{\omega}:&=\frac{1}{N}\sum_{l=-N_1}^{N_1}{}^*e(\frac{l}{N})\exp(-2\pi i\omega\frac{l}{2N})\\
	&={}^*\int_{{}^*\mathbb{R}}\left(\sum_{l=-N_1}^{N_1}{}^*e(\frac{l}{N})\exp(-2\pi i\omega\frac{l}{2N})\mathbf{1}_{s_l}\right)d({}^*\mu).
\end{align*}

We study the particular case of standard $\omega\in\mathbb{R}$. Here, we can use Lemma~\ref{lp-approx} with $p=1$, using the function $t\rightarrow e(t)\exp(-\pi i\omega t)$. Indeed, that function is continuous, integrable and $|e(t)\exp(-\pi i\omega t)|=e(t)$ is monotonic with respect to $|t|$, as we have seen before. Therefore, since $f\rightarrow\int_{\mathbb{R}}fd\mu$ is continuous on $L_1(\mathbb{R})$, we get, whenever $\omega$ is a standard real:

\begin{align*}
	\operatorname{st}(\tilde{h}_{\omega})=\int_{\mathbb{R}}e(t)\exp(-\pi i\omega t)d\mu(t)=\left(\frac{2}{\pi}\right)^{\frac{1}{4}}\int_{\mathbb{R}}e^{-\pi i \omega t}e^{-t^2}d\mu(t).
\end{align*}

We recognize the formula of the Fourier transform, which will be discussed more  in detail in the next subsections. Through complex analysis, we can calculate the integral, which results to, for standard $\omega$,
\begin{align*}
	\operatorname{st}(\tilde{h}_\omega)=(2\pi)^{\frac{1}{4}}\exp\left(-\frac{\pi^2\omega^2}{4}\right)= \sqrt{2g_0(\omega)}.
\end{align*}
We note, from inspection, that $\omega\rightarrow \operatorname{st}(\tilde{h}_\omega)$ is continuous.

Then, for $\omega_1$, $\omega_2\in{}^*\mathbb{R}$, we have:

\begin{align*}
	|\tilde{h}_{\omega_2}-\tilde{h}_{\omega_1}|&\leq \frac{1}{N}\sum_{l=-N_1}^{N_1}\left|{}^*e(\frac{l}{N})\exp(-2\pi i\omega_2\frac{l}{2N})-{}^*e(\frac{l}{N})\exp(-2\pi i\omega_1\frac{l}{2N}) \right|\\
	&\leq\frac{1}{N}\sum_{l=-N_1}^{N_1}2\pi\left|\frac{l}{2N}\right||\omega_2-\omega_1|{}^*e\left(\frac{l}{N}\right)\\
	&=\pi|\omega_2-\omega_1|{}^*\int_{{}^*\mathbb{R}}\left(\sum_{l=-N_1}^{N_1}\left|\frac{l}{N}\right|{}^*e(\frac{l}{N})\mathbf{1}_{s_l}\right)d({}^*\mu)\\
	&=\pi|\omega_2-\omega_1|\left\|\sum_{l=-N_1}^{N_1}\frac{l}{N}{}^*e(\frac{l}{N})\mathbf{1}_{s_l} \right\|_{L_1}.	
\end{align*}
We can use Lemma~\ref{lp-approx}, with $p=1$, on the function $t\rightarrow te(t)$, as it is continuous, in $L_1$ and decreasing at infinity. Thus $\sum_{l=-N_1}^{N_1}\frac{l}{N}{}^*e(\frac{l}{N})\mathbf{1}_{s_l}$ is nearstandard in $L_1$. From that, we conclude that there exists a standard real $C$ such that for any $\omega_1,\omega_2\in{}^*\mathbb{R}$, $|\tilde{h}_{\omega_2}-\tilde{h}_{\omega_1}|\leq C|\omega_2-\omega_1|$.

And so, whenever $\omega\in{}^*\mathbb{R}$ is limited, we have $\operatorname{st}(\tilde{h}_{\omega})=\operatorname{st}(\tilde{h}_{\operatorname{st}(\omega)})=\sqrt{2g_0(\operatorname{st}(\omega))}$.

We can now return to $\operatorname{st}(\tilde{\mu}(\tilde{\Omega}_a^b))=\frac{1}{2}\operatorname{st}\left(\frac{1}{N}\sum_{Na\leq k < Nb}|\tilde{h}_{\frac{k}{N}}|^2 \right)$. For any $k\in[Na,Nb[$, we have that 
\begin{align*}
	|\tilde{h}_{\frac{k}{N}}|^2&\simeq|(2\pi)^{\frac{1}{4}}\exp(-\frac{\pi^2(\operatorname{st}(\frac{k}{N}))^2}{4})|^2\\
	&\simeq |(2\pi)^{\frac{1}{4}}\exp(-\frac{\pi^2(\frac{k}{N})^2}{4})|^2\\
	&=\sqrt{2\pi}\exp(-\frac{\pi^2(\frac{k}{N})^2}{2}).
\end{align*}
Since $\frac{1}{N}\sum_{Na\leq k < Nb}1=\frac{\lceil Nb\rceil-\lceil Na\rceil}{N}<b-a+1$ is limited, we have, recognizing a Riemann sum:
\begin{align*}
	\operatorname{st}(\tilde{\mu}(\tilde{\Omega}_a^b))&=\frac{1}{2}\operatorname{st}\left(\frac{1}{N}\sum_{Na\leq k<Nb}\sqrt{2\pi}\exp(-\frac{\pi^2(\frac{k}{N})^2}{2})\right)\\
	&=\sqrt{\frac{\pi}{2}}\int_a^b\exp(-\frac{\pi^2\omega^2}{2})d\omega=\int_{[a,b]}g_0d\mu\\
	\implies \mu_L(\Omega_{\mathbb{R}})&=\lim_{n\rightarrow\infty}\operatorname{st}(\tilde{\mu}(\tilde{\Omega}_{-n}^n))=\sqrt{\frac{\pi}{2}}\int_{-\infty}^{\infty}\exp(-\frac{\pi^2\omega^2}{2})d\omega=1.
\end{align*}
\end{proof}
\begin{remark}
	In this proof, we have also shown  that for any $f_k\in\Omega_{\mathbb{R}}$ and limited hyperreal $\omega$,
		\begin{align*}
		 	\frac{1}{N}\sum_{l=-N_1}^{N_1}{}^*e(\frac{l}{N})\exp(-2\pi i\omega\frac{l}{2N})\simeq \sqrt{2g_0(\operatorname{st}(\omega))}.
		 \end{align*}
\end{remark}
Now, $U_L(e_j)$ simplifies a lot on $\Omega_{\mathbb{R}}$. If $f_k\in\Omega_{\mathbb{R}}$, we have:

\begin{align*}
	\sqrt{2N}(\tilde{e},f_k)&=\frac{i}{N^2}+\frac{1}{N}\sum_{l=-N_1}^{N_1}{}*e(\frac{l}{N})\exp(-2\pi i\frac{kl}{2N^2})\simeq \sqrt{2g_0(\operatorname{st}(\frac{k}{N}))}\in \mathbb{R}_{>0}
\end{align*}
and so
\begin{align*}
	 \frac{(\tilde{e},f_k)}{|(\tilde{e},f_k)|}&=\frac{\sqrt{2N}(\tilde{e},f_k)}{|\sqrt{2N}(\tilde{e},f_k)|}\simeq 1.
\end{align*}

Thus, for any $j\in\mathbb{N}$ and $f_k\in\Omega_{\mathbb{R}}$, we have:
\begin{align*}
	\operatorname{st}\left((\tilde{U}(\tilde{e}_j))(f_k)\right)=\operatorname{st}\left(\exp(-2\pi i\frac{k_j k}{2N^2})\frac{(\tilde{e},f_k)}{|(\tilde{e},f_k)|/\|\tilde{e}\|}\right)=\exp\left(2\pi i \frac{-q_j}{2}  \operatorname{st}(\frac{k}{N})\right).
\end{align*}

Therefore, for any $f_{k_1}, f_{k_2}\in \Omega_{\mathbb{R}}$, we have:

\begin{align*}
	\operatorname{st}(\tilde{d}(f_{k_1},f_{k_2}))=0&\iff \forall q\in\mathbb{Q},  \exp(2\pi i q\operatorname{st}(\frac{k_1}{N}))=\exp(2\pi i q\operatorname{st}(\frac{k_2}{N}))\\
	&\iff \forall n\in\mathbb{N},  \exp(2\pi i \frac{1}{n}\operatorname{st}(\frac{k_1-k_2}{N}))=1\\
	&\iff \forall n\in \mathbb{N}, \operatorname{st}(\frac{k_1-k_2}{N})\in n\mathbb{Z}\\
	&\iff \operatorname{st}(\frac{k_1-k_2}{N})=0\iff \operatorname{st}(\frac{k_1}{N})=\operatorname{st}(\frac{k_2}{N}).
\end{align*}

Also, we find $\tilde{\lambda}_{f_k}=N\sin(\pi\frac{k}{N^2})\simeq \pi\operatorname{st}(\frac{k}{N})$ for any such $f_k\in\Omega_{\mathbb{R}}$. And so, for almost all $f_k$ in $\Omega_L$, $m_L(f_k)= \pi\operatorname{st}(\frac{k}{N})$.

\subsection{Nature of the hull space}

For $a,b\in\mathbb{R}$, we define $\hat{\Omega}_a^b:=\hat{\nu}(\tilde{\Omega}_a^b)$. By a saturation argument, we know that since $\tilde{\Omega}_a^b$ is internal, $\hat{\Omega}_a^b$ is closed in $(\hat{\Omega},\hat{d})$, and so measurable. Furthermore, defining $\hat{\Omega}_{\mathbb{R}}:=\hat{\nu}(\Omega_{\mathbb{R}})=\cup_{n\in\mathbb{N}}\hat{\Omega}_{-n}^n$, we have that $\hat{\Omega}_{\mathbb{R}}\in\hat{\mathcal{A}}$  and $\hat{\mu}(\hat{\Omega}_{\mathbb{R}})=1$.

The previous results indicate a unique well-defined bijection $\phi:\mathbb{R}\rightarrow\hat{\Omega}_{\mathbb{R}}$ with $\phi(\operatorname{st}(\frac{k}{N}))=\hat{\nu}(f_k)$ for any $f_k\in\Omega_{\mathbb{R}}$. We show both $\phi$ and $\phi^{-1}$ are measurable.

We have, for $j\in\mathbb{N}$ and $t\in\mathbb{R}$, $(\hat{U}(e_j))(\phi(t))=\exp\left(2\pi i \frac{-q_j}{2}t\right)$. Therefore, for any $t_1,t_2\in\mathbb{R}$, we have:

\begin{align*}
	\hat{d}(\phi(t_1),\phi(t_2))&=\sum_{j\in\mathbb{N}}2^{-\frac{3j}{2}}\left|\exp(2\pi i\frac{-q_j}{2}t_1)-\exp(2\pi i\frac{-q_j}{2}t_2)\right|\\
	&=\sum_{j\in\mathbb{N}}2^{-\frac{3j}{2}}\left|\exp(2\pi i\frac{-q_j}{2}(t_1-t_2))-1\right|.
\end{align*}

By uniform convergence properties, we find that $\phi$ is continuous (usual distance in $\mathbb{R}$). Therefore, $\phi$ is measurable. Furthermore, if $F\subset\mathbb{R}$ is closed, then for any $n\in\mathbb{N}$, we have $F\cap [-n,n]$ is compact, and so $\phi(F\cap [-n,n])$ is compact in $\hat{\Omega}$. Therefore, $\phi(F)=\cup_{n\in\mathbb{N}}\phi(F\cap [-n,n])\in\hat{\mathcal{A}}$. We conclude $\phi^{-1}$ is measurable.

Next, we consider the pushforward measure on $\mathbb{R}$ $\mu':=(\phi^{-1})_*\hat{\mu}$. We know that $\mathcal{J}:L_2(\hat{\Omega},\hat{\mu})\rightarrow L_2(\mathbb{R},\mu')$ with $\mathcal{J}(f)=f\circ\phi$ is a unitary map. Let $U:=\mathcal{J}\circ\hat{U}:H\rightarrow  L_2(\mathbb{R},\mu')$.

We find that for any $a, b\in\mathbb{R}$, $\mu'([a,b])=\hat{\mu}(\phi([a,b]))=\hat{\mu}(\hat{\Omega}_a^b)=\mu_L(\tilde{\Omega}_a^b)=\int_{[a,b]}g_0d\mu$ . And so, $d\mu'=g_0d\mu$. Since $\int_{\mathbb{R}}|f|^2d\mu'=\int_{\mathbb{R}}|f|^2 g_0d\mu$, we have that $f\in L_2(\mathbb{R},\mu')\iff f\sqrt{g_0}\in L_2(\mathbb{R},\mu)$. We can now show the following:

\begin{prop}
	Both $U$ and $\hat{U}$ are surjective.
\end{prop}
  
\begin{proof}
We know that for any $j\in\mathbb{N}$, $(U(e_j))(\omega)=\exp(2\pi i\frac{-q_j}{2}\omega)$ $\mu'$-almost everywhere. Therefore, for $y\in U(H)^\perp\subset L_2(\mathbb{R},\mu')$, we have, for any $q\in\mathbb{Q}$:

\begin{align*}
	0&=\int_{\mathbb{R}}y(\omega)\exp(-2\pi i q \omega)d\mu'(\omega)\\
	&=\int_{\mathbb{R}} y(\omega)\exp(-2\pi i q \omega) g_0(\omega)d\mu(\omega)\\
	&=(y\sqrt{g_0},\exp(2\pi i q\boldsymbol{\cdot})\sqrt{g_0})_{L_2(\mathbb{R},\mu)}.
\end{align*}

And so, $y\sqrt{g_0} \in \overline{\operatorname{span}(\{\exp(2\pi i q(\boldsymbol{\cdot}))\exp(-\frac{\pi^2(\boldsymbol{\cdot})^2}{4})\}_{q\in\mathbb{Q}})}^{\perp}=\{0\}$. Indeed, we have seen earlier that functions of that form span a dense subset of $L_2(\mathbb{R},\mu)$. Since $g_0$ is nowhere $0$, we have $y=0$. Therefore, $U(H)=L_2(\mathbb{R},\mu')$. Thus, $U$ and $\hat{U}$ are surjective.
\end{proof}
\subsection{Relation to the Fourier transform}

We can say a few more interesting facts about this example. First, suppose $h\in L_2(\mathbb{R},\mu)\cap L_1(\mathbb{R},\mu)\subset H$. Then, for any $f_k\in \tilde{\Omega}$, we have
\begin{align*}
	(\tilde{U}({}^*h))(f_k)&=\frac{({}^*h,f_k)}{\sqrt{\tilde{\mu}(f_k)}}=\frac{({}^*h,\sqrt{2N}f_k)}{|\sqrt{2N}(\tilde{e},f_k)|}\\
	&= \frac{1}{|\sqrt{2N}(\tilde{e},f_k)|}({}^*h,\sum_{l=-N^2}^{N^2-1}\exp(\pi i \frac{k}{N}\frac{l}{N})\mathbf{1}_{s_l}).
\end{align*}

For a given standard $\omega\in\mathbb{R}$, if we define $(\psi_n)_{n\in\mathbb{N}}^{(\omega)}$ a sequence of simple functions given by $\psi_n^{(\omega)}=\sum_{l=-n^2}^{n^2-1}\exp(\pi i \omega\frac{l}{n})\mathbf{1}_{[\frac{l}{n},\frac{l+1}{n}[}$, then we know $\psi_n^{(\omega)}$ converges normally (i.e. uniformly on compacts) to $t\rightarrow\exp(\pi i \omega t)$, while being dominated by $1$. Therefore, by the dominated convergence theorem, the sequence $\overline{\psi_n^{(\omega)}} h$ converges to $t\rightarrow \exp(-\pi i \omega t)h(t)$ in $L_1(\mathbb{R})$. And so, $\lim_{n\rightarrow\infty}(h,\psi_n^{(\omega)})=\lim_{n\rightarrow\infty}\int_{\mathbb{R}}h\overline{\psi_n^{(\omega)}}d\mu=\int_{\mathbb{R}}\exp(-\pi i \omega t)h(t)d\mu(t)=(\mathcal{F}(h))(\frac{\omega}{2})$, where $\mathcal{F}$ is the Fourier transform.

Furthermore, using that $|h|$ is integrable, we can show that for any $f_k\in\Omega_{\mathbb{R}}$:
\begin{align*}
	({}^*h,\sum_{l=-N^2}^{N^2-1}\exp(\pi i \frac{k}{N}\frac{l}{N})\mathbf{1}_{s_l})&\simeq ({}^*h,\sum_{l=-N^2}^{N^2-1}\exp(\pi i \operatorname{st}(\frac{k}{N})\frac{l}{N})\mathbf{1}_{s_l})\\
	&=({}^*h,{}^*\psi_N^{(\operatorname{st}(\frac{k}{N}))})\simeq (\mathcal{F}(h))(\frac{1}{2}\operatorname{st}(\frac{k}{N})),
\end{align*}
and so,
\begin{align*}
	\operatorname{st}\left( (\tilde{U}({}^*h))(f_k)\right)&=\frac{(\mathcal{F}(h))(\frac{1}{2}\operatorname{st}(\frac{k}{N}))}{(\mathcal{F}(e))(\frac{1}{2}\operatorname{st}(\frac{k}{N}))}.
\end{align*}

We note $\sqrt{2g_0(\omega)}=(\mathcal{F}(e))(\frac{1}{2}\omega))$ for any real $\omega$. Therefore, we have that, almost everywhere, $(U(h))(\omega)=\frac{(\mathcal{F}(h))(\frac{\omega}{2})}{(\mathcal{F}(e))(\frac{\omega}{2})}$. This could be an effective definition of the Fourier transform $\mathcal{F}:L_2(\mathbb{R},\mu)\rightarrow L_2(\mathbb{R},\mu)$:
\begin{align*}
	(\mathcal{F}(h))(\omega)=(U(h))(2\omega)\cdot \sqrt{2g_0(2\omega)}=(2\pi)^{\frac{1}{4}}(U(h))(2\omega)\exp(-\pi^2\omega^2).
\end{align*}

This can be used directly to recover Plancherel's identity for such $h$:
\begin{align*}
	\int_{\mathbb{R}} |\mathcal{F}(h)|^2d\mu &=\frac{1}{2}\int_{\mathbb{R}} |(\mathcal{F}(h))(\frac{\omega}{2})|^2d\mu(\omega)\\
	&=\int_{\mathbb{R}}\frac{\left|(\mathcal{F}(h))(\frac{\omega}{2})\right|}{2g_0(\omega)}^2d\mu'(\omega)\\
	&=\|U(h)\|^2=\|h\|^2=\int_{\mathbb{R}}|h|^2d\mu.	
\end{align*}

We can also recover the differentiation formula using $U$. We have that for $f_k\in\Omega_{\mathbb{R}}$, $m_L(f_k)=\pi\operatorname{st}(\frac{k}{N})$, and so for any $h\in\operatorname{dom}(A)$ and for almost every $\omega\in\mathbb{R}$, $(U(Ah))(\omega)=\pi \omega (U(h))(\omega)$. Since for such $h\in\operatorname{dom}(A)$ both $h$ and $h'$ are in $L_1(\mathbb{R},\mu)\cap L_2(\mathbb{R},\mu)$ and $h'=iAh$, we conclude $(\mathcal{F}(h'))(\omega)=2\pi i \omega(\mathcal{F}(h))(\omega)$. That formula then works for every $h\in\operatorname{dom}(\overline{A})=H^1(\mathbb{R})$, the Sobolev space, as any sampling for $A$ is a sampling for $\overline{A}$.

As with Section \ref{section_shift}, the surprising part is not that the Fourier transform works, as it is widely known to be the natural unitary equivalence when it comes to $\frac{d}{dx}$. What seems quite interesting is that
it appears here, using a general method with a chosen sampling and scale. This appearance may be even more peculiar here compared to the previous section, because the method of this paper always ends with a probability space, while $(\mathbb{R},\operatorname{Borel}(\mathbb{R}),\mu)$ is not one.

Given the examples of Sections \ref{section_sa}, \ref{section_shift} and \ref{section_differential}, it would seem that this method can generate an explicit, natural unitary equivalence depending on how the chosen sampling and scale interact, and if those interactions are themselves explicit. It could be interesting to consider other examples, or try to generalize beyond symmetric operators.
\appendix
\section{Induced spectral resolution}\label{appendix_pullback}

The goal here is to use Theorem~\ref{theorem_spectral_symmetric} to prove the spectral resolution theorem:

\begin{theorem}[Spectral resolution theorem]
	For any self-adjoint operator $A$ on $H$, there exists a projection-valued measure $P$ on $\operatorname{Borel}(\mathbb{R})$ such that 
	\begin{align*}
		A=\int_{\mathbb{R}}\operatorname{id} dP.
	\end{align*}
\end{theorem}
\begin{remark}
	The proof outlined here will resemble part of the work done in \cite{GOLDBRING2021590} (and to an extent \cite{bernstein_spectral} and \cite{mooreboundedops}), whereas a suitable spectral measure is pulled from a bigger space through a natural embedding of $H$. In those works, the bigger space was the nonstandard hull, while here it will the one provided by Theorem~\ref{theorem_spectral_symmetric}. The idea of how will also be similar, but some of the details can differ, notably due to the fact that the Hilbert space might be real, excluding the direct use of operators such as $T-tiI$.
	
	Finally, we remark that no nonstandard analysis will be used in this section, and ${}^*$ will be used for the adjoint of a map.
\end{remark}

Here, we will assume that $A$ is a densely-defined self-adjoint operator on $H$. Furthermore, $(\Omega,\mathcal{A},\mu)$ is a probability space, $m:\Omega\rightarrow\mathbb{R}$ is a measurable function inducing self-adjoint multiplication operator $T$, and $U:H\rightarrow L_2(\Omega,\mu)$ is an isometry such that $U\circ A \leq T\circ U$. The one from Theorem~\ref{theorem_spectral_symmetric} works, or even the one from Theorem~\ref{theorem_spectral_loebspace}. We also note, for measurable $f:\Omega\rightarrow\mathbb{K}$, the induced multiplication operator $T_f$.

We start with the following proposition.
\begin{prop}\label{sa_reduction}
	We have that $U(H)$ reduces for $T$. Furthermore, $U\circ A=T\circ U$. 
\end{prop}	
\begin{proof}
	We first show that $U(H)$ reduces for $T$. By definition, we must show $\operatorname{proj}_{U(H)}\circ T\leq T\circ \operatorname{proj}_{U(H)}$. Let $x\in\operatorname{dom}(T)$. We first note $U^*x\in\operatorname{dom}(A)$. Indeed, for $y\in\operatorname{dom}(A)$, we have:
	\begin{align*}
		(Ay,U^*x)&=(U(Ay),x)=(T(Uy),x)=(Uy,Tx)=(y,U^*(Tx)).
	\end{align*}
	Thus, $U^*x\in\operatorname{dom}(A^*)=\operatorname{dom}(A)$, and $A(U^*x)=A^*(U^*x)=U^*(Tx)$. We note that $(UU^*y,Uz)=(y,Uz)$ holds for any $y\in L_2(\Omega)$ and $z\in H$, and so $\operatorname{proj}_{U(H)}=U U^*$. Using $U\circ A\leq T\circ U$, we get that $\operatorname{proj}_{U(H)}x\in\operatorname{dom}(T)$, and $T(\operatorname{proj}_{U(H)}x)=U(A U^*x)=UU^*(Tx)=\operatorname{proj}_{U(H)}Tx$. Since $x\in\operatorname{dom}(T)$ is arbitrary, we have $\operatorname{proj}_{U(H)}\circ T\leq T\circ\operatorname{proj}_{U(H)} $, and so $U(H)$ reduces for $T$.
	
	To prove $U\circ A= T\circ U$, it is sufficient to show $U^{-1}(\operatorname{dom}(T))\subset\operatorname{dom}(U\circ A)$. But for $x\in H$ such that $U(x)\in\operatorname{dom}(T)$ and $y\in\operatorname{dom}(A)$, we have
	\begin{align*}
		|(Ay,x)|&=|(UAy,Ux)|=|(TUy,Ux)|=|(Uy,TUx)|\leq\|y\|\|TUx\|.
	\end{align*}
	
	Therefore, $y\rightarrow (Ay,x)$ is a bounded map on $\operatorname{dom}(A)$, and since $A$ is self-adjoint, $x\in\operatorname{dom}(A^*)=\operatorname{dom}(U\circ A)$. Since $x\in U^{-1}(\operatorname{dom}(T))$ is arbitrary, we conclude $U\circ A= T\circ U$.
\end{proof}

We now consider the projection-valued measure $\hat{P}$ on $L_2(\Omega)$ given by $\hat{P}(V)=T_{\mathbf{1}_{m^{-1}(V)}}$ for any $V\in\operatorname{Borel}(\mathbb{R})$, so that $\hat{P}(V) f=\mathbf{1}_{m^{-1}(V)}\cdot f$ for any $f\in L_2(\Omega)$. It is straightforward to verify that $\hat{P}$ forms a projection-valued measure. Furthermore, for any $f\in L_2(\Omega)$ and $V\in\operatorname{Borel}(\mathbb{R})$:
\begin{align*}
	\int_{\mathbb{R}}\mathbf{1}_Vd\hat{\mu}_{f,f}&=\hat{\mu}_{f,f}(V)=(\hat{P}(V)f,f)=\int_{\Omega}\mathbf{1}_{m^{-1}(V)}|f|^2d\mu=\int_{\Omega}(\mathbf{1}_V\circ m)|f|^2d\mu\\
	\implies& \int_\mathbb{R}|\operatorname{id}|^2d\hat{\mu}_{f,f}=\int_{\Omega}(|\operatorname{id}|^2\circ m)|f|^2d\mu=\int_{\Omega} |mf|^2d\mu.
\end{align*}
And so, $\operatorname{dom}(\int_\mathbb{R}\operatorname{id}d\hat{P})=\operatorname{dom}(T)$. Furthermore, for any $f\in\operatorname{dom}(T)$:
\begin{align*}
	(\int_{\mathbb{R}}\operatorname{id}d\hat{P} f,f)=\int_{\mathbb{R}}\operatorname{id}d\hat{\mu}_{f,f}=\int_{\Omega}(\operatorname{id}\circ m)|f|^2d\mu=(m\cdot f,f)=(Tf,f).
\end{align*}
Therefore, since both maps are symmetric, $\int_{\mathbb{R}}\operatorname{id}d\hat{P}=T$, and so $\hat{P}$ forms a spectral resolution for $T$. 

We now show that $\hat{P}$ can be pulled back through $U$.

\begin{prop}
	For any $V\in\operatorname{Borel}(\mathbb{R})$ and $x\in U(H)$, $\hat{P}(V)x\in U(H)$.
\end{prop}
\begin{proof}
	It is sufficient to show the property for intervals of the form $[a,b]$ for $a,b\in\mathbb{R}$. Indeed, if that can be shown,  then for any $f\in U(H)$, $g\in U(H)^\perp$, we have $\hat{\mu}_{f,g}([a,b])=0$ for any real $a$ and $b$. Since such intervals generate $\operatorname{Borel}(\mathbb{R})$, and the measure is finite, we get, for any Borel set $V$, $\hat{\mu}_{f,g}(V)=0$, and so $\hat{P}(V)f\in(U(H)^\perp)^\perp=U(H)$.
	
	So, let $f\in U(H)$ and $a,b\in\mathbb{R}$. For $k\in\mathbb{N}$, let $f_k'=\mathbf{1}_{m^{-1}([-k,k])}f$ and $f_k=\operatorname{proj}_{U(H)}f_k'$. Then, for any $n\in\mathbb{N}$:
	\begin{align*}
		\int_{\Omega}|m^n\cdot f_k'|^2d\mu&=\int_{m^{-1}([-k,k])}|m|^{2n} |f|^2d\mu\leq k^{2n}\|f\|^2.
	\end{align*}
	
	Therefore, inductively, for any natural $n$, we have $f_k'\in\operatorname{dom}(T^n)$. Furthermore, still using induction, we can show $U(H)$ reduces for $T^n$ as well. Thus, for any $n\in\mathbb{N}$, $f_k\in\operatorname{dom}(T^n)$ and $m^n\cdot f_k=T^n f_k=\operatorname{proj}_{U(H)}T^n f_k'\in U(H)$. Using linearity, we have, for any real polynomial $p$, $(p\circ m)\cdot f_k=\operatorname{proj}_{U(H)}(p\circ m\cdot f_k')\in U(H)$.
	
	Using continuous functions and Stone-Weierstrass theorem as intermediaries, we can find a sequence $(p_n)_{n\in\mathbb{N}}$ of real polynomials such that $|p_n(x)|\leq 2$ for every $x\in[-k,k]$, and $p_n(x)\rightarrow \mathbf{1}_{[a,b]}(x)$ for every $x\in\mathbb{R}$. Since $(p_n\circ m)\cdot f_k'$ is dominated by $2|f_k'|$, it converges to $(\mathbf{1}_{[a,b]}\circ m)\cdot f_k'= \mathbf{1}_{m^{-1}([a,b])}\cdot f_k'$ in $L_2(\Omega)$ by the dominated convergence theorem.
	
	Then, since $\operatorname{proj}_{U(H)}$ is a bounded operator, we have $(p_n\circ m)\cdot f_k$ converges to $\operatorname{proj}_{U(H)}\mathbf{1}_{m^{-1}([a,b])}\cdot f_k'$ in $L_2(\Omega)$. But we also have that $(p_n\circ m)\cdot f_k$ converges punctually to $\mathbf{1}_{m^{-1}([a,b])}\cdot f_k$. Since any convergent sequence in $L_2(\Omega)$ has a subsequence that converges punctually almost everywhere to the $L_2$ limit, we conclude that as $L_2$ functions, $\mathbf{1}_{m^{-1}([a,b])}\cdot f_k=\operatorname{proj}_{U(H)}\mathbf{1}_{m^{-1}([a,b])}\cdot f_k'\in U(H)$.
	
	But then, since $f=\lim_{k\rightarrow\infty}f_k'$ in $L_2(\Omega)$ by the monotonic convergence theorem, we have 
	\begin{align*}
		f=\operatorname{proj}_{U(H)}f&=\lim_{k\rightarrow\infty}f_k\\
		\mathbf{1}_{m^{-1}([a,b])}\cdot f&=\lim_{k\rightarrow\infty}\mathbf{1}_{m^{-1}([a,b])}f_k.
	\end{align*}
	
	And so we conclude $\mathbf{1}_{m^{-1}([a,b])}\cdot f\in U(H)$.
\end{proof}

We then show, proving the theorem of this section:
\begin{theorem}[Induced spectral resolution]
	We have that the function $P:=U^*\hat{P}(\cdot)U$ is a well-defined projection-valued measure on  $H$, forming a spectral resolution of $A$.
\end{theorem}
\begin{proof}
	We must first show that for $V\in\operatorname{Borel}(\mathbb{R})$, $P(V)$ is a self-adjoint projection on $H$. $P(V)$ is indeed a well-defined operator on $H$, and since all operators involved are bounded, $P(V)^*=U^*\hat{P}(V)^*U^{**}=U^*\hat{P}(V) U=P(V)$ is self-adjoint. Furthermore, since $(\hat{P}(V)U)(H)\subset U(H)$ by the previous proposition:
	\begin{align*}
		P(V)^2&=U^* \hat{P}(V) U U^*\hat{P}(V) U=U^*\hat{P}(V) \operatorname{proj}_{U(H)}\hat{P}(V)U\\
		&=U^*\hat{P}(V)\hat{P}(V) U=U^*\hat{P}(V) U=P(V).
	\end{align*} 
	And so, $P(V)$ is a self-adjoint projection on $H$. Furthermore:
	\begin{itemize}
		\item $P(\emptyset)=U^* \circ 0 \circ U=0$, and $P(\mathbb{R})=U^* U=\operatorname{id}$.
		\item For any $V_1, V_2$ in $\operatorname{Borel}(\mathbb{R})$, 
		\begin{align*}
			P(V_1\cap V_2)=U^* \hat{P}(V_1)\hat{P}(V_2) U=U^* \hat{P}(V_1)\operatorname{proj}_{U(H)}\hat{P}(V_2) U=P(V_1)P(V_2).
		\end{align*}
		\item Since $L\rightarrow U^* L U\in B(B(L_2(\Omega)),B(H))$, we have that whenever $(V_n)_{n\in\mathbb{N}}$ is a pairwise disjoint sequence in $\operatorname{Borel}(\mathbb{R})$, $P(\bigsqcup_{n\in\mathbb{N}}V_n)=\sum_{n\in\mathbb{N}}P(V_n)$.
	\end{itemize}
	
	And so, $P$ is a projection valued measure on $H$. Furthermore, for any $x,y\in H$ and $V\in\operatorname{Borel}(\mathbb{R})$, $\mu_{x,y}(V)=(P(V)x,y)=(\hat{P}(V) Ux, Uy)=\hat{\mu}_{Ux,Uy}(V)$. Therefore, $\mu_{x,y}=\hat{\mu}_{Ux, Uy}$ holds for any $x,y\in H$. 
	
	For any $x$ in $H$,  $x\in\operatorname{dom}(\int_{\mathbb{R}}\operatorname{id}dP)$ if and only if $Ux\in\operatorname{dom}(\int_{\mathbb{R}}\operatorname{id}d\hat{P})=\operatorname{dom}(T)$, which holds if and only if $x\in\operatorname{dom}(T\circ U)=\operatorname{dom}(A)$. Therefore, $\operatorname{dom}(\int_{\mathbb{R}}\operatorname{id}dP)=\operatorname{dom}(A)$. Furthermore, for any $x\in\operatorname{dom}(A)$ and $y\in H$:
	\begin{align*}
		(\int_{\mathbb{R}}\operatorname{id}dPx,y)&=\int_{\mathbb{R}}\operatorname{id}d\mu_{x,y}=\int_{\mathbb{R}}\operatorname{id}d\hat{\mu}_{Ux,Uy}=(\int_{\mathbb{R}}\operatorname{id}d\hat{P}Ux,Uy)\\
		&=(TUx,Uy)=(UAx,Uy)=(Ax,y).
	\end{align*}
	Therefore, we have $\int_{\mathbb{R}}\operatorname{id}dP=A$, concluding the proof.  
\end{proof}
\section*{Acknowledgments}
I extend my sincere gratitude to Isaac Goldbring for providing a lot of helpful feedback, and for encouraging a proof of Theorem~\ref{spectral_theorem} (a conjecture at the time), which resulted in Section \ref{section_sa}. 

I am thankful to both of my research supervisors, Alexander Shnirelman and Alexey Kokotov, for providing helpful direction on my research leading to this paper. I thank Concordia University as well for providing the funding I needed for my research.

Finally, I am grateful to Alexis Langlois-Rémillard and Antoine Giard for providing feedback on the writing style. 
\printbibliography
\end{document}